\documentclass{article}

\usepackage{microtype}
\usepackage{graphicx}
\usepackage{subfigure}
\usepackage{diagbox}
\usepackage{url}
\usepackage{booktabs} 
\usepackage{natbib}
\usepackage{hyperref}

\usepackage[a4paper,margin=3cm]{geometry}

\usepackage{amsmath}
\usepackage{amssymb}
\usepackage{mathtools}
\usepackage{amsthm}
\usepackage{circuitikz}

\def\R{{\mathbb{R}}}

\def\N{{\mathbb{N}}}

\def\sphere{\mathbb{S}}
\def\1{\mathbb{1}}

\def\xcont{(x_t)_{t\ge 0}}
\def\xz{(x_t,z_t)_{t\ge 0}}
\def\xdisc{\{\tilde x_k \}_{k\ge 0}}

\def\mart{(M_t)_{t\ge 0}}
\def\xdisctilde{\{\tilde{x}_k \}_{k\in \N}}
\def\pl{{\text{PL}^\mu}}
\def\rsi{{\text{RSI}^\nu}}

\newcommand{\bpar}[1]{\left(#1\right)}

\newcommand{\E}[1]{\mathbb{E}\left[#1 \right]}
\newcommand{\condE}[1]{\mathbb{E}_{k}\left[#1 \right]}
\newcommand{\norm}[1]{\left\lVert#1\right\rVert}
\newcommand{\dotprod}[1]{\left< #1\right>}
\newcommand{\grad}[1]{\nabla f(#1)}

\newcommand{\sqc}{\text{SQC}_{\tau}^{\mu}}

\newcommand{\ls}{\text{LS}^{L}}

\newcommand{\strongconv}{\text{SC}^{\mu}}

\newcommand{\ac}{\text{AC}^{a}}
\newcommand{\acf}[1]{\text{AC}^{#1}}
\newcommand{\qg}{\text{QG}_{+}^{L_0}}
\newcommand{\qginf}{\text{QG}_{-}^{\mu_0}}
\newcommand{\qgf}[1]{\text{QG}_{+}^{#1}}
\newcommand{\qginff}[1]{\text{QG}_{-}^{#1}}
\newcommand{\sqcf}[2]{\text{SQC}_{#1}^{#2}}
\newcommand{\plf}[1]{\text{PL}^{#1}}
\newcommand{\rsif}[1]{\text{RSI}^{#1}}

\newcommand{\off}[1]{}

\newcommand{\bigO}{\mathcal{O}}

\def\argmin{\textup{argmin}\,}
\def\arginf{\textup{arginf}\,}

\usepackage[capitalize,noabbrev]{cleveref}

\theoremstyle{plain}
\newtheorem{theorem}{Theorem}[section]
\newtheorem{proposition}[theorem]{Proposition}
\newtheorem{lemma}[theorem]{Lemma}

\theoremstyle{definition}
\newtheorem{definition}[theorem]{Definition}
\newtheorem{assumption}[theorem]{Assumption}
\theoremstyle{remark}

\newtheorem*{remark*}{Remark}

\usepackage[textsize=tiny]{todonotes}
\usepackage{enumitem}

\author{Julien Hermant\\
Univ. Bordeaux, IMB}
\date{}

\title{Acceleration for Polyak–\L{}ojasiewicz Functions with a Gradient Aiming Condition}
\author{Julien Hermant\\
Univ. Bordeaux, IMB}
\date{}

\begin{document}





\maketitle

\vskip 0.3in




\begin{abstract}
It is known that when minimizing smooth Polyak--\L{}ojasiewicz (PL) functions, Nesterov's momentum algorithm cannot significantly improve the convergence bound of gradient descent, contrasting with the acceleration phenomenon occurring in the strongly convex case. To bridge this gap, the literature has proposed \textit{strongly quasar-convex} functions as an intermediate non-convex class, for which the accelerated bounds of the strongly convex case have been suggested to persist. We show that strong quasar-convexity actually does not ensure a similar acceleration phenomenon. As an alternative, we study PL functions under an \textit{aiming condition} that measures how well the descent direction points toward a minimizer. This viewpoint sheds new light on the possibility of theoretical acceleration with Nesterov's momentum when minimizing PL functions.
\end{abstract}

\section{Introduction}

Momentum \cite{polyak1964some,nesterov1983method} is an important component of several modern optimizers. In machine learning, it is observed to substantially improve the optimization of large-scale models in practice \cite{sutskever2013importance,he2016deep}.
Although the theoretical benefit of incorporating momentum in vanilla gradient descent is well established in the convex and strongly convex setting \cite{nemirovskij1983problem,nesterovbook}, non-convexity is, however, an inherent property of many problems \cite{dauphin2014identifying,ge2016matrix,ge2017no,bhojanapalli2016global,li2018visualizing}. Importantly, for some non-convex settings it is not possible to demonstrate acceleration using momentum \cite{lowerboundI,PLlowerbound}. 
Focusing on neural networks, some works establish acceleration properties \cite{wang2021modular, liu2022provable,liu2022convergence, liao2024provable}, but they rely on simplifying assumptions---such as a very high degree of over-parameterization---that make the learning problem similar to a strongly convex quadratic minimization problem.

Another line of work lies in considering classes of non-convex functions that relax strong convexity.
In particular, the class of Polyak--\L{}ojasiewicz (PL) functions \cite{Polyak1963,lojasiewicz1963topological} has emerged as an important framework and has stimulated extensive interest \cite{karimi2016linear,fazel2018global,oymak2019overparameterized,altschuler2021averaging,merigot2021non,apidopoulos2022convergence,apidopoulos2025heavy,gess2023exponential}. A continuously differentiable function $f:\R^d \to \R $ belongs to this class if there exists $\mu>0$ such that the following inequality holds for all $x\in\R^d$
\[ \norm{\nabla f(x)}^2 \ge 2\mu\Big(f(x)-\min_{u\in \R^d}f(u)\Big).\]
Among the remarkable properties of PL functions, we note that (i) it is a necessary and sufficient condition to achieve linear convergence using gradient descent \cite{abbaszadehpeivasti2023conditions}, and (ii) it can effectively model over-parameterized neural networks \cite{ liu2022loss,barboni2022global,chen2023overparameterized, xulocal}. Yet, PL functions are such that gradient descent's convergence bound cannot be significantly improved using momentum \cite{PLlowerbound}. This seems to contradict the often empirically observed benefit of momentum, and thus motivates the search for more structured relaxations of strong convexity.
An intermediate non-convex class that is expected to capture acceleration by momentum is \textit{strong quasar-convexity} \cite{hinder2020near}, sometimes referred to as \textit{weak-quasi-strong convexity }\cite{necoara2019linear,bu2020note}. A continuously differentiable function $f:\R^d \to \R $ belongs to this class if there exists a minimizer $x^\ast\in \R^d$ and parameters $(\tau,\mu) \in (0,1]\times \R_+^{\ast}$ such that for all $x\in \R^d$
\[\min_{u\in \R^d}f(u) \ge f(x) + \frac{1}{\tau}\dotprod{\nabla f(x),x^\ast-x} + \frac{\mu}{2}\norm{x-x^\ast}^2.\]
There has been growing interest in strongly quasar-convex functions in recent years \cite{gower2021sgd,wang2023continuized,alimisis2024characterization,alimisis2024geodesic,pun2024online,hermant2025continuized,chenefficient2025,lara2025delayed,de2025extending,farzin2025minimisation}.
In the $L$-smooth setting, minimizing these functions using momentum is often referred to as \textit{accelerated optimization} \cite{hinder2020near,fu2023accelerated,hermant2024study}, suggesting better convergence bounds than gradient descent. This expectation partly stems from the fact that its complexity bounds take a form similar to those obtained in the smooth, strongly convex case. In this work, we show that this intuition is misleading.


\subsection{Contributions}
Our work addresses what additional structure beyond the PL condition is required to obtain improved theoretical convergence bounds over gradient descent when using momentum. We summarize our contributions as follows:

\noindent \textbf{(1.)} 
We exhibit in Section~\ref{sec:sqc_pitfalls} flaws of strong quasar-convexity, and show that in contrast with existing interpretations, minimizing these functions with momentum does not necessarily yield better theoretical convergence bounds than gradient descent. We demonstrate that an acceleration phenomenon similar to strong convexity does not occur in general for this class.

\noindent \textbf{(2.)}  We propose in Section~\ref{sec:aiming_cond} a different viewpoint by considering an \textit{aiming condition}
\begin{equation}\label{eq:intro_ac}
    \dotprod{\nabla f(x),x-x^\ast} \ge a \norm{\nabla f(x)}\norm{x-x^\ast},
\end{equation}
for $1 \ge a>0$, which quantifies how well the descent direction aims toward a minimizer $x^\ast$. We show how \eqref{eq:intro_ac} relates to strong quasar-convexity, and show that it helps to define a specific reparameterization of this latter class. In Section~\ref{sec:cv_results}, we study the convergence of gradient descent, gradient flow, the Nesterov momentum algorithm and the Nesterov ODE for functions that are both PL and satisfy the aiming condition. Using a stochastic parameterization of the Nesterov momentum algorithm, we show that we obtain accelerated bounds when the alignment constant $a$ is sufficiently large (\cref{sec:mom_cv}). 


\noindent \textbf{(3.)} 
We relax the aiming condition by requiring it to hold only on average along the optimization trajectory, recovering similar convergence rates (Section~\ref{sec:avg}). 
 It shows that acceleration can be achieved as long as \eqref{eq:intro_ac} holds over a sufficiently large fraction of the optimization path,
 better reflecting the behavior of practical optimization algorithms.
%

\noindent \textbf{(4.)} We provide numerical experiments to support our theoretical insights. We exhibit a two-dimensional PL function with a unique minimizer that fails to
satisfy~\eqref{eq:intro_ac}, resulting in gradient descent outperforming momentum methods during the early stages of optimization (Section~\ref{sec:numerical_1}). We then provide a heuristic validation of our theoretical convergence bounds on the hard function for PL functions from \cite{PLlowerbound}, see Section~\ref{app:heuristic}.

\subsection{Related Work on Aiming Conditions}\label{sec:related_aiming}
Studying SGD,
\cite{liu2023aiming,guptanesterov,alimisis2025we} use conditions close to quasar-convexity, which they call \textit{aiming condition}. While they consider the interpretation of pointing toward a minimizer, this condition actually carries more structure. Studying gradient descent, \cite{guille2024no,singh2025directionality} consider \eqref{eq:intro_ac} from distinct perspectives: the former studies it empirically along the optimization path, whereas the latter uses it to prove a stability result, in the idea of better understanding the Edge of Stability phenomenon \cite{cohengradient}.  \cite{hermant2024study} establishes that a PL function with a unique minimizer is also strongly quasar-convex for some parameters, as long as \eqref{eq:intro_ac} is verified, but the derived parameters are such that momentum does not yield accelerated bounds, see Appendix~\ref{app:related:uaac}.
\section{Background}
In this work, $f:\R^d \to \R$ is continuously differentiable. We note $f^\ast := \min_{x \in \R^d} f(x)$, $x^\ast \in {\arg\min}_{x\in \R^d} f(x)$. 
$k \in \N$ denotes a discrete index and $t \in \R_+$ a continuous one. $\mathcal{E}(1)$ denotes the exponential law of parameter $1$. 

\subsection{Algorithms and Equations}
We introduce the algorithms and dynamical systems that we study throughout this work.

\paragraph{Gradient Descent and Gradient Flow \:}
Among the most classical first-order algorithms is gradient descent \eqref{eq:gd}
\begin{equation}\label{eq:gd}\tag{GD}
    \tilde x_{k+1} =\tilde x_k - \gamma \nabla f(\tilde x_k).
\end{equation}
It is the building block of many important algorithms used in modern applications \cite{hinton2012rmsprop,ghadimi2013stochastic,kingma2014adam,romano2017little}, and is still an important object of study \cite{malitsky2020adaptive,cohengradient}. 
 Writing gradient descent as $ \frac{x_{k+1} - x_k}{\gamma } = -\nabla f(x_k)$, setting the \textit{Ansatz} $\tilde x_k \approx x_{k\gamma }$ for some smooth curve $\xcont$, we consider the ODE limit of \eqref{eq:gd} as the stepsize $\gamma $ vanishes to zero, named gradient flow \eqref{eq:gf} 

\begin{equation}\label{eq:gf}\tag{GF}
    \dot{x}_t = - \nabla f(x_t). 
\end{equation}
In words, \eqref{eq:gf} is a continuous version of \eqref{eq:gd}. This dynamical system viewpoint is longstanding for gaining insights into the corresponding algorithms \cite{ambrosio2005gradient, benaim2006dynamics}. 

\paragraph{Nesterov's Momentum: ODE and Algorithm \:}
If \eqref{eq:gd} benefits from a rather simple formulation, there are many situations where modifications of this algorithm accelerate its convergence speed. Among these mechanisms, an important one is \textit{momentum}. The first momentum algorithm was Polyak's Heavy Ball \cite{polyak1964some}, which writes
\begin{equation}\label{eq:hb}\tag{HB}
     \tilde x_{k+1} = \tilde x_k + \underbrace{\alpha(\tilde x_k-\tilde x_{k-1})}_{\text{momentum}} - \underbrace{\eta \nabla f(\tilde x_k)}_{\text{gradient step}}.
\end{equation}
In this work, we focus on the Nesterov momentum \eqref{eq:nm} version \cite{nesterovbook}. 
\begin{align}\label{eq:nm}\tag{NM}\left\{
    \begin{array}{ll}
        \tilde{y}_k &= \tilde{x}_k+ \alpha_k(\tilde{z}_{k} - \tilde{x}_k),  \\
        \tilde{x}_{k+1} &=  \tilde{y}_k  - \gamma \nabla f(\tilde{y}_k),\\
        \tilde{z}_{k+1} &=   \tilde{z}_{k} + \beta_k(\tilde{y}_k - \tilde{z}_{k}) - \gamma' \nabla f(\tilde{y}_k).
    \end{array}
\right.
\end{align}
Although its formulation is slightly more involved than Polyak’s Heavy Ball,
the Heavy Ball method exhibits worse convergence guarantees in some settings \cite{ghadimi2015global,goujaud2025provable}.
As for gradient descent, taking the limit of \eqref{eq:nm} as the stepsize $\gamma$ vanishes yields a limit Nesterov momentum ODE \eqref{eq:nmo}, \cite[Appendix A.1]{pmlr-v202-kim23y}
\begin{align}\label{eq:nmo}\tag{NMO}\left\{
    \begin{array}{ll}
        \dot{x}_t &= \eta_t(z_t-x_t) - \gamma \nabla f(x_t), \\
        \dot{z}_t &= \eta'_t(x_t-z_t) - \gamma' \nabla f(x_t). 
    \end{array}
\right.
\end{align}
Since the strong connection between these ODEs and the corresponding algorithm has been established \cite{suboydcandes}, their study has attracted growing interest both in the convex \hypertarget{cont_param}{{}} \cite{siegel2021accelerated,shi2018understanding,attouch2020firstorder, aujdossrondPL,li2024linear,guo2024speed} and non-convex case \cite{okamura2024primitive,hermant2024study,guptanesterov,renaud2025provably}. 

\noindent
\textbf{Continuized Parameterization \:}
We use a specific parameterization of \eqref{eq:nm} with stochastic parameters that arises from the \textit{continuized} method \cite{even2021continuized}. Precisely, for some $\eta,\eta' \in \R$ such that $\eta+\eta'>0$, and some $\{T_k \}_{k \in \N}$ such that $T_{k+1}-T_k \overset{i.i.d}{\sim} \mathcal{E}(1)$, we fix in \eqref{eq:nm}:
\begin{align*}
    &\alpha_k = \frac{\eta}{\eta+\eta'}\bpar{1 -e^{-(\eta + \eta')(T_{k+1}-T_k)} },\\
    &\beta_k = \eta' \frac{\bpar{1- e^{-(\eta+\eta')(T_{k+1}-T_k)}}}{\eta' + \eta e^{-(\eta+\eta')(T_{k+1}-T_k)}}.
\end{align*}
With this parameterization, the algorithm is inherently stochastic, yielding convergence guarantees holding with high probability. 
In the case of minimizing smooth strongly quasar-convex functions, the best known bounds are achieved using this parameterization \cite{wang2023continuized}. In contrast, a more classical deterministic parameterization achieves such bounds only under additional assumptions \cite{hermant2024study}, or when using search procedures to compute parameters \cite{hinder2020near,nesterov2021primal,guminov2023accelerated}, see details in Appendix~\ref{app:related_work_tab}.
The analysis of \eqref{eq:nm} using the continuized parameterization involves a Poisson-driven SDE; we refer to \cite{hermant2025continuized} for a methodological introduction to this continuized viewpoint.
\subsection{Classical setting for acceleration}\label{sec:strongly_convex}
\begin{table*}[]
\caption{Comparison of convergence rates up to a numerical constant, for \eqref{eq:gf}, \eqref{eq:nmo}, \eqref{eq:gd} and \eqref{eq:nm} properly parameterized, for functions belonging to $\strongconv$, $\sqc$ and $\pl$, intersected with $\ls$ for algorithms. Most can be found in the literature, see Appendix~\ref{app:related_work_tab}.}
\label{table:sc_sqc_rates}
\vspace{0.3em}
\centering
\begin{tabular}{|l|l|l|l|l|}
\hline
\diagbox{Function class}{Equation/Algorithm}     & GF           & NMO                   & GD                     & NM                     \\ \hline
$\mu$-strongly convex ($\strongconv$)       & $ \mu$       & $ \sqrt{\mu}$         & $\mu/L$        & $\sqrt{\mu/L}$        \\ \hline
$(\tau,\mu)$-strongly quasar-convex ($\sqc$)& $\tau \mu$ & $ \tau \sqrt{\mu} $ & $\tau \mu/L$ & $\tau  \sqrt{\mu/L}$ \\ \hline
$\mu$-Polyak-\L{}ojasiewicz ($\pl$)       & $ \mu$       & $(-)$         & $\mu/L$        & $\mu/L$        \\ \hline
\end{tabular}
\end{table*}

We first introduce a geometrical setting in which momentum-based acceleration is long established.
\begin{definition}
    For $L>0$, we call $\ls$ the set of $L$-smooth functions, namely $f$ such that $\forall x,y\in \R^d,
 ~ f(x)-f(y)-\dotprod{\nabla f(y),x-y}  \le \frac{L}{2}\norm{x-y}^2$.
\end{definition}
\begin{definition}
    For $\mu>0$, we call $\strongconv$ the set of $\mu$-strongly convex functions, namely $f$ such that $\forall x,y \in \R^d,~ f(x)+\dotprod{\nabla f(x),y-x} + \frac{\mu}{2}\norm{x-y}^2 \le f(y).$
\end{definition}
Intuitively, $L$-smoothness and $\mu$-strong convexity provide global upper and lower bounds on the curvature of $f$.
For a $C^2$ function, $f \in\strongconv \cap \ls $ implies that any eigenvalue $\lambda(x)$ of $\nabla^2 f(x)$ is such that $\mu \le \lambda(x) \le L$, for any $x\in\R^d$. The ratio $L/\mu \ge 1 $ is called \textit{condition number}, and characterizes how ill-conditioned the minimization problem is—a higher condition number typically reflecting a more difficult problem. $\strongconv \cap \ls$ is a favorable setting for minimization: there exists a unique critical point $x^\ast$, satisfying $f(x^\ast) = f^\ast$,\hypertarget{cr}{{}} and linear convergence of function values is ensured by first order methods, namely
\[\log (f(x_t)-f^\ast) \le -\Lambda t + \log(K(x_0)), \]
where $\Lambda > 0$ is the \textbf{convergence rate} and $K(x_0) > 0$ depends on the initialization. Linear convergence also holds in the discrete case, \textit{i.e.}, replacing $\xcont$ by $\xdisc$ and $t$ by $k$.
In this work, we specifically focus on the convergence rate value $\Lambda$.
In the $\strongconv \cap \ls$ regime, the improvement of \eqref{eq:nm} over \eqref{eq:gd} is well established \cite{nesterov2004introductory}, see 
the convergence rates in Table~\ref{table:sc_sqc_rates}. Precisely, as typically $L/\mu \gg 1$, improving from $\mu/L$ to $\sqrt{\mu/L}$ yields a substantial gain, often called \textit{acceleration}. However, many problems of interest do not fall within $\strongconv \cap \ls$. This motivates to identify more general settings where momentum-based acceleration remains possible.
 \subsection{Relaxed setting}\label{sec:relax_setting}
A first relaxation of $\strongconv$ and $\ls$ consists in lower and upper
global quadratic bounds of the function.
        \begin{definition}\label{ass:upp_qg}
                For $\mu_0,L_0 > 0$, we call: $\qginf$ the set of functions $f$ that verify
  $  f(x) - f^\ast \ge \frac{\mu_0}{2}\lVert x - x^\ast \rVert^2$, $\qg$ the set of functions $f$ that verify 
  $  f(x) - f^\ast \le \frac{L_0}{2}\lVert x - x^\ast \rVert^2$.
         \end{definition}
         $\qginf \cap \qg$ includes functions with pathological behavior, \textit{e.g.} having an infinite amount of critical points in any vicinity of $x^\ast$ \cite[Proposition 9]{hermant2024study}. This may discard the possibility of targeting a global minimum $f^\ast$.
          A key relaxation of strong convexity that does not preclude this objective is the class of Polyak--\L{}ojasiewicz (PL) functions \cite{Polyak1963,lojasiewicz1963topological}.   
\begin{definition}
    For $\mu>0$, we call $\pl$ the set of functions $f$ that satisfy $\forall x \in \R^d,~ \norm{\nabla f(x)}^2 \ge 2\mu(f(x)-f^\ast).$
\end{definition}
It immediately follows that for $f \in \pl$, any critical point is a global minimizer of $f$. From an optimization perspective, two key results hold for the class $\pl \cap \ls$: (i) gradient descent achieves a linear convergence rate \cite{Polyak1963}, and (ii) this rate is optimal among first-order algorithms \cite{PLlowerbound}. As a consequence, \eqref{eq:nm} cannot yield any significant theoretical improvement over this class, see Table~\ref{table:sc_sqc_rates}. This motivates studying more structured non-convex relaxations of $\strongconv \cap \ls$. The class of strongly quasar-convex functions (SQC) is expected to fulfill the demand.
\begin{definition}
    For $(\tau,\mu)\in (0,1]\times \mathbb{R}^\ast_+$, we call $\sqcf{\tau}{\mu}$ the set of functions $f$ that verify $\forall x\in \mathbb{R}^d,~ f^\ast \ge f(x) + \frac{1}{\tau}\dotprod{\nabla f(x),x^\ast-x} + \frac{\mu}{2}\norm{x-x^\ast}^2$.
\end{definition}
$\sqc$ is a relaxation of $\strongconv$, while being more structured than $\pl$, see \eqref{eq:inclusions}. One possible application of these functions is the training of generalized linear models \cite[Section 3]{wang2023continuized}. Compared to $\pl$, an important supplementary structure of $\sqc$ is to keep the property of $\strongconv$ of having a unique minimizer, although it can be relaxed \cite{guptanesterov,hermant2025continuized}. These functions can exhibit strong oscillations, see Figure~\ref{fig:figure_sqc}. Unlike PL functions, this class admits distinct convergence rates for \eqref{eq:gd} and \eqref{eq:nm} when intersected with $\ls$, see \cref{table:sc_sqc_rates}.

\noindent
\textbf{Relation between classes \:}
The following inclusions hold \cite{karimi2016linear,hermant2024study}
\begin{equation}\label{eq:inclusions}
    \strongconv \subset \sqcf{1}{\mu} \subset{\sqc} \subset \plf{\mu\tau^2} \subset \qginff{\mu \tau^2},\quad \ls \subset \qgf{L}.
\end{equation}
    As observed by \cite{guille2021study}, the constants inherited from these inclusions may be severely suboptimal. For instance, oscillations of the functions outside a vicinity of minimizers may affect significantly the $\ls$ constant but not the $\qg$ constant, such that we could have that $f \in \ls \cap \qg$ with $L_0\ll L$. Similarly, a sudden drop of the gradient norm may affect the $\pl$ constant while letting the $\qginf$ almost untouched, such that $f \in \pl \cap \qginf$ while $\mu_0 \gg \mu$. This induces that comparing convergence rates across different classes  must be handled with care when using the same constants for different classes, $\mu$ or $L$ classically.

\section{Acceleration under Strong Quasar-Convexity: Hidden Pitfalls}\label{sec:sqc_pitfalls}
Using momentum to optimize functions in $\sqc \cap \ls$ has been described as accelerated optimization \cite{hinder2020near,fu2023accelerated,hermant2024study}. This would place this class as an interesting intermediate class between $\strongconv \cap \ls$ and $\pl \cap \ls$: retaining the acceleration property of the first, while allowing non-convex behavior as the second does. 
The interpretation of retaining the acceleration property of  $\strongconv \cap \ls$ is
motivated by an analogy between complexity bounds. We show that performing such analogy is misleading. 

\paragraph{A hard PL instance is strongly quasar-convex}
\cite{PLlowerbound} introduces a \textit{hard function} belonging to $\pl \cap \ls$ such that the number of gradient calls needed to obtain a point $\tilde x$ ensuring $f(\tilde x)-f^\ast \le \varepsilon$ scales at best as $\bigO(\frac{L}{\mu}\log(\frac{1}{\varepsilon}))$ for any first-order algorithms. This bound corresponds to a $\mu/L$ convergence rate, inducing that, up to numerical constants, the convergence rate of \eqref{eq:gd} is optimal, see Table~\ref{table:sc_sqc_rates}. Perhaps surprisingly, this hard function satisfies some strong regularity, for instance it has a unique global minimizer. More interesting in our case, it is actually strongly quasar-convex.
\begin{proposition}\label{prop:hard_fonc_sqc}
   The hard function belonging to $\pl \cap \ls$ from \cite{PLlowerbound}, for which the complexity bound of \eqref{eq:gd} can be improved at most by a numerical factor, also belongs to $\sqcf{\tau}{\mu'} \cap \ls$ for some parameters $(\tau,\mu')\in (0,1]\times \R_+^\ast$.
\end{proposition}
The proof is essentially a collection of known result, see Appendix~\ref{app:hard_fonc_sqc}. This proposition shows that an acceleration of the same nature as for $\strongconv \cap \ls$ is not possible in general when considering $\sqc \cap \ls$. Indeed, in the case of $\strongconv \cap \ls$, the ratio of the convergence rates between \eqref{eq:nm} and \eqref{eq:gd} is $\sqrt{\mu/L}/(\mu/L) =  \sqrt{L/\mu}$, which goes to infinity with the conditioning $L/\mu$ going to infinity. In particular, if we fix $L$, this shows that for this class, when taking $\mu \to 0$ the bound of \eqref{eq:nm} is arbitrarily better compared with \eqref{eq:gd}. Such behavior is not possible for the function from \cite{PLlowerbound}.

\paragraph{Weaker Assumption, Sharper Rate}
To compare the convergence rates of \eqref{eq:gd} and \eqref{eq:nm} when minimizing functions in
$\sqc \cap \ls$, it is natural to rely on the rates established in the literature for this
class, namely $\tau\mu/L$ for \eqref{eq:gd} and $\tau\sqrt{\mu/L}$ for \eqref{eq:nm},
see Table~\ref{table:sc_sqc_rates}.
However, any function belonging to $\sqc \cap \ls$ also belongs to $\plf{\mu'} \cap \ls$, for some parameter $\mu'$. 
As a result, one
may alternatively consider the convergence rate $\mu'/L$ of gradient descent derived for functions in
$\plf{\mu'} \cap \ls$.
Strikingly, in some cases, this latter rate
\textbf{can be sharper} than those given for the class $\sqc \cap \ls$.
In Figure~\ref{fig:pl_sqc_low_corr}, we plot two examples of functions that belong to $\sqc \cap \ls$, and so also belong to $\plf{\mu'}  \cap \ls$, for some parameters. For a given pair of algorithm and function class, we numerically compute the parameters of this class such that the associated convergence rate is maximized. 
For the first example, we find that the rate of \eqref{eq:nm} characterized by $\sqc$ is $\approx 10^{2}$ times larger than the one of \eqref{eq:gd} characterized by $\plf{\mu'}$.
However, for the second example, we find that the rate of \eqref{eq:gd} characterized by $\plf{\mu'}$ is $\approx 10^2$ times better than the rate of \eqref{eq:nm} characterized by $\sqc$. Thus, the worst-case bounds of \eqref{eq:gd} and \eqref{eq:nm} under $\sqc \cap \ls$ may be improved if considering the worst-case bound of \eqref{eq:gd} under $\pl \cap \ls$.

These findings indicate that our understanding of acceleration with strong quasar convexity is lacking. In fact, this class suffers from some problems, as we discuss in the next section.
\subsection{Problems of Strong Quasar-Convexity}\label{sec:specific_pitfall}
The intuition that a similar acceleration phenomenon occurs for $\sqc \cap \ls$ and $\strongconv \cap \ls$ can results from an apparent improvement from $\mu/L$ with \eqref{eq:gd} to $\sqrt{\mu/L}$ with \eqref{eq:nm} in both cases, see Table~\ref{table:sc_sqc_rates}, and from carrying the interpretation of $L/\mu$ as a condition number. This is somewhat deceptive, as we detail below. 

\paragraph{Parameter Interpretation}
The constant $\mu$ in $\strongconv$ parametrizes a global lower bound on the curvature. Reusing the notation $\mu$ in a similar way in $\sqc$ may therefore suggest a similar interpretation. In particular, for $f \in \sqc \cap \ls$, one may be tempted to view $L/\mu$ as a conditioning measurement, analogously to the setting $\strongconv \cap \ls$. Under this interpretation, one would expect $L/\mu \ge 1$, see Section~\ref{sec:strongly_convex}. 
However, since $f \in \sqcf{\tau}{\mu} \Rightarrow f \in \sqcf{\theta \tau}{\mu/\theta}$, for any $\theta \in (0,1]$ \cite[Observation 5]{hinder2020near}, it is possible to have $f \in \sqcf{\tau}{\mu} \cap \ls$ for some parameters satisfying $\mu > L$. This observation challenges both the condition-number interpretation of $L/\mu$ and the interpretation of the convergence rates in this setting.


\paragraph{The Optimal Choice of Parameters}
For many function classes, there exists a natural notion of optimal parameters.
For instance, if $f \in \pl$
for a range of constants $\mu>0$, the tighter class to describe $f$ is $\plf{\hat{\mu}}$ with $\hat \mu$
the largest admissible value.
Accordingly, the rates reported in Table~\ref{table:sc_sqc_rates} are maximized by choosing $\mu = \hat \mu$.
In contrast, because of the double parameterization, there is no
canonical way to choose among all admissible pairs $(\tau,\mu)$ an optimal one for a function in $\sqc$. Choosing the largest admissible $\tau$ and then the largest admissible $\mu$ leads to a different parameter pair than optimizing in the reverse order, see Appendix~\ref{app:choice_param_sqc}.
More critically, for a given function $f$, among the admissible pairs $(\tau,\mu)$ ensuring $f \in \sqc$, the one 
that maximizes a given convergence rate depends on the algorithm or
dynamic under consideration, as illustrated in \cref{fig:sqc_flaws}.
As a consequence, statements of the form
\emph{“the rates of \eqref{eq:gf} and \eqref{eq:nmo} scale as
$\tau\mu$ and $\tau\sqrt{\mu}$”} without clarifying how the parameters are chosen may result in misleading comparisons.
\begin{figure}
    \centering
    \includegraphics[width=0.3\linewidth]{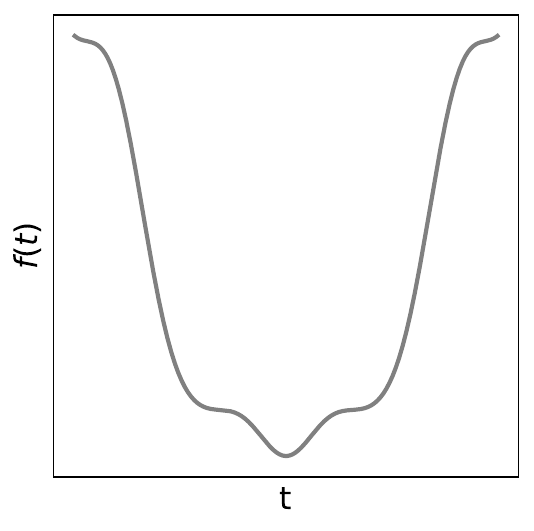}
    \hspace{-7pt}
    \includegraphics[width=0.3\linewidth]{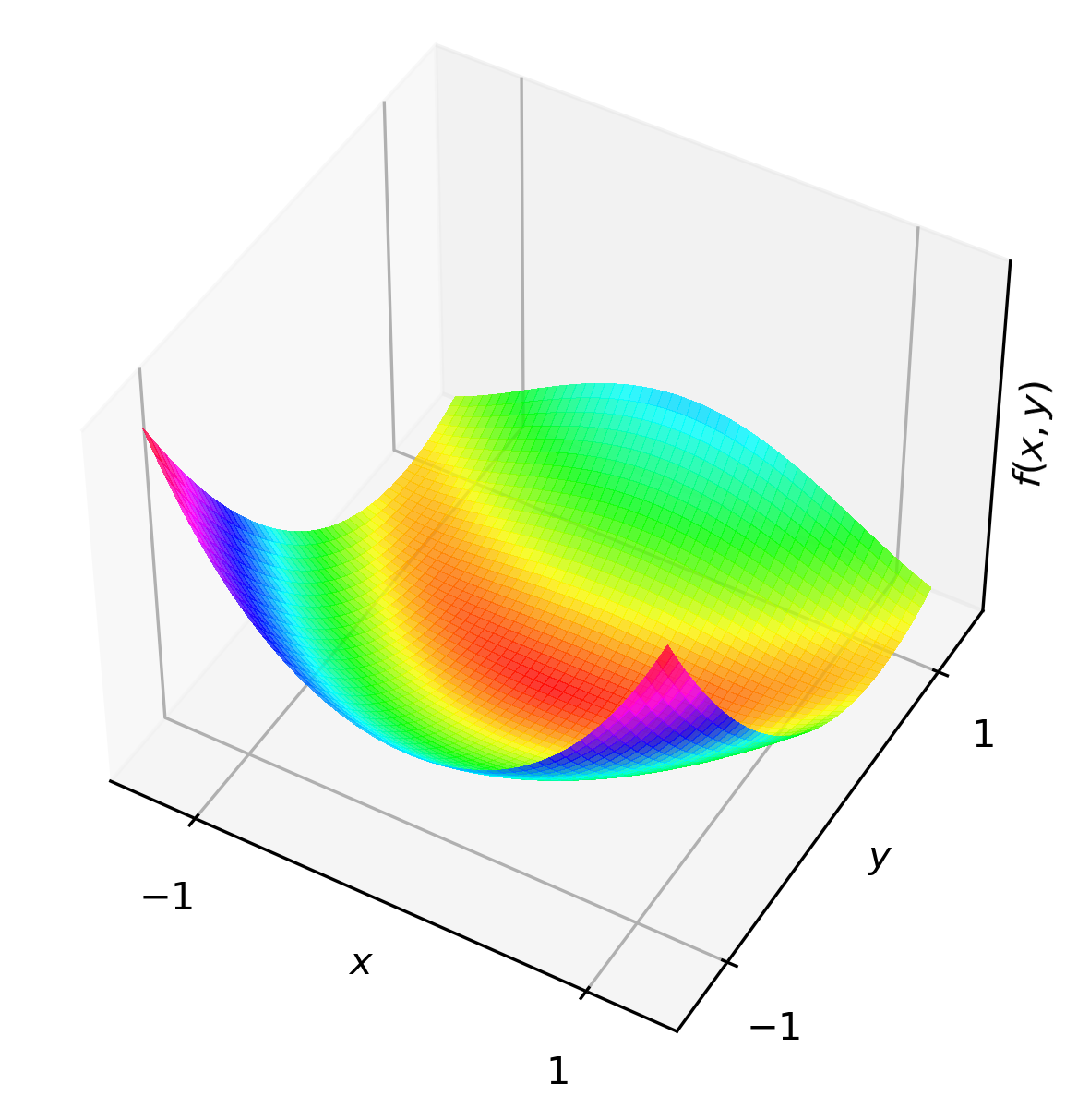}
    \begin{tabular}{|l|l|l|}
\hline
Algorithm (class) & $\Lambda$ (left plot)   & $\Lambda$ (right plot)       \\ \hline
GD ($\pl$)        & $3.2 \cdot 10^{-4}$  & $1.6 \cdot 10^{-2}$ \\ \hline
GD ($\sqc$)       & $1.3 \cdot 10^{-2}$   & $7.7 \cdot 10^{-5}$\\ \hline
NM ($\sqc$)       & $1.8 \cdot 10^{-2}$ & $1.8 \cdot 10^{-4}$  \\ \hline
\end{tabular}
    \caption{ Left: graph of $f(t) = 5(t+0.19\sin(5t))^2$, minimizer is $0$. Right: graph of $f(x,y) = 0.5(0.5x^2 -y)^2 + 0.05x^2$, minimizer is $(0,0)$. Bottom: table giving for each function the numerical value of the theoretical convergence rate for \eqref{eq:gd} or \eqref{eq:nm}, specifying which function class is used to characterize the bound (see Table~\ref{table:sc_sqc_rates}), where $\Lambda$ is the convergence rate. See implementation details in Appendix~\ref{app:param_details}. It shows that depending on the functions, $\pl$-based bound may be sharper than  $\sqc$-based bounds, and conversely.}
    \label{fig:pl_sqc_low_corr}
\end{figure}
\begin{figure}[t]
    \centering
    \includegraphics[width=0.85\linewidth]{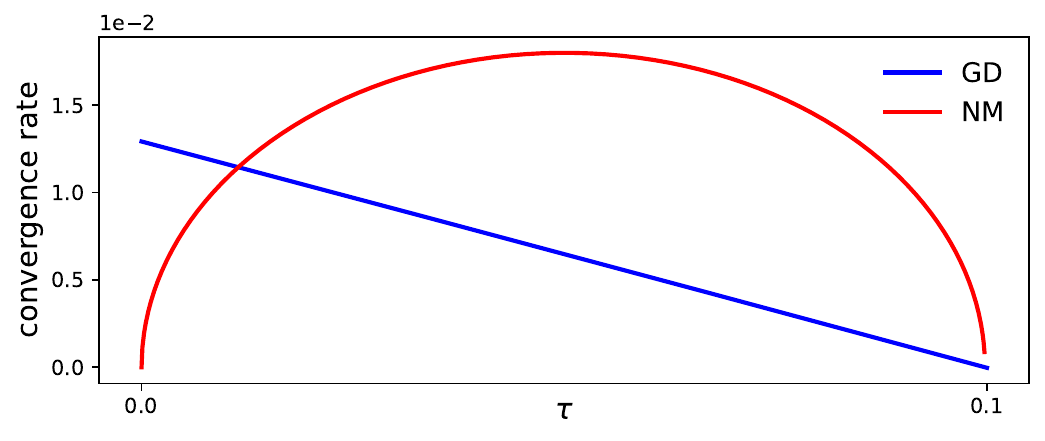}

     \caption{For a range of parameters $\tau$, we numerically compute the highest admissible $\mu$ ensuring $f \in \sqc $, and plot the numerical values of the associated theoretical convergence rates for \eqref{eq:gd} and \eqref{eq:nm} under $\sqc \cap \ls$, namely $\tau \mu/L$ and $\tau \sqrt{\mu/L}$, with $f(t) = 5(t+0.19\sin(5t))^2$.  See implementation details in Appendix~\ref{app:param_details}.
Note that the pair $(\tau, \mu)$ that maximizes the convergence rate of \eqref{eq:gd} differs from the pair that maximizes the one of \eqref{eq:nm}.} 
    \label{fig:sqc_flaws}
\end{figure}

\subsection{Discussion}
The previous discussion shows that the potential benefit provided by using \eqref{eq:nm} instead of \eqref{eq:gd} in strong quasar-convex optimization is not yet fully understood. We note that the problems raised in \cref{sec:specific_pitfall} can be solved by splitting $\sqc$ into two assumptions, \textit{i.e.} assuming $f \in \qginf$ and that it is quasar-convex, namely $\exists \tau \in(0,1]$ such that $f^\ast \ge f(x) + \frac{1}{\tau}\dotprod{\nabla f(x),x^\ast-x}$. However, the other points remain. The problem stems in the fact that convergence rates deduced under $\pl$ are not directly comparable to those deduced for classes such as $\sqc$, as the underlying geometrical constants encompass different information. Our idea in the remainder of the paper is then to keep working with the $\pl$ class, and pairing it with an additional and intuitive geometrical information.

\section{Aiming condition as an Alternative Characterization of Convergence}\label{sec:aiming_cond}

We introduce a class of functions that we will further intersect with $\pl$.
We assume that there exists a unique global minimizer $x^\ast$.
We initially assume the existence
of a unique minimizer $x^\ast$. We note that as discussed in \cref{sec:sqc_pitfalls}, functions that serve as lower bound for $\pl \cap \ls$ are designed in this setting, such that it remains an interesting setting to discuss acceleration. However, this assumption can be relaxed, as discussed in
Section~\ref{sec:avg} 
\begin{definition}\label{def:ac}
    For $a\in(0,1]$, we call $\ac$ the set of functions satisfying the \textit{aiming condition},\textit{ i.e.} $\forall x\in \R^d,~ \dotprod{\nabla f(x),x-x^\ast} \ge a \norm{\nabla f(x)}\norm{x-x^\ast}$.
\end{definition}

The parameter $a$ measures the alignment between the descent direction $-\nabla f$ and the
direction toward the minimizer. Values of $a$ close to $1$ indicate a strong alignment. In one dimension, any function in $\ac$ has perfect alignment, yielding $a=1$.
Such strong alignment  is particularly favorable to momentum methods, as it ensures that
the accumulated momentum reinforces progress toward $x^\ast$.
In contrast, if $a$ is small, momentum may amplify steps taken in poorly aligned directions.
For example, the rightmost function in Figure~\ref{fig:pl_sqc_low_corr} exhibits
a small aiming condition value, $a \approx 3\cdot10^{-4}$, which helps to explain why the convergence rate of
\eqref{eq:gd} is sharper than that of \eqref{eq:nm} in this case.  We discuss the important connection between $\ac$ and $\sqc$ in Section~\ref{app:link}. Before, we provide a toy construction that belongs to $\ac$.


\paragraph{A toy example}
 Let $f:\R_+ \to \R_+$ such that $f(0) = 0$ and $f'(t)> 0$ for all $t> 0$. Let $g : \R^d \to \R$, such that $g\ge 1$. Define $h :\R^d \to \R$
  \begin{equation}\label{eq:h}
    h(x) = f(\lVert x \rVert)g\left(\frac{x}{\lVert x \rVert} \right).
\end{equation}
$f$ describes the radial movement, while $g$ describes the angular oscillations. Such constructions were introduced as quasar-convex examples \cite{lee2016optimizingstarconvex}. In particular, $f \in \sqc$ implies $h \in \sqc$ \cite[Proposition 8]{hermant2024study}. 
 The more $g$ oscillates, the lower $a$ is. This is formalized with the following result.
\begin{proposition}\label{lem:toy_example_ac}
Let $h$ be defined by \eqref{eq:h}. Then, $h \in \ac$ with
    \[a= \inf_{r>0,u \in \mathbb{S}^{d-1}} \frac{1}{\sqrt{1+ \bpar{\frac{f(r)}{rf'(r)}}^2 \bpar{\frac{\norm{\nabla_{\sphere} g(u)}}{g(u)}}^2}},  \]
    where $\nabla_{\sphere} g(u)$ is the spherical gradient of $g$ restricted to $ \mathbb{S}^{d-1}$.
\end{proposition}
The spherical gradient $\nabla_{\sphere} g(u)$ encodes the angular variation. If $g$ is constant or has no angular variation, \textit{e.g.} $g(x) = \norm{x}^2$, then $\nabla_{\sphere} g(u)= 0$. In this case $h$ is a radial function and it follows that $h \in \acf{1}$. Supplementary structure on $f$--such as belonging to $\pl$ or $\qginf$--will transfer to $h$. However it appears clearly in our result that it will also impact the value of $a$ through the ratio $f(r)/(rf'(r))$.
\subsection{Link between aiming condition and strong quasar convexity}\label{app:link}
In this section, we investigate the link between $\ac \cap \pl$ and relaxations of strong convexity. 
\begin{lemma}\label{lem:scalar_prod_bound}
Let $f \in \pl \cap \ac \cap \qg $. Let $
\mu_0 := \sup\{\mu' \ge \mu : f \in \qginff{\mu'}\}$.
Then, for all $x\in\R^d$,
\[ f^\ast \ge f(x) + \frac{1}{a}\sqrt{\frac{L_0}{\mu}}
\dotprod{\nabla f(x), x^\ast - x}
 + \frac{\sqrt{\mu_0 L_0}}{2}\norm{x-x^\ast}^2.
\]
That is, $f \in \sqcf{a\sqrt{\mu/L_0}}{~ \sqrt{\mu_0 L_0}}$.
\end{lemma}
\begin{proof}
If $x=x^\ast$, the inequality holds trivially. Assume $x\neq x^\ast$.

\medskip
\noindent\textbf{Step 1 (AC + PL).}
Since $f\in\ac$, we have
\[
\dotprod{\nabla f(x),x-x^\ast} \ge a \norm{\nabla f(x)}\norm{x-x^\ast},
\quad\text{hence}\quad
\dotprod{\nabla f(x),x^\ast-x} \le -a \norm{\nabla f(x)}\norm{x-x^\ast}.
\]
Since $f\in\pl$, we also have $\norm{\nabla f(x)} \ge \sqrt{2\mu}\sqrt{f(x)-f^\ast}$, so
\begin{equation}\label{eq:pl_ac_core}
\dotprod{\nabla f(x),x^\ast-x}
\le -a\sqrt{2\mu}\,\sqrt{f(x)-f^\ast}\,\norm{x-x^\ast}.
\end{equation}

\medskip
\noindent\textbf{Step 2 (use $\qg$).}
Since $f\in\qg$,
\[
f(x)-f^\ast \le \frac{L_0}{2}\norm{x-x^\ast}^2
\quad\Rightarrow\quad
\norm{x-x^\ast}\ge \sqrt{\frac{2}{L_0}}\sqrt{f(x)-f^\ast}.
\]
Plugging this into \eqref{eq:pl_ac_core} yields
\begin{equation}\label{eq:pl_ac_0_corr}
\dotprod{\nabla f(x),x^\ast-x}
\le -a\sqrt{2\mu}\sqrt{\frac{2}{L_0}}\,(f(x)-f^\ast)
= -2a\sqrt{\frac{\mu}{L_0}}\,(f(x)-f^\ast).
\end{equation}

\medskip
\noindent\textbf{Step 3 (use $\qginff{\mu_0}$).}
By definition of $\mu_0$, we have
\[
f(x)-f^\ast \ge \frac{\mu_0}{2}\norm{x-x^\ast}^2
\quad\Rightarrow\quad
\sqrt{f(x)-f^\ast}\ge \sqrt{\frac{\mu_0}{2}}\norm{x-x^\ast}.
\]
Plugging this into \eqref{eq:pl_ac_core} yields
\begin{equation}\label{eq:pl_ac_1_corr}
\dotprod{\nabla f(x),x^\ast-x}
\le -a\sqrt{2\mu}\sqrt{\frac{\mu_0}{2}}\norm{x-x^\ast}^2
= -a\sqrt{\mu \mu_0}\,\norm{x-x^\ast}^2.
\end{equation}

\medskip
\noindent\textbf{Step 4 (combine the two bounds).}
From \eqref{eq:pl_ac_0_corr} and \eqref{eq:pl_ac_1_corr}, we have
\[
\dotprod{\nabla f(x),x^\ast-x}
\le -\max\Bigl\{2a\sqrt{\frac{\mu}{L_0}}\,(f(x)-f^\ast),\; a\sqrt{\mu\mu_0}\,\norm{x-x^\ast}^2\Bigr\}.
\]
Since $\max\{u,v\}\ge \frac{u+v}{2}$ for $u,v\ge 0$, we obtain
\[
\dotprod{\nabla f(x),x^\ast-x}
\le
-\frac12\Bigl(
2a\sqrt{\frac{\mu}{L_0}}\,(f(x)-f^\ast) + a\sqrt{\mu\mu_0}\,\norm{x-x^\ast}^2
\Bigr),
\]
which is the claim.
\end{proof}
This lemma shows that $\pl \cap \ac \cap \qg $ induces a specific parameterization of strong quasar-convexity.  We note that this parameterization does not suffer from the problem raised in Section~\ref{sec:specific_pitfall}. Importantly, it also follows that similar convergence analysis ideas can be used for $\pl \cap \ac \cap \qg $ and $\sqc$.  
Conversely, if assuming some smoothness regularity, strong quasar-convexity implies the aiming condition.
\begin{proposition}\label{prop:converse_link}
    Let $f \in \sqc$. If $f$ also satisfies 
    \begin{equation}
         \forall x \in \R^d,~ \norm{\nabla f(x)}\le L\norm{x-x^\ast},
    \end{equation}
    then $f \in \acf{\frac{2\mu}{(2-\tau)L}}$.
\end{proposition}
See the proof in Appendix~\ref{app:prop:converse}.

\section{Acceleration for PL Functions with an Aiming Condition}\label{sec:cv_results}


In Section~\ref{sec:local_weight}, we observe the impact of the aiming condition when studying standard gradient methods \eqref{eq:gf} and \eqref{eq:gd}. We consider in Section~\ref{sec:mom_cv} the convergence of 
\eqref{eq:nmo} and \eqref{eq:nm}, and the question of acceleration.




\subsection{$\ac$ Gives Weight to the Local Information at $x^\ast$}\label{sec:local_weight}
We first compare convergence bounds obtained for \eqref{eq:gf}, depending on whether $f$ belongs to $\ac$ or not.
\begin{theorem}\label{thm:gf_pl_&_ac}
    Let $\xcont \sim \eqref{eq:gf}$. 

    \noindent
    (i) \cite[Theorem 3]{polyak2017lyapunov}  If $f \in \pl$,  then
    $$ f(x_t)-f^\ast \le (f(x_0)-f^\ast)e^{-2 \mu t}. $$
    (ii) If $f \in \pl \cap \ac \cap \qg$, then
    $$ f(x_t)-f^\ast \le \bpar{1+\sqrt\frac{L_0}{\mu_0}}(f(x_0)-f^\ast)e^{-a\sqrt{\mu \mu_0}t}, $$
    where $\mu_0 := \sup\{\mu' \ge\mu : f \in \qginff{\mu'}\} $.
\end{theorem}
See the proof in Appendix~\ref{app:gf_pl_&_ac}. The assumption $f \in  \qg$ for bound (ii) is fairly mild. When considering algorithms, we will assume $f \in \ls$, such that they will naturally belong to $\qg$. We note that under the aiming condition, there is a worse constant factor $1+\sqrt{L_0/\mu_0}$, which mainly affects the bound for early iterations. In this work, we specifically focus on \hyperlink{cr}{convergence rates}.

\paragraph{Comparison of Convergence Rates} In addition to the PL constant $\mu$, the \hyperlink{cr}{convergence rates} in (ii) depend on a quadratic lower bound parameter $\mu_0$ and the aiming condition parameter $a$. As a result, rate (ii) improves upon rate (i) whenever $a \ge 2\sqrt{\mu/\mu_0}$. So, for functions with bad alignment, characterizing the convergence rate solely through the PL condition may lead to sharper rates. It is interesting to observe that $\mu_0$ appears in (ii), but not in (i). In the presence of flat plateaus, we may have $\mu \ll \mu_0$. This is due to $\mu_0$ being less affected by global irregularities of the function landscape, as we discuss next.

 \noindent\textbf{The Aiming Condition Exploits the Local Information at $x^\ast$ \:} For functions in $\strongconv \cap \ls$, the slowest direction toward the minimizer is the one with the smallest curvature~$\mu$, since it corresponds to the flattest direction (left plot of Figure~\ref{fig:worst_case} in the appendix).  
For functions in $\pl \cap \ls$, the situation is less clear: the trajectory may approach the minimizer either through low-curvature directions near $x^\ast$ or through plateau-like regions (right plot of Figure~\ref{fig:worst_case}). Thus, it is ambiguous whether the relevant $\mu$ is determined by local curvature at $x^\ast$ or by the plateaus, as $\pl$ captures both types of information.
Bound~(i) of Theorem~\ref{thm:gd_pl_&_ac} does not separate these effects, but bound~(ii) does: it reveals a trade-off between global plateaus and local curvature~$\mu_0$. In the special case $a=1$, if $\mu=\mu_0$, flatness near $x^\ast$ dominates and we recover the rate $\mathcal{O}(e^{-\mu t})$. If $\mu<\mu_0$, the plateaus matter, yielding a rate $e^{-\sqrt{\mu\mu_0}}$, which shows that plateaus do not fully determine the convergence.

\paragraph{Convergence of \eqref{eq:gd}}  Assuming the function belongs to $\ls$, we deduce similar convergence rates for \eqref{eq:gd}, with the difference that the $L$-smoothness constant intervenes.
\begin{theorem}\label{thm:gd_pl_&_ac}
    Let $f \in \ls$, $\xdisctilde \sim \eqref{eq:gd}$, $\gamma = \frac{1}{L}$.

    \noindent
    (i) \cite[Theorem 1]{karimi2016linear} If we also have $f \in \pl$,  then
    $$ f(\tilde x_k)-f^\ast \le \bpar{1-\frac{\mu}{L}}^k(f(\tilde x_0)-f^\ast). $$
    (ii) If we also have $f \in \pl \cap \ac $, then
    $$ f(\tilde x_k)-f^\ast \le  K_0\bpar{1-a\frac{\sqrt{\mu \mu_0}}{L}}^k(f(\tilde x_0)-f^\ast). $$
    with $K_0 := \bpar{1+\sqrt\frac{L_0}{\mu_0}}$, $\mu_0 := \sup\{\mu' \ge \mu: f \in \qginff{\mu'}\} $ and $L_0 :=\inf\{L' \le  L:f \in \qgf{L'}\}$.
\end{theorem}
Broadly speaking, the convergence rates of Theorem~\ref{thm:gd_pl_&_ac} are the same as those of Theorem~\ref{thm:gf_pl_&_ac}, divided by $L$, such that we can draw similar insights, see the proof in Appendix~\ref{app:gd_pl_ac}.

\subsection{Acceleration using Momentum with Large Enough Aiming Condition Constant}\label{sec:mom_cv}
We state the convergence rates of \eqref{eq:nmo} and \eqref{eq:nm}, under the aiming condition.
\begin{theorem}\label{thm:pl_accel}
Let $f\in \pl \cap \ac \cap \ls$, $L_0 := \inf \{L'\le L:  f \in \qgf{L'} \}$, $\mu_0 := \sup\{\mu' \ge \mu : f \in \qginff{\mu'}\}$.

(i)  Let $\xcont \sim \eqref{eq:nmo}$ and constant parameters $\gamma\ge 0$,  $\gamma' = (\mu_0 L_0)^{-1/4}$, $\eta = (\mu_0 L_0)^{1/4}$, $\eta' = a\bpar{\mu_0/L_0}^{1/4}\sqrt{\mu}$. Then,
     \begin{equation*}
     f(x_t)-f^\ast \le K_0 e^{-a \bpar{\frac{\mu_0}{L_0}}^{\frac{1}{4}}\sqrt{\mu}t},
 \end{equation*}
 where $K_0 :=  \bpar{1+\sqrt\frac{L_0}{\mu_0}}(f(x_0)-f^\ast)$ 
 
(ii) Let $\xdisctilde$ verify \eqref{eq:nm} with the \hyperlink{cont_param}{continuized parameterization}, with 
$\eta = (\mu_0 L_0)^{1/4}/\sqrt{L},\quad \gamma = 1/L,\quad \gamma' = (( \mu_0 L_0)^{1/4}\sqrt{ L})^{-1}, $ and $\eta' = a\bpar{\mu_0/L_0}^{1/4}\sqrt{\mu/L}$. Then,  
    with probability  $1-\frac{1}{c_0} - \exp(-(c_1-1-\log(c_1))k)$, for some $c_0 > 1$ and $c_1 \in (0,1)$ we have
$$ f(\tilde x_k)-f^\ast \le K_1e^{-a\bpar{\frac{\mu_0}{L_0}}^{1/4}\sqrt{\frac{\mu}{L}}(1-c_1)k}, $$
where $K_1 := c_0 \bpar{1+\sqrt\frac{L_0}{\mu_0}}(f(x_0)-f^\ast)$.
\end{theorem}
Statement (ii) can almost be deduced from \cite{wang2023continuized} by applying Lemma~\ref{lem:scalar_prod_bound}, we still provide a proof for completeness (Appendix~\ref{app:pl_accel}). The result of Theorem~\ref{thm:pl_accel} (ii) holds with high probability, due to the use of the \textit{continuized} version of Nesterov. 
In the following, we compare convergence rates despite this difference.

\noindent
\textbf{Comparison of the Convergence Rates: Discrete Case \:} The \hyperlink{cr}{convergence rates}  in Theorem~\ref{thm:gd_pl_&_ac} and Theorem~\ref{thm:pl_accel} (ii) are
\[
\frac{\mu}{L},\qquad 
a\frac{\sqrt{\mu \mu_0}}{L},\qquad 
a\!\left(\frac{\mu_0}{L_0}\right)^{1/4}\!\sqrt{\frac{\mu}{L}}(1-c_1).
\]
Ignoring the factor $(1-c_1)$, the latter always dominates $a\frac{\sqrt{\mu\mu_0}}{L}$, such that we compare $\mu/L$ and $a(\mu_0/L_0)^{1/4}\sqrt{\mu/L}$. 
Thus, for functions in $\pl\cap\ac\cap\ls$, the rate of  \eqref{eq:nm} is better than the rate of \eqref{eq:gd} whenever
\begin{equation}\label{eq:fill}
    a \;\ge\; \left(\frac{L_0}{\mu_0}\right)^{1/4}\sqrt{\frac{\mu}{L}}.
\end{equation}
The bound \eqref{eq:fill} suggests that, in the presence of oscillations and flat plateaus, leading to $\mu/L \ll 1$, there is increased tolerance for poor alignment while still allowing acceleration.
We test this bound numerically on the PL lower-bound function of \cite{PLlowerbound}, and find that the two quantities scale similarly, in a sense described in Appendix~\ref{app:heuristic}, providing a heuristic support for the relevance of our bound.

\paragraph{Acceleration for Dynamical Systems: Flatness Matters}
Comparing the rates of Theorem~\ref{thm:gf_pl_&_ac} (i) and Theorem~\ref{thm:pl_accel} (i) yields an improvement of momentum if $a \ge (L_0/\mu_0)^{1/4}\sqrt{\mu}$, namely a similar bound as a in the discrete case, up to the $\sqrt{L}^{-1}$ factor. Contrasting with the discrete case, the rate of Theorem~\ref{thm:pl_accel} (i) does not necessarily improve upon Theorem~\ref{thm:gf_pl_&_ac} (ii). Improvement occurs whenever
\[
\mu_0 L_0 < 1.
\]
Since $L_0 \ge \mu_0$, this implies that if $\mu_0>1$ (the function is sharp near the minimizer), then \eqref{eq:gf} is preferable, whereas if $L_0<1$ (flatness near minimizer), then \eqref{eq:nmo} performs better.
In the intermediate regime $\mu_0 < 1 < L_0$, the outcome depends on the balance between the smallest and largest local curvatures: if $\mu_0$ is much smaller than~$1$ while $L_0$ is only moderately above~$1$, flatness dominates and \eqref{eq:nmo} wins; conversely, in the opposite case.
The advantage of \eqref{eq:nmo} on flat landscapes is consistent with previous observations for the Heavy Ball dynamics when $f\in\pl$ \cite{apidopoulos2022convergence, gess2023exponential}, whose dynamics are close to \eqref{eq:nmo}.

\subsection{Decoupling information: more interpretability, but may create sub-optimality.} 
From Lemma~\ref{lem:scalar_prod_bound}, it follows that we do not study in this Section~\ref{sec:cv_results} a class of function that is fundamentally different from $\sqc$. We rather introduce a new characterization of these kinds of functions with the aim of improving interpretability. This is done by decoupling the information: angular information with $\ac$, and curvature with $\pl$. In contrast, all this information is blended with $\sqc$, and according to Section~\ref{sec:sqc_pitfalls}, in a fairly hidden way.

This separation of information comes at a cost: this may create sub-optimality. A clear example is that for a quadratic function in $\strongconv \cap \ls$, one may have $a=2\sqrt{\mu L}/(\mu+L)$. Plugged into our Table bounds, this creates significantly poorer bounds compared with the known bounds, see Section~\ref{table:sc_sqc_rates}. This sub-optimality follows from the definition of the constants parameterizing the class, that follow from \textit{worst-case points}. Define $a(x) = \dotprod{\nabla f(x),x-x^\ast}/({\norm{\nabla f(x)}\norm{x-x^\ast}})$, and assume the worst-case point $\hat{x} = \arginf_{x \in \R^d} a(x)$ is such that $a(\hat x) > 0$. Then, we have that $f \in \ac$ for $a = a(\hat x)$. For the PL condition, the $\mu$ constant is also determined by such a worst-case point. However, the worst case points might not be the same for the two conditions; this is exactly what happens for a strongly convex quadratic function, where the worst aiming condition point is reached where the curvature is high. When an assumption blends the two information in a unique inequality, the associated worst-case point balances the two conditions. For such rigid functions such as quadratic, the different information (angular, curvature) are strongly coupled, such that a single inequality can efficiently encompass all the information. In a non-convex case these information may be decoupled. This may be seen with our toy example  in Section~\ref{sec:aiming_cond}, where we can fix $a=1$ while playing with the curvature through $f$. This indicates that such a decoupling can bring interesting insight for non-convex functions, compared to the more traditional $\sqc$ based on a unique inequality.

Anyway, we note that the aforementioned limitation is rather artificial. A first reason is that it can be corrected if allowing a more involved definition of the classes. Precisely, the aiming condition will be used in our convergence proofs through Lemma~\ref{lem:scalar_prod_bound}. Its proof remains valid for $a$ defined as  \[a := \frac{\dotprod{\nabla f(\hat x),\hat x-x^\ast}}{\norm{\hat x-x^\ast}\norm{\nabla f(\hat x)}},\]  where
 \[\hat{x} :=\argmin_{x \in \R^d }\frac{\dotprod{\nabla f(x),x-x^\ast}}{\norm{x-x^\ast}\norm{\nabla f(x)}} \cdot \frac{\norm{\nabla f(x)}}{\sqrt{f(x)-f^\ast}}.\]
 Namely, the aiming condition constant may be considered at a point that creates a balance with the PL condition. This allows us to retain the information of the aiming condition, while not inducing an artificial sub-optimality. A second justification is provided in the next section.

\section{Relaxed Setting: Aiming Condition on Average}\label{sec:avg}
There is a gap between the convergence observed in practice and the theoretical bounds established for globally defined classes such as PL functions. 
The reason is intuitive: if the global landscape of non-convex functions can be highly irregular, nevertheless, the optimization path may remain in a favorable region. Importantly, the performance of first-order algorithms such as \eqref{eq:gd} and \eqref{eq:nm} only depends on the geometry surrounding their optimization path.

\noindent
\textbf{Naive Relaxation \: }A first direct idea to take this intuition into account is to consider an algorithm $\xdisc$ and a minimizer $\hat x$ such that $\tilde x_k \to_{k \to +\infty} \hat x$, and assuming
\begin{equation}\label{eq:naive_relax}
  \forall k\in \N,\dotprod{\nabla f(\tilde x_k),\tilde x_k-\hat x} \ge a \norm{\nabla f(\tilde x_k)}\norm{\tilde x_k-\hat x}.
\end{equation}
In words, $f$ is assumed to satisfy the $a$-aiming condition along the trajectory $\xdisc$, with respect to its convergence point $\hat x$. For a function in $\pl \cap \ls$, if $\xdisc \sim \eqref{eq:gd}$ or \eqref{eq:nm} verify \eqref{eq:naive_relax}, we would deduce for this sequence the same result as for Section~\ref{sec:cv_results}. 
The relaxation is twofold: (i) we get rid of the unique-minimizer assumption, and (ii) the convergence rate is not affected by potential pathological regions of $f$, as long as the sequence does not cross it. However, it suffices for \eqref{eq:naive_relax} to not be verified for a single iteration to make this reasoning fail. We thus derive a relaxation that, in the case where pathological regions remain a negligible section of the optimization path, still allows to keep accelerated results. 

\noindent
\textbf{Averaging Relaxation \: } We relax the viewpoint \eqref{eq:naive_relax}. We address the continuous case; the discrete case follows similarly, see Appendix~\ref{app:adptative_condition}.

\begin{assumption}\label{ass:avg_corelation}
    For some minimizer $x^\ast \in \R^d$, $\theta > 0$, $a_{avg} > 0$, and $\xcont$, we have
    \begin{align*}
        \int_0^t e^{\theta s}\Delta_sds
        \ge 0,
    \end{align*}
   $\Delta_s :=\dotprod{\nabla f(x_s),x_s-x^\ast}-a_{avg}\norm{\nabla f(x_s)}\norm{x_s-x^\ast}$.
\end{assumption}
A function $f$ satisfies Assumption~\ref{ass:avg_corelation} as long as the aiming condition holds along a specific average of the trajectory.
If $f \in \ac$, Assumption~\ref{ass:avg_corelation} holds with $a_{avg} = a$ for any 
initialization. However, we could have $a \ll a_{avg}$, or even if the path crosses some regions where the aiming condition is not verified, while still having $a_{avg} > 0$. 
\begin{theorem}\label{thm:gf_pl_&_avg_cor}
    Let $f \in \pl \cap \qg$, $\mu_0 := \sup\{\mu' \ge \mu: f \in \qginff{\mu'}\} $ and $a_{avg}>0$. If:
    
    (i) $\xcont \sim \eqref{eq:gf}$ for a given initialization $x_0 \in \R^d$ is such that Assumption~\ref{ass:avg_corelation} holds with $\theta = a_{avg}\sqrt{\mu \mu_0}$, then
    $$ f(x_t)-f^\ast \le K_0e^{-a_{avg}\sqrt{\mu \mu_0}t}(f(x_0)-f^\ast). $$
    
    (ii) $\xcont \sim \eqref{eq:nmo}$ for $\eta' = a_{avg}\bpar{\mu_0/L_0}^{1/4}\sqrt{\mu}, \gamma' =  (\mu_0 L_0)^{-1/4},~\eta = (\mu_0 L_0)^{1/4}, ~ \gamma \ge 0$ and such that 
Assumption~\ref{ass:avg_corelation} holds with $\theta =  a_{avg}\bpar{\mu_0/L_0}^{1/4}\sqrt{\mu}$,  then
    $$f(x_t)-f^\ast \le K_0e^{- a_{avg}\bpar{\frac{\mu_0}{L_0}}^{\frac{1}{4}}\sqrt{\mu}} (f(x_0)-f^\ast).$$
    In both statements, $K_0 :=  \bpar{1+\sqrt\frac{L_0}{\mu_0}}.$
\end{theorem}
The proof is in Appendix~\ref{app:gf_pl_&_avg_cor}. We obtain the same convergence rate as for Theorem~\ref{thm:gf_pl_&_ac} (ii) and Theorem~\ref{thm:pl_accel} (i), replacing $a$ with $a_{avg}$. However, the rates of Theorem~\ref{thm:gf_pl_&_avg_cor} could be either significantly sharper, or allow for obtaining such a rate while the algorithm crosses regions with negative aiming condition values. This reasoning can be extended to other geometrical conditions, such as $\pl$. This perspective might explain why, in practice, grid-search selection of parameters often leads to relatively large constants, compared to what would be allowed by geometrical constants that parameterize globally the associated landscape.

\section{Numerical Experiments}\label{sec:numerical}
We provide numerical experimentss to support our theoretical findings. We first design a simple $2$d example that belongs to $\pl \cap \ls$, has a unique minimizer, but does not belong to $\ac$. This will cause momentum to slow convergence in the early stage of optimization, see Section~\ref{sec:numerical_1}. Then, in Section~\ref{app:heuristic}, we provide a heuristic validation of the bound to have acceleration with \eqref{eq:nm} provided in Section~\ref{sec:mom_cv}.


\subsection{An Example with Negative Aiming Condition}\label{sec:numerical_1}
In this section, we design an example that shows how negative aiming condition values can hurt the convergence of momentum algorithms, and then discuss the aiming condition in neural networks. For our experiment, we consider an alternative formulation of \eqref{eq:nm} that we call \eqref{eq:nm_prime}.
\begin{equation}\label{eq:nm_prime}\tag{NM'}
    \begin{array}{c}
\left\{
\begin{aligned}
&\tilde y_k = \tilde x_k + \alpha(\tilde x_k -\tilde x_{k-1})  \\
&\tilde x_{k+1} =\tilde y_k - \gamma\, \nabla f(\tilde y_k)
\end{aligned}
\right.
\end{array}
\end{equation}
While \eqref{eq:nm} with the continuized parameterization is convenient for theoretical analysis, \eqref{eq:nm_prime} is more amenable to practical implementation, as it involves only two easily tunable parameters with clear interpretations: the stepsize $\gamma$ and the momentum parameter $\alpha$. See \cite{defazio2019curved} for a discussion on the different forms of Nesterov momentum.

Consider the 2-dimensional function\[F_{\varepsilon}(x,y) = 0.5(y-\sin(x))^2 + 0.5\varepsilon x^2.\]
The special case $\varepsilon = 0$ is introduced in \cite{apidopoulos2022convergence}. 
For $\varepsilon > 0$, $F_{\varepsilon}$ has $(0,0)$ as unique minimizer, and it belongs to $\pl$ with $\mu \le \frac{2+\varepsilon
-\sqrt{\varepsilon^2+4}}{2}$, see Lemma~\ref{lem:toy_example_pl}. We display $F_{0.001}$ in Figure~\ref{fig:toy_prob}. We also provide a three-dimensional visualization in the appendix, see Figure~\ref{fig:3d_vizualization}. 
At certain points of the landscape, the descent direction is anti-correlated with the direction pointing toward the minimizer, thereby violating the aiming condition.
\begin{figure}[h]
    \centering
    \begin{minipage}{0.65\linewidth}
        \centering
        \begin{minipage}{0.48\linewidth}
            \centering
            \includegraphics[width=\linewidth]{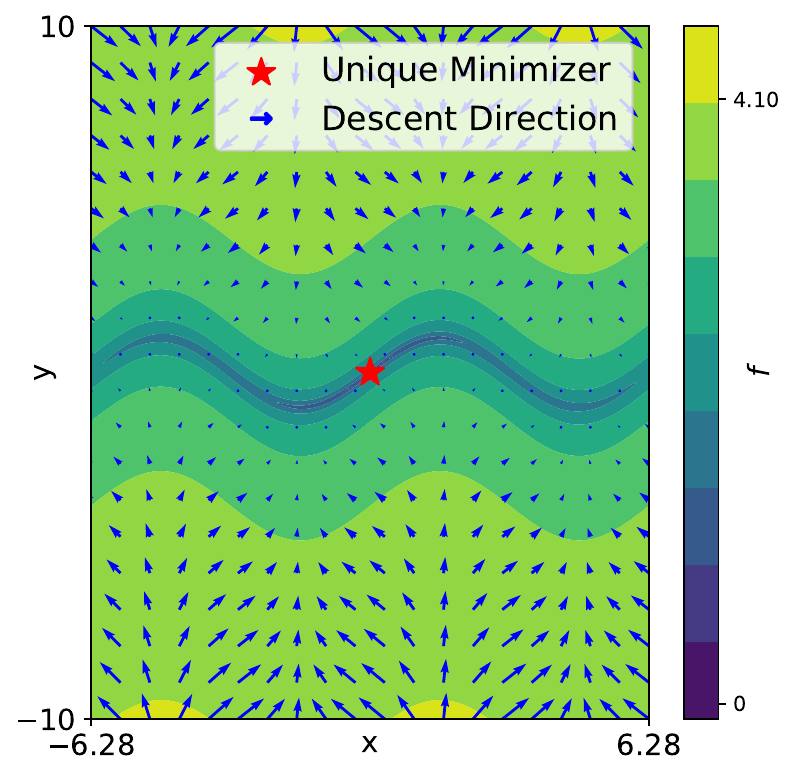}
        \end{minipage}
        \begin{minipage}{0.48\linewidth}
            \centering
            \includegraphics[width=\linewidth]{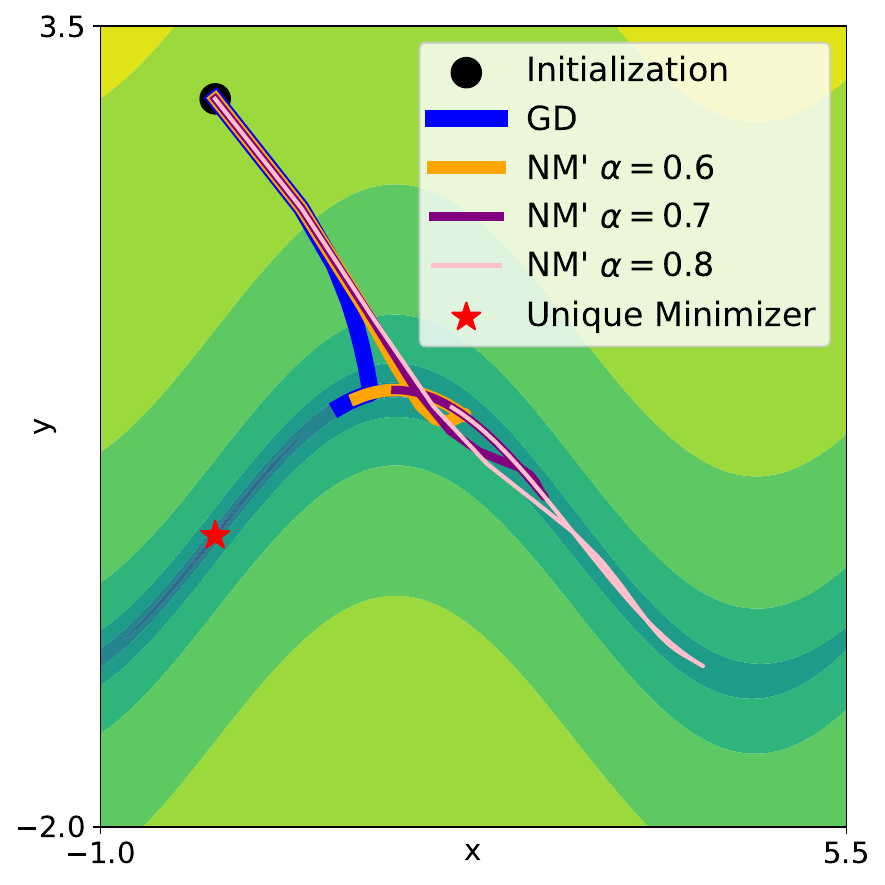}
        \end{minipage}
    \end{minipage}
    \begin{minipage}{0.3\linewidth}
        \centering
        \includegraphics[width=\linewidth]{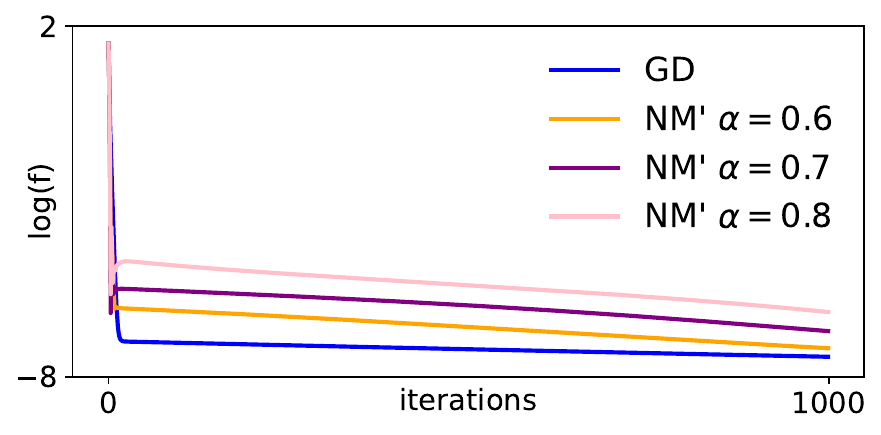}

        \vspace{0.3cm}

        \includegraphics[width=\linewidth]{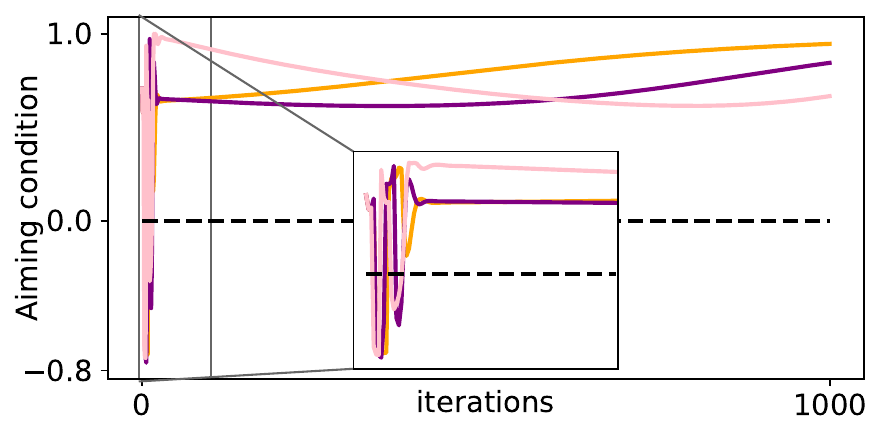}
    \end{minipage}

  \caption{ Left: Heatmap of $F_{0.001}(x,y) = 0.5(y-\sin(x))^2 + 0.001\cdot0.5x^2$. Blue arrows indicate the descent direction $-\nabla F_{0.001}$. Center: First $1000$ iterations of the trajectories of \eqref{eq:gd} and \eqref{eq:nm_prime}, both starting from the initialization point $(0,3)$, for different values of momentum parameter $\alpha$.  Top right: Corresponding decrease of $\log(f)$. Bottom right: Values of the aiming condition along the iterates, zoomed on the first 100 iterations. Early negative aiming condition values cause momentum to drive the trajectory away from the minimizer, leading \eqref{eq:gd} to outperform momentum in early iterations, with increasing effect for larger $\alpha$.}
    \label{fig:toy_prob}
\end{figure}

\noindent
\textbf{Negative aiming condition shoots \eqref{eq:nm_prime} away \:} In Figure~\ref{fig:toy_prob}, we compare \eqref{eq:gd} and \eqref{eq:nm_prime} for different values of momentum parameter $\alpha$, see details in Appendix~\ref{app:toy_prob}. Due to negative aiming condition values during the early phase of the optimization process, iterates of \eqref{eq:nm_prime} are initially driven away from the minimizer. Increasing the momentum parameter amplifies this effect, causing the trajectory to deviate even further. 
As a result, \eqref{eq:gd} outperforms \eqref{eq:nm_prime} for up to $1000$ iterations in terms of function value decrease. This example illustrates how negative aiming condition values can hinder momentum methods. 
More broadly, designing functions with unfavorable aiming behavior while retaining desirable optimization properties—such as satisfying the PL condition and having a unique minimizer—appears to be an interesting direction for better understanding landscapes on which momentum is ineffective.

\paragraph{\eqref{eq:nm_prime} eventually catches up due to favorable averaged aiming condition \:}
Figure~\ref{fig:toy_prob} shows that, although the aiming condition takes negative values at early iterations, it then stabilizes around a positive value close to one. As discussed in Section~\ref{sec:avg}, this suggests that \eqref{eq:nm_prime} should eventually benefit from momentum for a large enough number of iterations. We confirm this behavior in Figure~\ref{fig:convergence_longer} in the appendix.

\paragraph{Aiming Condition and Neural Networks}

Because of the curse of dimensionality and their high level of non-convexity, we could expect that large-scale models such as those involved in deep learning do not verify any aiming condition. This intuition is challenged by the fact that well-designed neural networks are known to have favorable optimization properties. It was observed in \cite{li2018visualizing} that around 90$\%$ of the variations of the path of SGD with large batch size lie in a two-dimensional subspace, despite the high dimensionality of the models. Of special interest to our work, \cite{guille2024no} performs an extensive empirical evaluation of the aiming condition, which they name the \textit{cosine similarity}. They observe that for various tasks and architectures, $a$ remains between $10^{-1}$ and $10^{-3}$, despite the high dimensionality of the model. This observation was consistent across several stochastic first-order algorithms, including SGD with momentum. Although our theoretical results do not apply to general deep learning models, these experiments tend to highlight the crucial role of the aiming condition as a key ingredient in explaining the effectiveness of momentum.
\subsection{Heuristic Evaluation of the Acceleration Bound}\label{app:heuristic}
In this section, we provide a heuristic validation of the bound on the aiming condition to provide acceleration, derived in Section~\ref{sec:mom_cv}. To do so, we consider the $\pl$ function introduced in \cite{PLlowerbound}, which serves to obtain a lower bound on the number of gradient calls needed by first order algorithms--algorithms that make use of gradient information--to obtain a point $\hat x \in \R^d$ such that $f(\hat x)-\min f \le \varepsilon$. Precisely, it is shown that the number of needed gradient calls scales at best as $\bigO(\frac{L}{\mu}\log(\frac{1}{\varepsilon}))$ for any first-order algorithms. This bound corresponds to a $\mu/L$ convergence rate, inducing that, up to constants, the convergence rate of \eqref{eq:gd} is optimal, see Table~\ref{table:sc_sqc_rates}. The definition of the hard instance is rather involved see details in Appendix~\ref{app:hard_fonc_sqc}. 

\textbf{Non acceleration of momentum on the hard function \:}
We observe empirically that, as expected, momentum does not allow improvement over gradient descent when minimizing the hard function of \cite{PLlowerbound}. We consider \eqref{eq:gd} and \eqref{eq:nm} with the \hyperlink{cont_param}{continuized parameterization}. We chose a parameterization of the hard function such that it belongs to $\pl \cap \ls$ with $\mu=10^{-4}$, $L = 10^3$, and the dimension is $4550$ (see details in Appendix~\ref{app:hard_fonc_sqc}). For both \eqref{eq:gd} and \eqref{eq:nm}, we tune the parameters such that the decrease at iteration $1000$ is maximized. We plot the decrease of function values of the hard function in Figure~\ref{fig:hard_instance}. We observe that tuning \eqref{eq:nm} only allows it to match the decrease achieved by \eqref{eq:gd}, but provides no acceleration. 

\textbf{Heuristic test of the acceleration bound \:}
In Section~\ref{sec:mom_cv}, we stated that for $f \in \pl \cap \ac \cap \ls$, momentum provides acceleration as long as 
\begin{equation}\label{eq:recall_acc_bound}
    a \ge \bpar{\frac{L_0}{\mu_0}}^{1/4}\sqrt{\frac{\mu}{L}}.
\end{equation}
It is natural to test whether this bound is verified or not on the aforementioned hard function. Importantly, the following discussion should be regarded as a \textbf{heuristic}, rather than a rigorous derivation.

For our test, we do not consider the global parameters $\mu, a ,L$ such that $f$ belongs to $\pl \cap \ac \cap \ls$. This is because (i) we do not know the optimal parameters $a, \mu_0$ and $L_0$, and (ii) in the lower bound provided in \cite{PLlowerbound}, there is a hidden constant with values $\approx 10^6$, which is not negligible. Moreover, the decrease observed in Figure~\ref{fig:hard_instance} is significantly faster than what is predicted by the theoretical bound derived under the global parameters, such that it actually poorly describes the algorithm's behavior. This motivates to use a perspective that has a similar spirit to the one described in Section~\ref{sec:avg}, such that we will consider the values of PL and aiming condition taken along the optimization path. Along the iterates of $\xdisc$ generated by \eqref{eq:gd} or \eqref{eq:nm}, we compute the point-wise PL and aiming condition values, namely
\[\mu(\tilde x_i) :=  2\frac{\norm{\nabla f(\tilde x_i)}^2}{f(\tilde x_i)-f^\ast}, \quad a(\tilde x_i) := \frac{\dotprod{\nabla f(\tilde x_i),\tilde x_i-x^\ast}}{\norm{\nabla f(\tilde x_i)}\norm{\tilde x_i-x^\ast}}.\] Empirically, we observe that $\mu(\tilde x_i) \approx 1$ (see Figure~\ref{fig:mu_corr_values}), while the function is $\pl$ with $\mu = 10^{-4}$. This suggests that the algorithm stays in favorable regions compared to those existing on the global landscape. As a result, in the case of \eqref{eq:gd}, compared to the rate $\mu/L$, we expect that its convergence rate is more accurately described by 
\[\tilde \mu_{gd}/L,\]

with $\tilde \mu_{gd} = \frac{1}{N}\sum_{i=1}^N \mu(\tilde x_i)$ is the averaged point-wise PL values along the $N$ iterations of the algorithm. Moreover, because we observed that \eqref{eq:gd} achieves a better decrease with stepsize $\gamma_{gd} \ge 1/L$, running the algorithm with this stepsize suggests that an even more accurate rate is
\begin{equation}\label{cvmodel_bound_gd}
    \tilde \mu_{gd}\gamma_{gd}.
\end{equation}
We plot the curve associated to this theoretical rate in Figure~\ref{fig:hard_instance}. We observe that, if not matching perfectly the decrease of \eqref{eq:gd}, it is significantly more accurate compared with the rate $\mu/L$.
Following the same reasoning, we will model the rate for \eqref{eq:nm} as (see Theorem~\ref{thm:pl_accel} (ii))
\begin{equation}\label{cvmodel_bound_nm}
    \tilde a (\tilde \mu_0/ \tilde L_0)^{1/4} \sqrt{\tilde \mu_{nm}\gamma_{nm}},
\end{equation}

where $\tilde a = \frac{1}{N}\sum_{i=1}^N a(\tilde x_i)$ and $\tilde \mu_{nm} = \frac{1}{N}\sum_{i=1}^N \mu(\tilde x_i)$ are respectively the point-wise aiming condition value and point-wise PL values along the algorithm, and $\tilde \mu_0$ and $\tilde L_0$ are estimates of the optimal constants involved in $\qginf$ and $\qg$, and $\gamma_{nm}\ge L$ is the used stepsize for \eqref{eq:nm} that arises from the tuning of parameters. The parameters used in \eqref{eq:nm} result from grid-search rather than the theoretical recommendation of Theorem~\ref{thm:pl_accel}, and we observe that \eqref{eq:nm} has a slightly better decreasing rate compared to the bound \eqref{cvmodel_bound_nm}.
Up to associating the empirical decreasing rate of \eqref{eq:gd} and \eqref{eq:nm} with the bounds \eqref{cvmodel_bound_gd} and \eqref{cvmodel_bound_nm}, we thus test the acceleration condition of Section~\ref{sec:mom_cv} by comparing $\tilde a $ with

\[\frac{ \tilde \mu_{gd}\gamma_{gd} }{\sqrt{\gamma_{nm}\tilde \mu_{nm}}}  \bpar{\frac{\tilde L_0}{\tilde \mu_0}}^{\frac{1}{4}}.\]
We find numerically that $\tilde a \approx 0.16$, while $ \frac{ \tilde \mu_{gd}\gamma_{gd} }{\sqrt{\gamma_{nm}\tilde \mu_{nm}}}  \bpar{\frac{\tilde L_0}{\tilde \mu_0}}^{\frac{1}{4}} \approx 0.19$ such that they are roughly the same. So, in this setting where \eqref{eq:nm} does not accelerate over \eqref{eq:gd}, we observe that the acceleration condition of Section~\ref{sec:mom_cv} is indeed not verified. This lends heuristic support to the relevance of this bound. Finally, this suggests that the function that serves as lower bound for PL functions, which actually belongs to $\pl \cap \ac \cap \ls$, is such that the aiming condition values do not scale properly with other geometrical constants to permit acceleration using momentum.
\begin{figure}
    \centering
\includegraphics[width=0.8\linewidth]{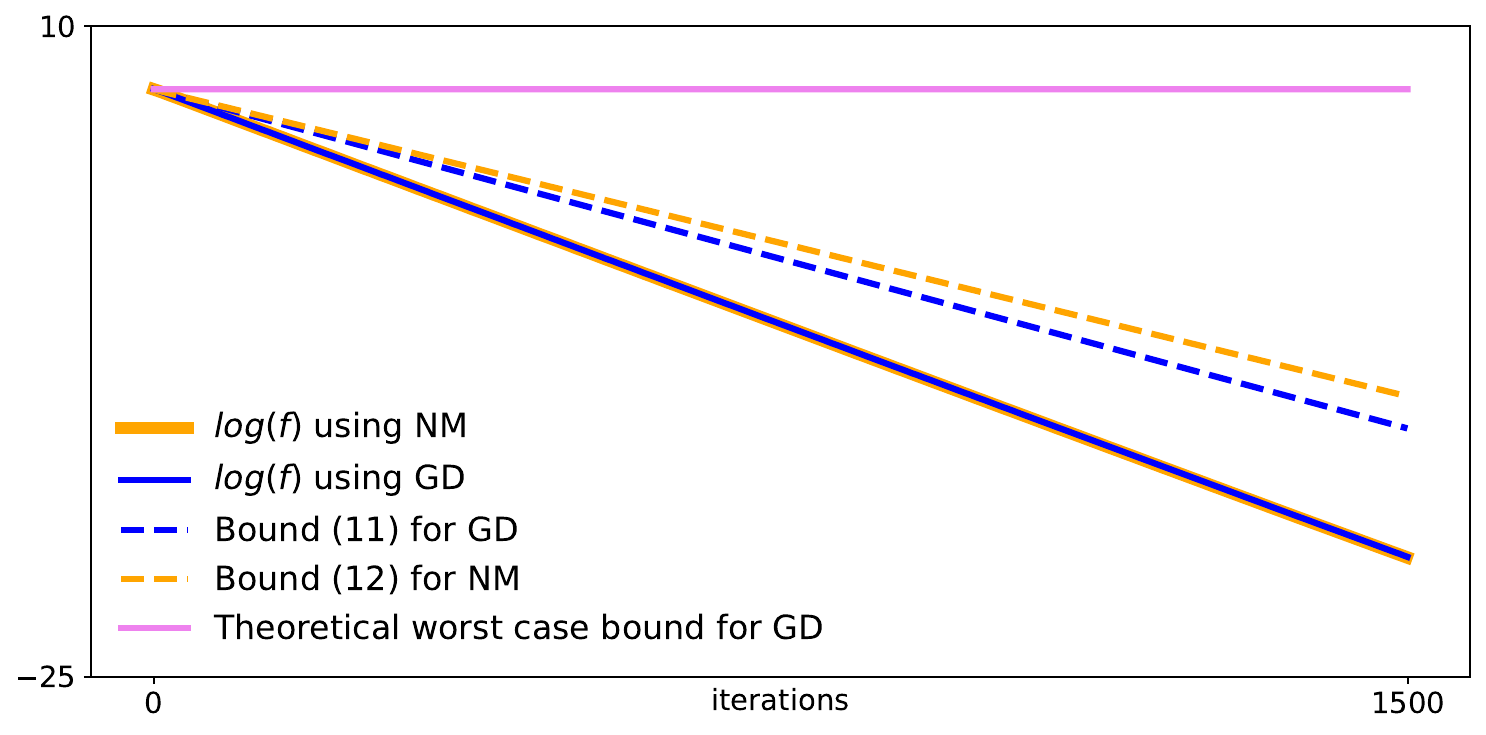}
    \caption{Function values decrease when running \eqref{eq:gd} and \eqref{eq:nm} on the hard instance of \cite{PLlowerbound}. We also plot the log of the theoretical convergence upper bound $\log((1-\mu \gamma)^k(f(x_0)-f^\ast)$ with $\mu$ the global PL constant, and the log of the theoretical convergence upper bound using parameter values across the path, see \eqref{cvmodel_bound_gd} and \eqref{cvmodel_bound_nm} in Section~\ref{app:heuristic}. It can be observed that \eqref{eq:gd} and \eqref{eq:nm} decrease at the same speed, and that the bounds proposed at \eqref{cvmodel_bound_gd} and \eqref{cvmodel_bound_nm} are significantly more accurate compared with the global worst case bound.}
    \label{fig:hard_instance}
\end{figure}

\begin{figure}[t!]
        \centering
  \includegraphics[width=0.45\linewidth]{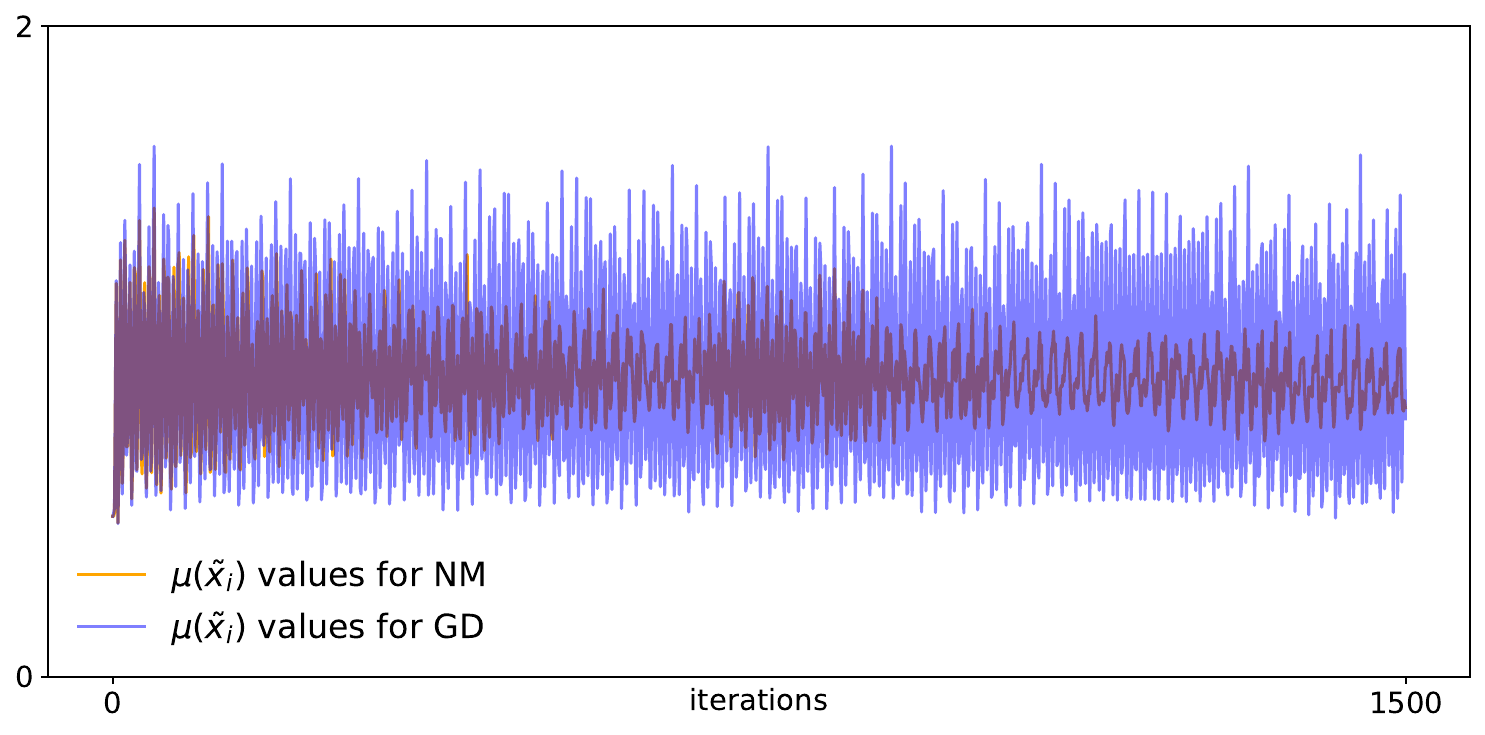}
    \includegraphics[width=0.45\linewidth]{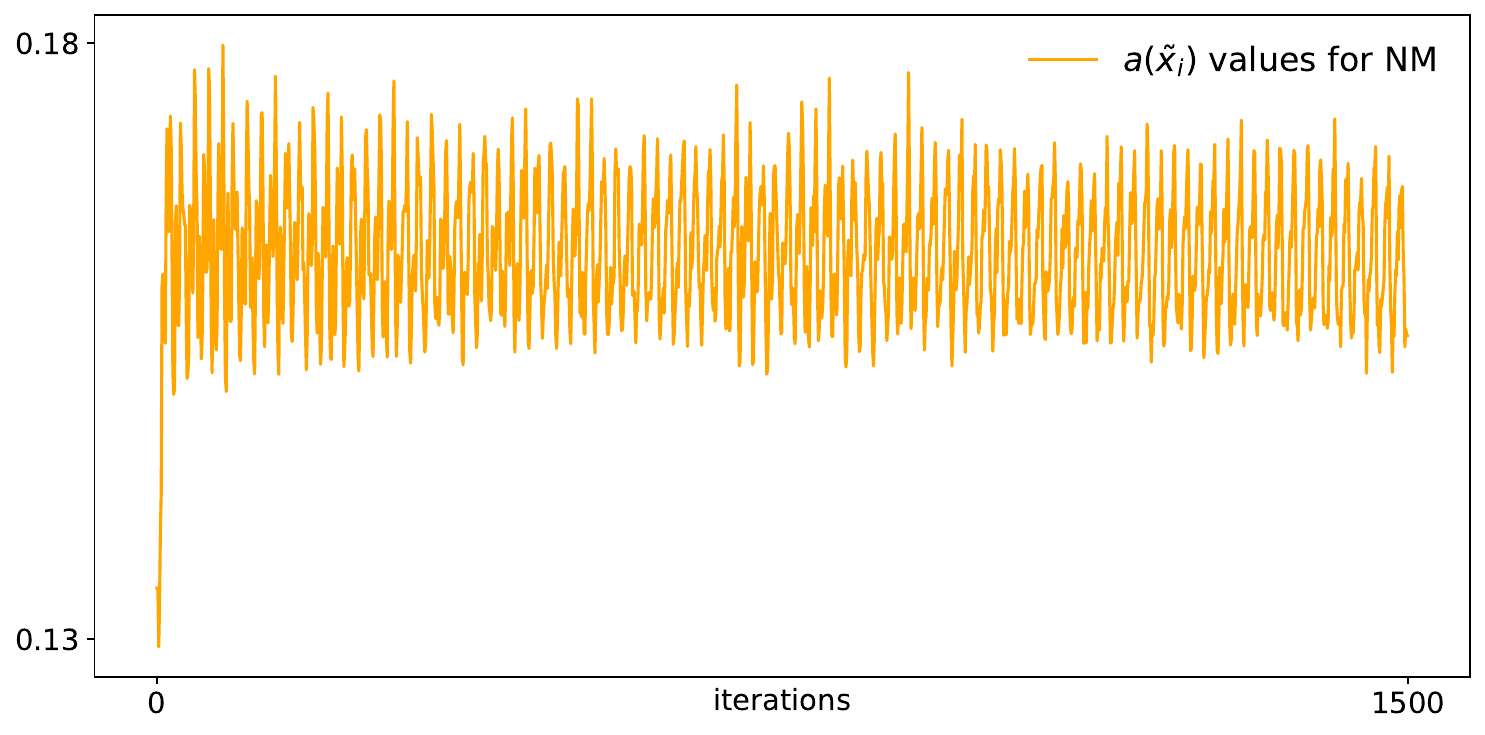}

    \caption{Left: Values of $\mu(\tilde x_i)$ along the iterates of \eqref{eq:gd} and \eqref{eq:nm}. Right: Values of $a(\tilde x_i)$ along the iterates of \eqref{eq:nm}. The underlying function is the hard instance of \cite{PLlowerbound}, see details in Appendix~\ref{app:hard_fonc_sqc}.} 
    \label{fig:mu_corr_values}
\end{figure}
\section*{Reproducibility Statement}
Our figures and experiments are reproducible, see our code \url{https://github.com/J-Hermant/Nesterov_Acceleration_Aiming_Condition/tree/main}.
\section*{Acknowledgment}
This work is supported by PEPR PDE-AI and the ANR SOS2ID (grant ANR-24-CE40-3786). We thank Jean-François Aujol, Erell Gachon and Marien Renaud for their feedback, which helped improve the presentation and clarity of this work. 
\bibliographystyle{plain} 
\bibliography{bib}

\newpage
\appendix
\onecolumn
\section{Related Works and Background Details}
In this section, we first provide bibliographical details about Table~\ref{table:sc_sqc_rates} (Appendix~\ref{app:related_work_tab}), then give more details about a specific related work in Appendix~\ref{app:related:uaac}.

\subsection{Details for Table~\ref{table:sc_sqc_rates}}\label{app:related_work_tab}
In this section, we detail Table~\ref{table:sc_sqc_rates}, giving in particular references for the results. We note that some of these results do not appear in the literature. Before, we state the following result that we will use in this section.
\begin{proposition}\label{prop:equiv_hb_nmo}
    Let $\xz$ verify \eqref{eq:nmo}. Then, it verifies the Heavy Ball equation with Hessian damping
    \begin{equation}
    \ddot{x}_t
    + \bpar{\eta_t - \frac{\dot{\eta}_t}{\eta_t} + \eta'_t }\dot{x}_t
    + \gamma \nabla^2 f(x_t)\dot{x}_t
    + \bpar{ \eta'_t \gamma + \eta_t \gamma' -\frac{\dot{\eta}_t}{\eta_t}\gamma }\nabla f(x_t)
    = 0.
\end{equation}
\end{proposition}
Proposition~\ref{prop:equiv_hb_nmo} states that \eqref{eq:nmo} can be written as the Heavy Ball equation with Hessian damping \cite{alvarez2002second}. We will use results established for this equation in the case of \eqref{eq:nmo}.

\textbf{About $\strongconv$ \:} The rate for \eqref{eq:gf} can be deduced from the one in the $\pl$ case. The rate stated for \eqref{eq:nmo} can be found in the case of heavy ball with Hessian damping in \cite{shi2018understanding}, from which we deduce the same rate for \eqref{eq:nmo}.  The rates for \eqref{eq:gd} and \eqref{eq:nm} can be found in \cite{nesterovbook}.

\textbf{About $\sqc$ \:} See Proposition~\ref{prop:gf_sqc} for the rate of \eqref{eq:gf}. The rate of \eqref{eq:gd} can be found in \cite[Proposition 3]{hermant2024study}. The rate stated for \eqref{eq:nmo} can be found in the case of heavy ball with Hessian damping in \cite[Theorem 5]{hermant2024study}, from which we deduce the same rate for \eqref{eq:nmo}. The rate for \eqref{eq:nm} requires some care. It can be found in \cite[Theorem 1]{hinder2020near}, but involves a line-search procedure that requires computing evaluations of $f$ and $\nabla f$ at each iteration to compute one of the parameters. It can also be found in \cite[Theorem 2]{hermant2024study} up to a supplementary assumption of weak-convexity, namely $f + \rho/2\norm{\cdot}^2$ is convex for some $\rho \in \R$, with $\rho \le L$. Finally, it can be found in \cite[Theorem 2]{wang2023continuized} up to non-deterministic results, because of the use of the continuized parameterization.

\textbf{About $\pl$ \:} The rate $\mu$ for \eqref{eq:gf} can be found in \cite[Theorem 3]{polyak2017lyapunov}. The rate $\mu/L$ for \eqref{eq:gd} has been known for a long time \cite{Polyak1963,bolte2017error}, see \cite[Theorem 1]{karimi2016linear} for a direct proof. The rate for \eqref{eq:nm} can be found in \cite[Proposition 7]{hermant2025continuized}, but holds with high probability because of the use of the continuized framework. We note that using the Heavy Ball algorithm \eqref{eq:hb}, this rate can be improved to a rate close to $\sqrt{\mu/L}$ if assuming that $f$ is four times differentiable, and if the algorithm is initialized close enough to the minimizer \cite{kassing2024polyak}. The case of \eqref{eq:nmo} is more complicated. For Polyak's momentum (Heavy Ball equation), a result exists assuming $f$ is also in $\ls$, yielding a convergence rate $\approx \mu/\sqrt{L}$ \cite{apidopoulos2022convergence}. For the same equation, the rate $\sqrt{\mu}$ holds asymptotically, provided that $f$ is four times differentiable, providing strong local regularity around minimizers \cite{kassing2024polyak}. In the case of Heavy Ball with Hessian damping, a rate $\sqrt{\mu}$ is deduced \cite{apidopoulos2025heavy}, but the gradient term in the equation is rescaled by $\bigO(1/\sqrt{\mu})$, which means we are minimizing a rescaled version of the function. In Proposition~\ref{prop:nmo_pl}, we actually show without assuming $f \in \ls$ or smoother than $C^1$--and with a considerably simpler proof than the previously mentioned works--that \eqref{eq:nmo} achieves a rate $\gamma \mu$, where $\gamma$ is a parameter of \eqref{eq:nmo}. It highlights that we can achieve whatever rate we want, as long as we take $\gamma$ big enough. However, in the discrete case, $\gamma$ corresponding intuitively to the stepsize, it appears clearly that the descent lemma induced by $f \in \ls$ will impose a bound on $\gamma$, namely $\gamma \le 1/L$, yielding a rate $\mu/L$.

\subsubsection{Proof of not existing rates of Table~\ref{table:sc_sqc_rates}}\label{app:table_proof}
\begin{proposition}\label{prop:gf_sqc}
Let $f\in\sqc$, $\xcont\sim~\eqref{eq:gf}$. Then,
\[
f(x_t)-f^\ast = \bigO(e^{-\tau \mu t}).
\]
\end{proposition}
\begin{proof}
Let
\[
E_t =f(x_t)-f^\ast + \frac{\mu}{2}\norm{x_t-x^\ast}^2.
\]
Then
\begin{align*}
\dot E_t
&= \dotprod{\grad{x_t},\dot x_t} + \mu \dotprod{x_t-x^\ast,\dot x_t}\\
&\overset{(i)}= -\norm{\grad{x_t}}^2 - \mu\dotprod{\grad{x_t},x_t-x^\ast}\\
&\le - \mu\dotprod{\grad{x_t},x_t-x^\ast}\\
&\overset{(ii)}\le -\mu \tau\Bigl(f(x_t)-f^\ast + \frac{\mu}{2}\norm{x_t-x^\ast}^2\Bigr)
= -\mu\tau E_t,
\end{align*}
where (i) follows from \eqref{eq:gf} and (ii) from $f\in\sqc$.
Integrating yields $E_t\le e^{-\mu\tau t}E_0$, hence $f(x_t)-f^\ast\le E_t=\bigO(e^{-\mu\tau t})$.
\end{proof}

\begin{proposition}\label{prop:nmo_pl}
Let $f \in \pl$. Let $\xcont \sim$ \eqref{eq:nmo} with $\eta_t\equiv \eta$ and $\eta'_t\equiv \eta'$ constants and under the constraint $\gamma = (\eta+\eta')/\mu$. Then, we have
\[
f(x_t)-f^\ast = \bigO(e^{-2\mu \gamma t}).
\]
\end{proposition}
\begin{proof}
Let
\[
E_t = f(x_t)-f^\ast + \frac{\delta}{2}\norm{x_t - z_t}^2.
\]
Then
\begin{align*}
\dot{E}_t
&= \dotprod{\grad{x_t}, \dot{x}_t} + \delta \dotprod{x_t - z_t,\dot{x}_t - \dot{z}_t}\\
&\overset{(i)}=
\dotprod{\grad{x_t},  \eta(z_t-x_t) - \gamma \grad{x_t)}}
+ \delta \dotprod{x_t - z_t, (\eta+\eta')(z_t-x_t) + (\gamma'-\gamma)\grad{x_t}}\\
&=
\bigl(\eta+\delta(\gamma-\gamma')\bigr)\dotprod{\grad{x_t},  z_t-x_t}
- \gamma \norm{\grad{x_t}}^2
- \delta(\eta + \eta')\norm{z_t-x_t}^2,
\end{align*}
where (i) follows from \eqref{eq:nmo}.
Using $f\in\pl$, we have $-\gamma \norm{\grad{x_t}}^2 \le -2\mu\gamma (f(x_t)-f^\ast)$, hence
\[
\dot E_t \le
\bigl(\eta+\delta(\gamma-\gamma')\bigr)\dotprod{\grad{x_t},  z_t-x_t}
-2\mu\gamma (f(x_t)-f^\ast)
- \delta(\eta + \eta')\norm{z_t-x_t}^2.
\]
Impose $\eta+\delta(\gamma-\gamma')=0$ (\textit{i.e.} $\delta=\eta/(\gamma'-\gamma)$) and $\eta+\eta'=\mu\gamma$.
Then $\dot E_t \le -2\mu\gamma E_t$, so $E_t \le e^{-2\mu\gamma t}E_0$.
Since $f(x_t)-f^\ast \le E_t$, the claim follows.
\end{proof}
\subsubsection{Proof of Proposition~\ref{prop:equiv_hb_nmo}}
The following proof is borrowed from \cite[Proposition 27]{hermant2025continuized}.
 From the first line of
\eqref{eq:nmo}, we have
\begin{equation}\label{eq:z_as_x_constgam}
    z_t = x_t + \frac{1}{\eta_t}\bpar{\dot{x}_t +  \gamma \nabla f(x_t)}.
\end{equation}
Differentiating \eqref{eq:z_as_x_constgam} yields
\begin{align}
    \dot{z}_t
    &= \dot{x}_t - \frac{\dot{\eta}_t}{\eta_t^2}\bpar{\dot{x}_t +  \gamma \nabla f(x_t)}
    + \frac{1}{\eta_t}\bpar{\ddot{x}_t +  \gamma \nabla^2 f(x_t)\dot{x}_t}.
\end{align}
Using this expression and the second line of \eqref{eq:nmo}, we obtain
\begin{align}
     \dot{x}_t - \frac{\dot{\eta}_t}{\eta_t^2}\bpar{\dot{x}_t +  \gamma \nabla f(x_t)}
     + \frac{1}{\eta_t}\bpar{\ddot{x}_t +  \gamma \nabla^2 f(x_t)\dot{x}_t}
     &=  \eta'_t(x_t-z_t) - \gamma' \nabla f(x_t) \\
     &= -\frac{\eta'_t}{\eta_t}\dot{x}_t -  \frac{\eta'_t \gamma}{\eta_t}\nabla f(x_t) - \gamma' \nabla f(x_t),
\end{align}
where we used $x_t-z_t = -\frac{1}{\eta_t}\bpar{\dot{x}_t +  \gamma \nabla f(x_t)}$.
Multiplying by $\eta_t$ and rearranging, we get
\begin{equation}\label{eq:second_order_x_constgam}
    \ddot{x}_t
    + \bpar{\eta_t - \frac{\dot{\eta}_t}{\eta_t} + \eta'_t }\dot{x}_t
    + \gamma \nabla^2 f(x_t)\dot{x}_t
    + \bpar{ \eta'_t \gamma + \eta_t \gamma' -\frac{\dot{\eta}_t}{\eta_t}\gamma }\nabla f(x_t)
    = 0.
\end{equation}

\subsection{Comparison with Theorem 7 of \cite{hermant2024study}}\label{app:related:uaac}
\begin{theorem}[\cite{hermant2024study}]\label{thm:uaac}
    Let $f \in \pl$, with a unique minimizer $x^\ast$. Then, there exists $(\tau,\mu') \in (0,1] \times \mathbb{R}^\ast_+$ such that $f\in \sqcf{\tau}{\mu'}$ if and only if
          \begin{equation}\label{eq:uaac}
                          \exists a > 0,  \forall x \in \mathbb{R}^d, \quad 1 \geqslant \frac{\langle \nabla f(x),x - x^\ast \rangle}{\lVert  \nabla f(x) \rVert \lVert x - x^\ast \rVert} \geqslant a> 0,
          \end{equation}
  In particular if (\ref{eq:uaac}) holds, then as long as $\tau < \frac{2 \mu a}{L}$, $f$ is $(\tau, \frac{\mu a}{\tau} - \frac{L}{2})$-strongly quasar-convex.
\end{theorem}
 The parameters derived in Theorem~\ref{thm:uaac} are significantly different compared to those we deduce from our analysis, see Appendix~\ref{app:link}. Theorem~\ref{thm:uaac} ensures that, at best, we can take $\tau = \bigO(a\frac{\mu}{L})$, inducing $\mu' = \bigO(L)$. We are in the setting where the condition number interpretation fails for strongly quasar-convex functions, see Section~\ref{sec:sqc_pitfalls}. Also, plugging these choices of parameters in Table~\ref{table:sc_sqc_rates} in the line of convergence rate for $\sqc$, if the convergence rates of \eqref{eq:gf} and \eqref{eq:gd} remain the same, namely $a \mu$ and $a \mu /L$ respectively, the convergence rates of \eqref{eq:nmo} and \eqref{eq:nm} however also reduce to $a \mu$ and $a \mu /L$ respectively, such that we do not obtain accelerated bound with Nesterov's momentum.
\subsection{Choice of Parameters for $\sqc$.}\label{app:choice_param_sqc}
In Section~\ref{sec:sqc_pitfalls}, we discuss that $ \sup_\tau  \sup_\mu \{(\tau,\mu) \in (0,1] \times \R_+: f \in \sqc \}$ and $\sup_\mu \sup_\tau \{(\tau,\mu) \in (0,1] \times \R_+: f \in \sqc \}$ lead to different choices of parameters. To justify this statement, we assume there exists $\tau, \mu \in (0,1]\times \R_+$ such that $f \in \sqc$. We first recall
$f \in \sqcf{\tau}{\mu} \Rightarrow f \in \sqcf{\theta \tau}{\mu/\theta}$, for all $\theta \in (0,1]$ \cite[Observation 5]{hinder2020near}. Letting $\theta \to 0$, it follows that
$$\arg \sup_\tau \arg \sup_\mu \{ (\tau,\mu) \in (0,1] \times \R_+:f \in \sqc \} = (0,+\infty).$$
On the other hand, it is clear that
$\arg \sup_\mu \arg \sup_\tau \{ (\tau,\mu)\in (0,1] \times \R_+ : f \in \sqc \} \neq (0,+\infty)$.
\section{More Details on Figures and Numerical Experiments}
In this section, we provide details about our method to generate Figure~\ref{fig:pl_sqc_low_corr}, Figure~\ref{fig:sqc_flaws} (Appendix~\ref{app:param_details}) and Figure~\ref{fig:toy_prob} (Appendix~\ref{app:toy_prob}). We provide in Section~\ref{app:heuristic} a heuristic discussion of the bound over aiming condition to have acceleration derived in Section~\ref{sec:mom_cv}. We also provide in Appendix~\ref{app:conceptual_figures} additional figures, useful for the main text but that could not be included because of space limitations. Our figures and experiments are reproducible; see the code in the supplement.
\subsection{Details about Figure~\ref{fig:pl_sqc_low_corr} and Figure~\ref{fig:sqc_flaws}}\label{app:param_details}
In this section, we detail how we numerically compute the convergence rate of the functions displayed in Figure~\ref{fig:pl_sqc_low_corr} and Figure~\ref{fig:sqc_flaws}. We split our presentation between the two underlying functions.

\textbf{One-dimensional example \:}
The function displayed on the left in Figure~\ref{fig:pl_sqc_low_corr} is
\[f(t) = 5(t+0.19\sin(5t))^2, \]
defined on $[-2,2]$. It has $0$ as its unique critical point and minimizer, which satisfies $f(0) = 0$. We define a grid of $10^{5}$ points in $[-2,2]$, equally spaced. This grid has been refined until the subsequent numerical values are not impacted by further refinement. For the $i$-th point $x_i$ in this grid, we compute the associated point-wise PL value
\[ \mu_i = \frac{f'(x_i)^2}{2f(x_i)},  \]
and point-wise $L$-smooth value
\[ L_i = |f''(x_i)|.\]
We define $\mu := \min_{1\le i\le 10^5} \mu_i$ and $L := \max_{1\le i \le 10^5} L_i$ the constant such that numerically, $f \in \pl \cap \ls$ on $[-2,2]$. On this domain, we thus obtain a convergence rate $\mu/L$ for \eqref{eq:gd}.

Because of the double parameterization, the case of $\sqc$ is slightly more involved. We define a grid of $1000$ potential $\tau$ parameters, equally spaced in $[10^{-5},0.1]$. Note that we take $10^{-5}$ as the lower bound for this interval instead of zero, to avoid dividing by zero. For a given $\tau_j$ in the grid on $[10^{-5},0.1]$ , and each $x_i$ belonging to the grid on $[-2,2]$, we compute the point-wise strong quasar convexity constant $\mu'_{i,j}$
\[ \mu'_{i,j} = -2\frac{\Big( f(x_i) - f'(x_i)x_i \Big)}{x_i^2}.  \]
For this $\tau_j$, we define $\mu'_j = \min_{1\le i \le 10^5} \mu'_{i,j}$ the constant such that numerically, $f \in \sqcf{\tau_j}{\mu'_j}$. In the case where $\mu'_j \le 0$, we consider $\tau_j$ is not admissible, which occurs for $\tau_j \ge 0.1$. The pairs $(\tau_j,\mu'_j)_{1\le j\le 1000}$ thus generate Figure~\ref{fig:sqc_flaws}, by plotting $(\tau_j \mu'_j/L)_{1\le j \le 1000 }$ for \eqref{eq:gd} and $(\tau_j \sqrt{\mu'_j/L})_{1\le j \le 1000 }$ for \eqref{eq:nm}.  As a supplement, we plot in Figure~ $(\tau_j \mu'_j)_{1\le j \le 1000 }$ and $(\tau_j \sqrt{\mu'_j})_{1\le j \le 1000 }$ which are the convergence rates of \eqref{eq:gf} and \eqref{eq:nmo}. We observe a similar behavior compared with the discrete case. Finally, we obtain the convergence rate values of Figure~\ref{fig:pl_sqc_low_corr} by considering $\max_{1\le j \le 1000} \tau_j \mu'_j/L$ for \eqref{eq:gd}, and $\max_{1\le j \le 1000} \tau_j \sqrt{\mu'_j/L}$ for \eqref{eq:nm}. As illustrated in Figure~\ref{fig:sqc_flaws}, the pair $(\tau_{j_0},\mu'_{j_0})$ that maximizes the rate of \eqref{eq:gd} differs from the pair $(\tau_{j_1},\mu'_{j_1})$ that maximizes the rate of \eqref{eq:nm}.

\textbf{Two-dimensional example \:}
The function displayed on the right in Figure~\ref{fig:pl_sqc_low_corr} is
\[f(x,y) = 0.5(0.5x^2 -y)^2 + 0.05x^2,\]
defined on the square $[-1.2638,1.2638]^2$. It has a unique minimizer and critical point at $(0,0)$. To compute the convergence rate for this function, we proceed similarly to the one-dimensional case. The only difference is that we define a two-dimensional grid on $[-1.2638,1.2638]^2$ of $10^6$ points, equally spaced. The grid has been refined until the numerical values we compute are not impacted by further refinement.

\subsection{Details about Figure~\ref{fig:toy_prob}}\label{app:toy_prob}

In Figure~\ref{fig:toy_prob}, we optimize
\[F_{\varepsilon}(x,y) = 0.5(y-\sin(x))^2 + 0.5\varepsilon x^2,\]
with $\varepsilon =10^{-3}$. It is clear that $F_{\varepsilon}(0,0) = 0$ is a global minimizer, and for all $(x,y) \neq (0,0)$, $F_{\varepsilon}(x,y) > 0$. This function has no other critical point, which is induced by the fact that it is PL.

\begin{lemma}\label{lem:toy_example_pl}
    For $x,y\in\R,\ \varepsilon>0$ the function $F_{\varepsilon}(x,y)=\tfrac12\,(y-\sin x)^2+\tfrac12\,\varepsilon x^2$ belongs to $\pl$ with
  $$\mu \le \frac{2+\varepsilon
-\sqrt{\varepsilon^2+4}}{2}.$$

\end{lemma}
\begin{proof}
    Consider the function
\[
F_{\varepsilon}(x,y)=\tfrac12\,(y-\sin x)^2+\tfrac12\,\varepsilon x^2,
\qquad x,y\in\R,\ \varepsilon>0.
\]

Let
\[
u(x,y) := y - \sin x.
\]
Then
\[
F_{\varepsilon}=\tfrac12\,(u^2+\varepsilon x^2).
\]

The gradient is
\[
\nabla F_{\varepsilon}=
\begin{pmatrix}
-u\cos x + \varepsilon x\\[3pt]
u
\end{pmatrix},
\qquad
\|\nabla F_{\varepsilon}\|^2 = u^2 + (-u\cos x + \varepsilon x)^2.
\]

Writing both sides as quadratic forms in \(z=(u(x,y),x)^\top\), we obtain
\[
2F_{\varepsilon}
= z^\top 
\begin{pmatrix}
1 & 0\\[3pt]
0 & \varepsilon
\end{pmatrix}z,
\qquad
\|\nabla F_{\varepsilon}\|^2
= z^\top
\begin{pmatrix}
1+\cos^2 x & -\varepsilon\cos x\\[3pt]
-\varepsilon\cos x & \varepsilon^2
\end{pmatrix}z.
\]

Define
\[
N=\begin{pmatrix}1&0\\0&\varepsilon\end{pmatrix},\qquad
M(x)=\begin{pmatrix}
1+\cos(x)^2 & -\varepsilon \cos(x)\\[3pt]
-\varepsilon \cos(x) & \varepsilon^2
\end{pmatrix},
\qquad c:=\cos x\in[-1,1].
\]
So, $f$ belongs to $\pl$ with
\begin{equation}
    \begin{aligned}
        \mu = \min_{\substack{z=(u(x,y),x)\neq 0,\\(x,y)\in \R^2}}~ \frac{z^T M(x)z}{z^T N z} &\ge \min_{\substack{z=(z_1,z_2),\\(z_1,z_2)\in \R^2\backslash (0,0)}}~ \frac{z^T M(x)z}{z^T N z} \\&\overset{(i)}{=}  \min_{\substack{z=(z_1,z_2),\\(z_1,z_2)\in \R^2\backslash (0,0)}}~ \frac{z^T N^{-1/2}M(x)N^{-1/2}z}{z^Tz} \\
        &\overset{(ii)}{=} \lambda_{\min}(N^{-1/2}M(x)N^{-1/2}),
    \end{aligned}
\end{equation}
where (i) holds with the reparameterization $z' = N^{-1/2}z$, and (ii) is a property of the Rayleigh coefficient. We consider the matrix
\[
N^{-1/2}M(x)N^{-1/2}
=
\begin{pmatrix}
1+\cos(x)^2 & -\sqrt{\varepsilon}\,\cos(x)\\[3pt]
-\sqrt{\varepsilon}\,\cos(x) & \varepsilon
\end{pmatrix}
:= \begin{pmatrix}
1+c^2 & -\sqrt{\varepsilon}\,c\\[3pt]
-\sqrt{\varepsilon}\,c & \varepsilon
\end{pmatrix}=: B(c),\quad c:= \cos(x) \in [-1,1].
\]
So, we have
$$\lambda_{\min}(N^{-1/2}M(x)N^{-1/2}) \ge \inf_{c\in[-1,1]}\ 
\lambda_{\min}(B(c)).$$

The eigenvalues of \(B(c)\) are
\[
\lambda_\pm(c)
=
\frac{\operatorname{tr}(B(c))
\pm
\sqrt{\operatorname{tr}(B(c))^2 - 4\det(B(c))}}{2}.
\]

We have
\[
\det(B(c))=\varepsilon,
\qquad
\operatorname{tr}(B(c)) = 1 + c^2 + \varepsilon.
\]
Such that we solve
$$\min_{c \in [-1,1]} \frac{1+c^2 + \varepsilon - 
\sqrt{(1+c^2 + \varepsilon)^2 - 4\varepsilon}}{2} =\min_{t \in [0,1]} \frac{1+t + \varepsilon - 
\sqrt{(1+t + \varepsilon)^2 - 4\varepsilon}}{2} := \phi(t).$$
We have for $t \in [0,1]$,
$$2\phi'(t) = 1 - \frac{1+t+\varepsilon}{\sqrt{(1+t + \varepsilon)^2 - 4\varepsilon}} \le 0,$$
such that $\phi$ is decreasing on $[0,1]$. So, the minimum is reached for $t = 1 \Rightarrow c^2 = 1$, which leads to the eigenvalue
$$\inf_{c\in[-1,1]}\ 
\lambda_{\min}(B(c)) = \lambda_{\min}(B(1)) = \frac{2+\varepsilon
-\sqrt{\varepsilon^2+4}}{2}
>0.$$

Therefore, for all \((x,y)\in\R^2\),
\[
\|\nabla F_{\varepsilon}(x,y)\|^2 \;\ge\; 2\mu\, F_{\varepsilon}(x,y),
\]
with 
$$\mu \le \frac{2+\varepsilon
-\sqrt{\varepsilon^2+4}}{2}.$$
\end{proof}
The Hessian matrix of $F_{\varepsilon}$ at $(x,y)$ is
\[ \nabla^2  F_{\varepsilon}(x,y) = \begin{pmatrix}
   \cos(x)^2 + \sin(x)(y-\sin(x) +\varepsilon & -\cos(x) \\ -\cos(x) & 1
\end{pmatrix}.\]
The function is not globally $L$-smooth, because of the $y$ factor in $\nabla^2_{1,1} F_{\varepsilon}(x,y)$. However, for a fixed $y$, the eigenvalues are bounded for any $x\in \R$.

\textbf{Implementation Details \:} We fix $\varepsilon = 10^{-3}$, and chose $(0,3)$ as the initialization point. We compute the $L$-smooth value of $F_{0.001}$ on the domain $[-2\pi,2\pi] \times [-3,3]$, using a grid of $10^6$ equally spaced points. We then fix $\gamma = 1/L$ for both \eqref{eq:gd} and \eqref{eq:nm_prime}.

\subsection{Proof of Proposition~\ref{prop:hard_fonc_sqc}}\label{app:hard_fonc_sqc}
 The definition of the hard instance is rather involved. The core object of its definition is the following function
\[ g_{T,t}(x) = q_{T,t}(x) + \sum_{i=1}^{Tt}v_{y_i}(x_i), \]
where $g_{T,t}:\R^{Tt}\to\R$ and $x_i$ is the $i$-th coordinate of $x\in\R^{Tt}$; here $q_{T,t}(x)$ is defined as follows
\[q_{T,t}(x) = \frac{1}{2}\sum_{i=0}^{t-1}\left[ \Big(\frac{7}{8}x_{iT}-x_{iT+1}\Big)^2 + \sum_{j=1}^{T-1}(x_{iT+j+1}-x_{iT+j})^2 \right],\]
and for all $y\in \R$, $v_y(t)$ is the following one-dimensional function
\begin{equation*}
     v_y(t) = \left\{ \begin{array}{ll}
     \frac{1}{2}t,& t\le \frac{31}{32}y,  \\
     \frac{1}{2}t^2 - 16(t-\frac{31}{32}y)^2,& \frac{31}{32}y < t \le y,\\
     \frac{1}{2}t^2 - \frac{y^2}{32} + 16(t-\frac{33}{32}y)^2, &y<t\le \frac{33}{32}y,\\
     \frac{1}{2}t^2 - \frac{y^2}{32},&t>\frac{33}{32}y.\end{array} \right.  
\end{equation*}
The function $q_{T,t}$ belongs to the class of zero-chain functions \cite{nesterov2004introductory,lowerboundI,carmon2021lower}, which are hard to minimize with first-order algorithms. The key idea is that for these functions, if starting from $0$, first-order algorithms will be able to explore only one new coordinate at each iteration for a number of iterations equal to the dimension of the problem. Here, the $v_y(\cdot )$ components introduce non-convexity. Up to a rescaling, the function $g_{T,t}$ belongs to $\pl \cap \ls$ for desired constants $\mu$ and $L$, see all details in \cite{PLlowerbound}.

\textit{Proof of Proposition~\ref{prop:hard_fonc_sqc}.}
 The function $q_{T,t}$ is a strongly convex quadratic function with unique minimizer $0$, such that it belongs to $\sqc$ for some parameters. The $v_y(\cdot )$ components introduce non-convexity.  For any $y\in \R$, the function $v_y(\cdot)$ has zero as its unique minimizer, and belongs to $\pl \cap \ls$ for some parameter, which implies it belongs to $\sqc \cap \ls $ because the function is one-dimensional. It follows that as a sum of one-dimensional functions in $\sqc \cap \ls$, the function $ \sum_{i=1}^{Tt}v_{y_i}(x_i)$ belongs to $\sqc \cap \ls$ for some parameters, with unique minimizer $0 \in \R^d$. Then, because sums of strongly quasar-convex functions with the same minimizer remain strongly quasar-convex, $g_{t,T}$ is strongly quasar convex \cite[Appendix D.3]{hinder2020near}. Finally, to make the function $g_{T,t}$ belong to $\pl \cap \ls$ for desired constants $\mu$ and $L$, it remains to rescale it as $\tilde g(x) = a_0g_{t,T}(y-a_1x)$ for some constants $a_0,a_1 \in \R$ and where $y = (y_1,\dots,y_{Tt})$, see all details in \cite{PLlowerbound}. Such rescaling preserves the strong quasar convexity property, see \cite[Appendix D.3]{hinder2020near}. 

\subsection{Supplementary Figures}
In this section, we provide some figures that are not in the main text, due to available space. 

\textbf{About Figure~\ref{fig:figure_sqc} \:}\label{app:conceptual_figures} We display in Figure~\ref{fig:figure_sqc} two synthetic examples of functions belonging to $\sqc$, in dimension one and two. The one-dimensional example is
\[f(t) =  0.5\cdot(t +0.15\sin(5t))^2.\]
The two dimensional example is $h: \mathbb{R}^2 \to \mathbb{R}$, built as follows:
  \begin{equation}
    h(x) = f(\lVert x \rVert)g\left(\frac{x}{\lVert x \rVert} \right)
\end{equation}
where $f(t) = t^2$ and 
\begin{equation}
    g(x_1,x_2) = \frac{1}{4N} \sum_{i=1}^N \left(a_i \sin(b_i x_1)^2 + c_i \cos(d_i x_2)^2 \right) + 1
\end{equation}
with $N = 10$ and the $\{ a_i \}_i$, $\{ c_i \}_i$ are independently and uniformly distributed on $[0,20]$, and the $\{ b_i \}_i$, $\{ d_i \}_i$ are independently and uniformly distributed on $[-25,25]$. It is borrowed from \cite{hermant2024study}.
\begin{figure}[ht]
    \centering
    \begin{minipage}[b]{0.4\textwidth}
        \centering
        \includegraphics[width=\textwidth]{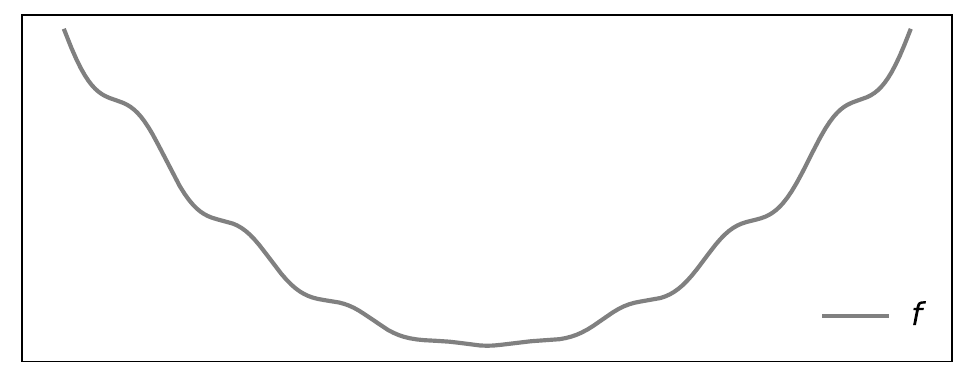}\\[1ex]
        \includegraphics[width=\textwidth]{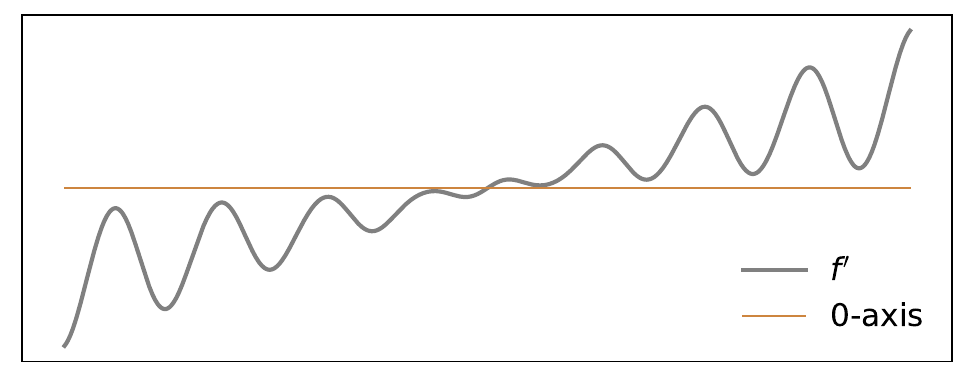}
    \end{minipage}%
    \begin{minipage}[b]{0.38\textwidth}
        \centering
        \includegraphics[width=\textwidth]{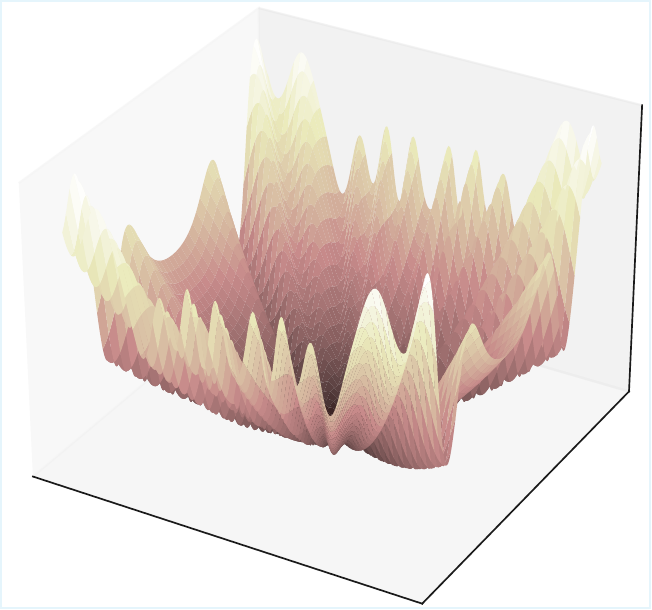}
    \end{minipage}
    \caption{Visualization of functions belonging to $\sqc$, for some parameters. The figure is borrowed from \cite{hermant2025continuized}, see details in Appendix~\ref{app:conceptual_figures}.}
    \label{fig:figure_sqc}
\end{figure}

\textbf{About Figure~\ref{fig:quad_bounds} \:} Figure~\ref{fig:quad_bounds} displays a simple one-dimensional example where $f \in \ls \cap \qg$ with $L$ significantly larger than $L_0$, and similarly for $f \in \pl \cap \qginf$. Such differences are created by local variations outside of the vicinity of the minimizer. This is because the parameterizing $\qg$ and $\qginf$ is mainly affected by the geometry in the vicinity of the minimizer, while $\ls$ and $\pl$ are strongly sensitive to variations at any point.

\textbf{About Figure~\ref{fig:worst_case} \:} Figure~\ref{fig:worst_case} illustrates that, for a function belonging to $\pl$, the optimal $\mu$ may depend on the local smaller curvature at a minimizer $x^\ast$, or on plateaus. 
\begin{figure}
    \centering
    \includegraphics[width=0.5\linewidth]{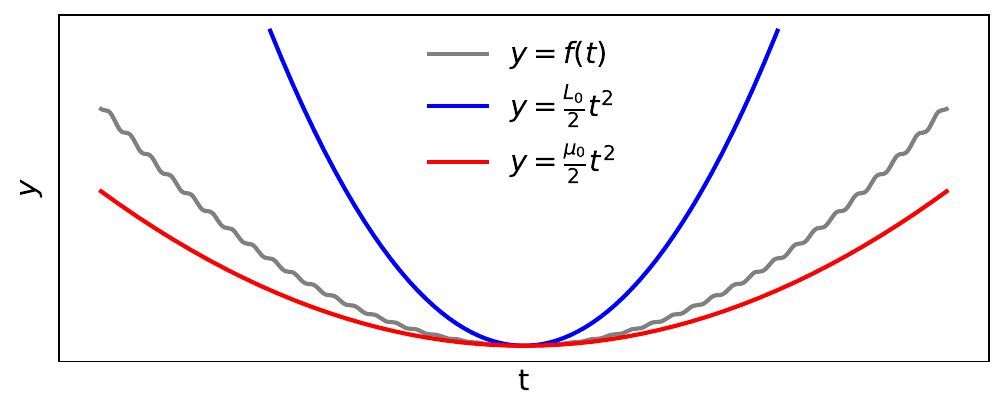}
    \caption{ The grey curve is $f(t) = 0.5 \cdot  5(t+0.07\sin(13t))^2$. The blue (red) curve is a quadratic upper (lower) bound. We compute the largest $\mu$ and smallest $L$ such that  $f \in \pl \cap \ls$ on the displayed domain, numerically evaluated at $\mu \approx 4\cdot10^{-2}$, $L \approx 6\cdot 10^2$. Also, the parameters $L_0$ and $\mu_0$ that parameterize the quadratic bounds are $\mu_0 \approx 3$ and $L_0 \approx 18$. On this prototype 1-dimensional example, $\frac{\mu_0}{L_0} \approx 0.2$ while $\frac{\mu}{L} \approx 7\cdot 10^{-5}$. This highlights the possibility of a significant gap between these two ratios.}
    \label{fig:quad_bounds}
\end{figure}
\begin{figure}[t!]
        \centering
  \includegraphics[width=0.4\linewidth]{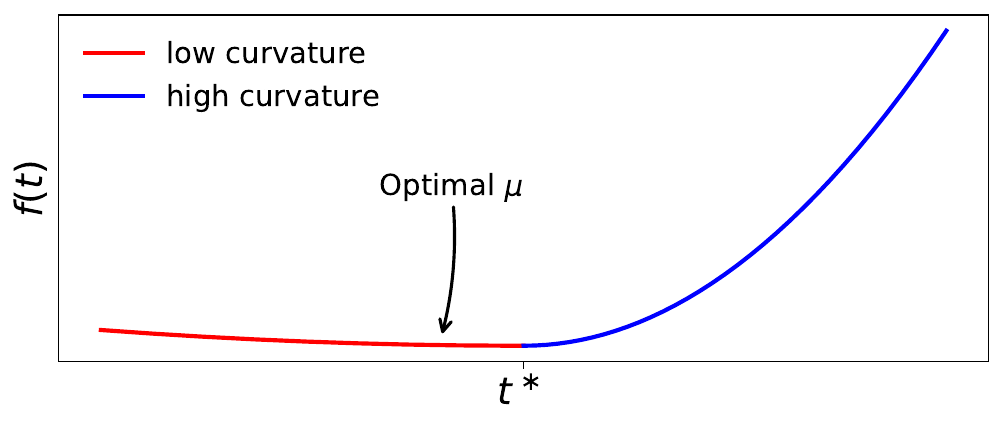}
    \includegraphics[width=0.4\linewidth]{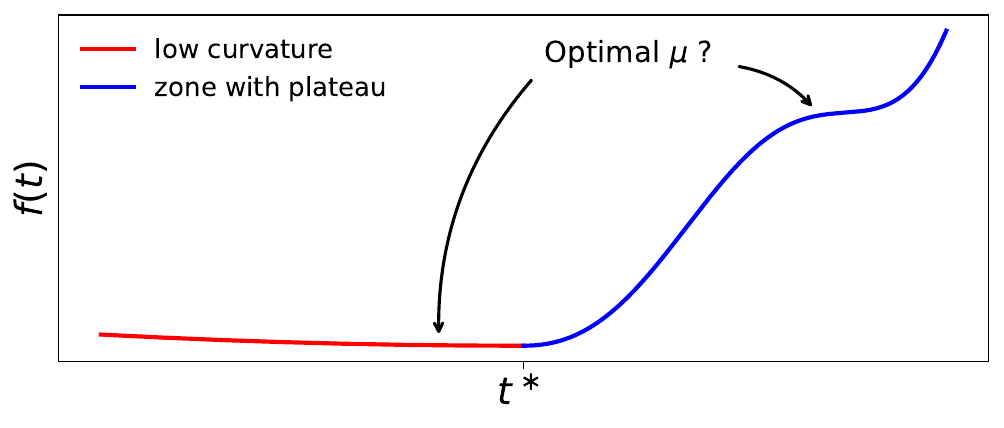}

    \caption{Left: Simple $1$-dimensional functions in $\strongconv$, so also in $\pl$. The worst approaching direction toward the minimizer $t^\ast$ is the red zone, which determines the optimal $\mu$ constant. Right: Simple non-convex function that belongs to $\pl$ for some $\mu > 0$. The $\mu$ constant could be determined either by the local flatness around the minimizer $t^\ast$ (red zone), or by the flatness of the plateau (blue zone).} 
    \label{fig:worst_case}
\end{figure}
\begin{figure}
    \centering
    \includegraphics[width=0.4\linewidth]{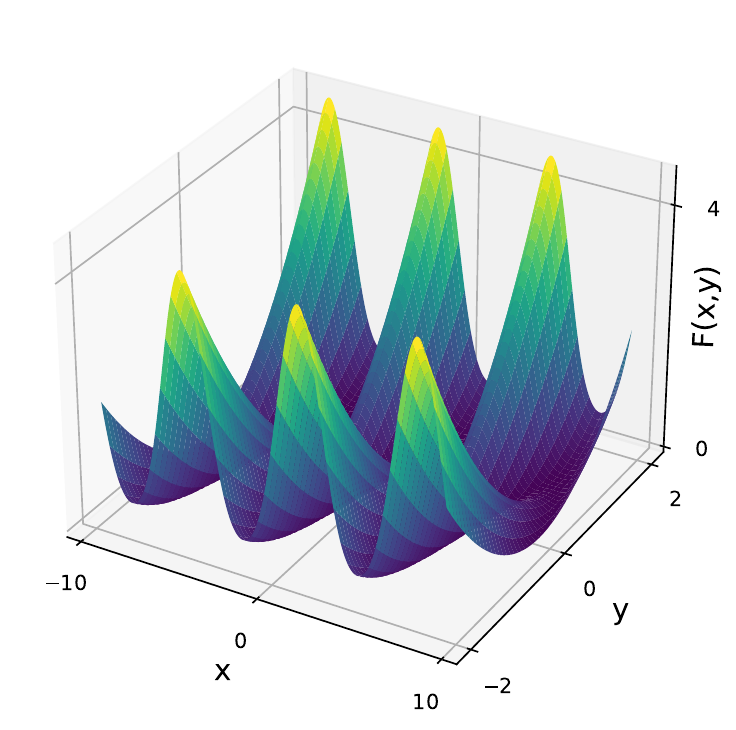}
    \caption{ $3$-dimensional visualization of $F_{0.001}(x,y)= 0.5(y-\sin(x))^2 + 0.001\cdot0.5x^2$, considered in Section~\ref{sec:numerical}}
    \label{fig:3d_vizualization}
\end{figure}

\begin{figure}
    \centering

        \begin{minipage}{0.48\linewidth}
            \centering
            \includegraphics[width=\linewidth]{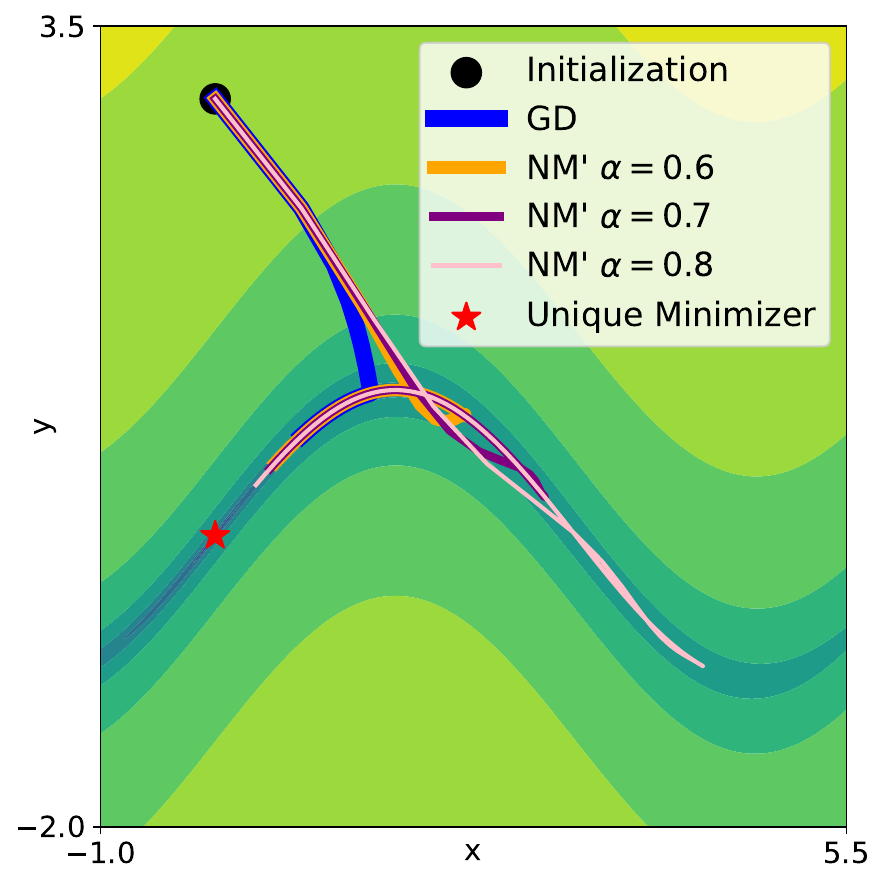}
        \end{minipage}
    \hfill
    \begin{minipage}{0.5\linewidth}
        \centering
        \includegraphics[width=\linewidth]{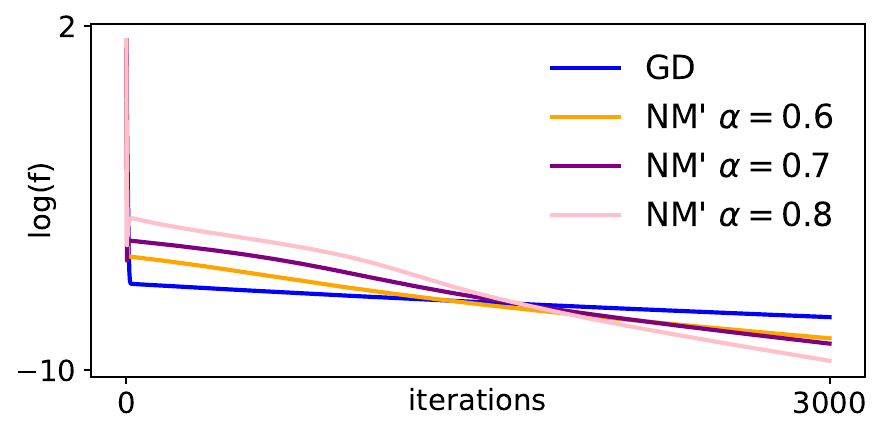}

        \vspace{0.3cm}

        \includegraphics[width=\linewidth]{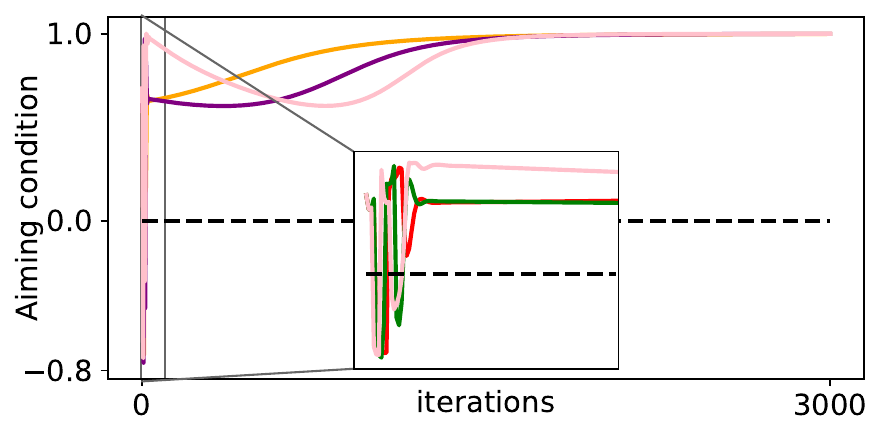}
    \end{minipage}
    \caption{ Plot of the first $3000$ iterations of the trajectories of \eqref{eq:gd} and \eqref{eq:nm_prime} on $f$, when minimizing $F_{0.001}(x,y) = 0.5(y-\sin(x))^2 + 0.001\cdot0.5x^2$. This is a version of Figure~\ref{fig:toy_prob} with $3000$ iterations instead of $1000$. Observe that the decrease of the momentum method \eqref{eq:nm_prime} catches up with the one of \eqref{eq:gd}, which is coherent with the values of the aiming condition.
     }
    \label{fig:convergence_longer}
\end{figure}
\begin{figure}
    \centering
    \includegraphics[width=0.8\linewidth]{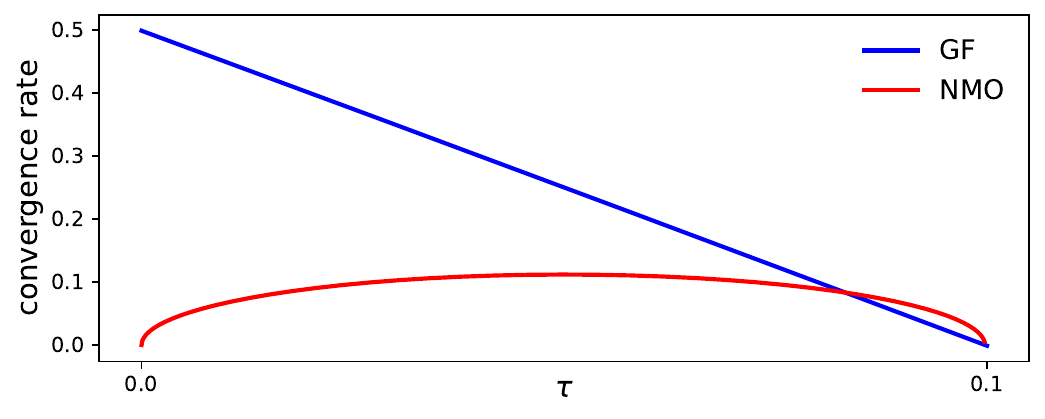}
    \caption{For a range of parameters $\tau$, we compute the highest admissible $\mu$ ensuring $f \in \sqc $, and plot the associated theoretical convergence rates of \eqref{eq:gf} and \eqref{eq:nmo} under $\sqc$, namely $\tau \mu$ and $\tau \sqrt{\mu}$, with $f(t) = 5(t+0.19\sin(5t))^2$. See details in Appendix~\ref{app:param_details}. This is the continuous counterpart of Figure~\ref{fig:sqc_flaws}.
Similarly to the discrete setting, the pair $(\tau, \mu)$ that maximizes the convergence rate differs in both cases.}
    \label{fig:bound_cont}
\end{figure}

\section{Technical Lemmas}
The following lemmas serve as technical tools for our convergence proofs.

\begin{lemma}\label{lem:prod}
Let $f \in \pl \cap \qg$. Noting $
\mu_0 := \sup \{\mu'\ge \mu: f \in  \qginff{\mu'} \}$, 
we have for all $x\in\R^d$,
\[
\norm{x-x^\ast}\norm{\nabla f(x)}
\ge \sqrt{\frac{\mu}{L_0}}(f(x)-f^\ast) + \frac{\sqrt{\mu \mu_0}}{2}\norm{x-x^\ast}^2.
\]
\end{lemma}

\begin{proof}
If $x = x^\ast$, the statement is trivial. Assume $x \ne x^\ast$ and write
\[
\norm{x-x^\ast}\norm{\nabla f(x)}
= \norm{x-x^\ast}\sqrt{f(x)-f^\ast}\,\frac{\norm{\nabla f(x)}}{\sqrt{f(x)-f^\ast}}.
\]

\medskip
\noindent\textbf{(1) Step 1 ($\pl$ and $\qg$).}
Since $f\in\pl$, $\norm{\nabla f(x)}\ge \sqrt{2\mu}\sqrt{f(x)-f^\ast}$, and since $f\in\qg$,
\[
f(x)-f^\ast \le \frac{L_0}{2}\norm{x-x^\ast}^2
\quad\Rightarrow\quad
\norm{x-x^\ast}\ge \sqrt{\frac{2}{L_0}}\sqrt{f(x)-f^\ast}.
\]
Therefore,
\begin{equation}\label{eq:prod_0_corr}
\norm{x-x^\ast}\norm{\nabla f(x)}
\ge \sqrt{\frac{2}{L_0}}(f(x)-f^\ast)\sqrt{2\mu}
= 2\sqrt{\frac{\mu}{L_0}}(f(x)-f^\ast).
\end{equation}

\medskip
\noindent\textbf{Step 2 ($\pl$ and $\qginff{\mu_0}$).}
By definition of $\mu_0$
\[
f(x)-f^\ast \ge \frac{\mu_0}{2}\norm{x-x^\ast}^2
\quad\Rightarrow\quad
\sqrt{f(x)-f^\ast}\ge \sqrt{\frac{\mu_0}{2}}\norm{x-x^\ast}.
\]
Combining with $\norm{\nabla f(x)}\ge \sqrt{2\mu}\sqrt{f(x)-f^\ast}$ gives
\begin{equation}\label{eq:prod_1_corr}
\norm{x-x^\ast}\norm{\nabla f(x)}
\ge \sqrt{\mu\mu_0}\,\norm{x-x^\ast}^2.
\end{equation}
\medskip
\noindent\textbf{Combine
 the bounds}
From \eqref{eq:prod_0_corr} and the previous bound,
\[
\norm{x-x^\ast}\norm{\nabla f(x)}
\ge \max\Bigl\{2\sqrt{\frac{\mu}{L_0}}(f(x)-f^\ast),\ \sqrt{\mu\mu_0}\,\norm{x-x^\ast}^2\Bigr\}.
\]
Using $\max\{u,v\}\ge \frac{u+v}{2}$ for $u,v\ge 0$, we obtain
\[
\norm{x-x^\ast}\norm{\nabla f(x)}
\ge \sqrt{\frac{\mu}{L_0}}(f(x)-f^\ast) + \frac{\sqrt{\mu\mu_0}}{2}\norm{x-x^\ast}^2,
\]
which concludes the proof.
\end{proof}

    

\begin{lemma}\label{lem:tec:exact_scalar_prod_bound}
Assume $x\neq x^\ast$ and $f(x)\neq f^\ast$. Define the point-wise quantities
\[
\mu(x) = \frac{1}{2}\frac{\norm{\nabla f(x)}^2}{f(x)-f^\ast},\qquad
a(x) =\frac{\dotprod{\nabla f(x),x-x^\ast}}{\norm{\nabla f(x)}\norm{x-x^\ast}},\qquad
\mu_0(x) = 2\frac{f(x)-f^\ast}{\norm{x-x^\ast}^2}.
\]
Then,
\[
\dotprod{\nabla f(x), x^\ast - x}
=  -a(x) \sqrt{\frac{\mu(x)}{\mu_0(x)}}(f(x)-f^\ast)
   - \frac{a(x)}{2}\sqrt{\mu(x)\mu_0(x)}\,\norm{x-x^\ast}^2.
\]

If $f \in \pl$, it follows
\[
\dotprod{\nabla f(x), x^\ast - x}
=  -a(x) \sqrt{\frac{\mu}{\mu_0}}(f(x)-f^\ast)
   - \frac{a(x)}{2}\sqrt{\mu\mu_0}\,\norm{x-x^\ast}^2,
\]
with $
\mu_0 := \sup \{\mu'\ge \mu: f \in  \qginff{\mu'} \}$.
\end{lemma}

\begin{proof}

By definition of $a(x)$,
\[
\dotprod{\nabla f(x), x^\ast - x} = -a(x)\norm{\nabla f(x)}\norm{x-x^\ast}.
\]
Moreover, by the definitions of $\mu(x)$ and $\mu_0(x)$,
\[
\norm{\nabla f(x)} = \sqrt{2\mu(x)}\sqrt{f(x)-f^\ast},
\qquad
\norm{x-x^\ast} = \sqrt{\frac{2}{\mu_0(x)}}\sqrt{f(x)-f^\ast}.
\]
Hence
\[
\norm{\nabla f(x)}\norm{x-x^\ast}
= 2\sqrt{\frac{\mu(x)}{\mu_0(x)}}\,(f(x)-f^\ast).
\]
Also, using $\norm{\nabla f(x)} = \sqrt{2\mu(x)}\sqrt{f(x)-f^\ast} =\sqrt{2\mu(x)} \sqrt{\frac{\mu_0(x)}{2}}\norm{x-x^\ast}$,
\[
\norm{\nabla f(x)}\norm{x-x^\ast}
= \sqrt{\mu(x)\mu_0(x)}\,\norm{x-x^\ast}^2.
\]
Combining the two equal expressions and using $\xi = \frac12\xi+\frac12\xi$ with
$\xi= \norm{\nabla f(x)}\norm{x-x^\ast}$ gives
\[
\norm{\nabla f(x)}\norm{x-x^\ast}
=
\sqrt{\frac{\mu(x)}{\mu_0(x)}}\,(f(x)-f^\ast)
+\sqrt{\mu(x)\mu_0(x)}\,\frac{1}{2}\norm{x-x^\ast}^2,
\]
and multiplying by $-a(x)$ yields the claim.
\end{proof}

\section{Proofs of Section~\ref{sec:cv_results}}\label{app:cv_results}
This section is dedicated to proving the results of Section~\ref{sec:cv_results}. In the discrete setting, we will actually prove a more general result, replacing the gradient with a stochastic estimator, under a strong growth condition, also sometimes named multiplicative noise assumption.\begin{assumption}[Strong growth condition]\label{ass:sgc}
    We assume we have access to stochastic approximations $\nabla f(x,\xi)$ of the real gradient $\nabla f(x)$, where $\xi \sim \mathcal{P}(\Xi)$, such that
    \begin{enumerate}[label=(\roman*)]
        \item (Unbiased estimator) $\forall x \in \R^d, ~ \E{\nabla f(x,\xi)} = \nabla f(x)$. 
        \item (Strong Growth Condition) $\exists \rho \ge 1, \forall x \in \R^d,~ \mathbb{E}_{\xi}\left[ \norm{\nabla f(x,\xi)}^2 \right] \le \rho \norm{\nabla f(x)}^2$. 
    \end{enumerate} 
\end{assumption}
To our knowledge, this assumption was introduced in \cite{polyak1987introduction}, and further used by \cite{cevher2019linear} and \cite{vaswani2019fast} as a relaxation of the maximal strong growth condition \cite{tseng1998incremental,solodov1998incremental}.
 We avoid this more general setting in the main text, to avoid adding a layer of complexity. The deterministic case is a special instance of Assumption~\ref{ass:sgc}, corresponding to $\rho = 1$.
For functions in $\ls$ that verify Assumption~\ref{ass:sgc}, we have a stochastic descent lemma, see \cite[Lemma 16]{gupta2024nesterov}.
\begin{lemma}\label{lem:descent_sgc}
If $f \in \ls$ and Assumption~\ref{ass:sgc} holds, then we have

\begin{equation}
     \mathbb{E}_\xi\left[ f\Big(x-\frac{1}{\rho L}\nabla f(x,\xi)\Big) - f(x) \right] \leq -\frac{1}{2\rho L}\lVert  \nabla f(x) \rVert^2.
\end{equation}
\end{lemma} 
In the subsequent sections, for an algorithm $\xdisc$ we denote $\condE{\cdot} = \E{\cdot \: | \: \mathcal{F}_k }$. For a sequence $\{\xi_k \}_{k\in \N}$ of random variables, we note $\mathcal{F}_k = \sigma(\xi_0,\dots,\xi_k)$ the filtration generated by $\xi_0,\dots,\xi_k$.
\subsection{Proof of Theorem~\ref{thm:gf_pl_&_ac}}\label{app:gf_pl_&_ac}

\textbf{(i) $f \in \pl$ \:} The following proof is classical \cite{polyak2017lyapunov}. We provide a proof for the sake of completeness.
Let
\[
E_t = f(x_t)-f^\ast.
\]
By definition of \eqref{eq:gf}, we have
\[
\dot E_t = \dotprod{\nabla f(x_t),\dot x_t}
= -\norm{\nabla f(x_t)}^2
\le -2\mu\,(f(x_t)-f^\ast)
= -2\mu\,E_t,
\]
where we used $f\in\pl$. By integration,
\[
f(x_t)-f^\ast \le e^{-2\mu t}(f(x_0)-f^\ast).
\]

\textbf{(ii) $f \in \pl \cap \ac \cap \qg$ \:}
Let
\[
E_t = f(x_t)-f^\ast + \frac{\beta}{2}\norm{x_t-x^\ast}^2.
\]
Then
\begin{align*}
\dot E_t
&= \dotprod{\nabla f(x_t),\dot x_t} + \beta \dotprod{x_t-x^\ast,\dot x_t}\\
&= -\norm{\nabla f(x_t)}^2 - \beta \dotprod{\nabla f(x_t),x_t-x^\ast}\\
&\le - \beta \dotprod{\nabla f(x_t),x_t-x^\ast}
= \beta \dotprod{\nabla f(x_t),x^\ast-x_t}.
\end{align*}
By Lemma~\ref{lem:scalar_prod_bound}, for all $t\ge 0$,
\[
\dotprod{\nabla f(x_t),x^\ast-x_t}
\le -a\sqrt{\frac{\mu}{L_0}}\,(f(x_t)-f^\ast)
     -\frac{a}{2}\sqrt{\mu\mu_0}\,\norm{x_t-x^\ast}^2,
\]
hence
\[
\dot E_t \le
-\beta a\sqrt{\frac{\mu}{L_0}}\,(f(x_t)-f^\ast)
-\frac{\beta a}{2}\sqrt{\mu\mu_0}\,\norm{x_t-x^\ast}^2.
\]
Choose $\beta$ so that the two coefficients match:
\[
\beta a\sqrt{\frac{\mu}{L_0}} = a\sqrt{\mu\mu_0}
\quad\Longrightarrow\quad
\beta = \sqrt{\mu_0 L_0}.
\]
Then
\[
\dot E_t \le -a\sqrt{\mu\mu_0}\Bigl(f(x_t)-f^\ast + \frac{\beta}{2}\norm{x_t-x^\ast}^2\Bigr)
= -a\sqrt{\mu\mu_0}\,E_t,
\]
and by integration,
\[
E_t \le e^{-a\sqrt{\mu\mu_0}\,t}E_0.
\]
Since $f\in\qginff{\mu_0}$, we have
\[
\frac{\beta}{2}\norm{x_0-x^\ast}^2
\le \beta\,\frac{1}{\mu_0}\,(f(x_0)-f^\ast)
= \sqrt{\frac{L_0}{\mu_0}}\,(f(x_0)-f^\ast),
\]
so $E_0 \le \bpar{1+\sqrt{\frac{L_0}{\mu_0}}}\,(f(x_0)-f^\ast)$.
Finally, since $f(x_t)-f^\ast\le E_t$, we deduce
\[
f(x_t)-f^\ast
\le \bpar{1+\sqrt{\frac{L_0}{\mu_0}}}\,e^{-a\sqrt{\mu\mu_0}\,t}\,(f(x_0)-f^\ast).
\]

\subsection{Proof of Theorem~\ref{thm:gd_pl_&_ac}}\label{app:gd_pl_ac}
We prove the following more general result.

\begin{theorem}\label{thm:sto_gd_pl_&_ac}
Let $f \in \ls$. Let $\xdisctilde \sim \eqref{eq:gd}$ with $\nabla f(\tilde x_k)$ replaced by
$\nabla f(\tilde x_k,\xi_k)$, where $\xi_k$ is independent of $\tilde x_k$, and assume Assumption~\ref{ass:sgc}.
Let $\gamma = \frac{1}{\rho L}$.

(i) If $f \in \pl$, then
\[
\E{ f(\tilde x_k)-f^\ast} \le \bpar{1-\frac{\mu}{\rho L}}^k\,(f(\tilde x_0)-f^\ast).
\]

(ii) If $f \in \pl \cap \ac \cap \qg$, then
\[
\E{f(\tilde x_k)-f^\ast}
\le
\bpar{1+\sqrt{\frac{L_0}{\mu_0}}}\,
\bpar{1-\frac{a\sqrt{\mu \mu_0}}{\rho L}}^k\,(f(\tilde x_0)-f^\ast),
\]
with $\mu_0 := \sup\{\mu' \ge \mu: f \in \qginff{\mu'}\}$ and
$L_0 :=\inf\{L' \le  L:f \in \qgf{L'}\}$.
\end{theorem}

The result for Theorem~\ref{thm:gd_pl_&_ac} is deduced in the case $\rho =1$, which removes the expectations. 

\begin{proof}[Proof of Theorem~\ref{thm:sto_gd_pl_&_ac}]
\textbf{(i) $f \in \pl$ \:} This result and proof is well known, see, for instance, \cite{karimi2016linear} in the deterministic setting, \cite{vaswani2019fast} for the stochastic case. We provide a proof for completeness.
Define \[E_k = f(\tilde x_k)-f^\ast.\] By Lemma~\ref{lem:descent_sgc},
\[
\condE{E_{k+1}-E_k}
=\condE{f(\tilde x_{k+1})-f(\tilde x_k)}
\le -\frac{1}{2\rho L}\norm{\nabla f(\tilde x_k)}^2.
\]
Since $f\in\pl$, $\norm{\nabla f(\tilde x_k)}^2 \ge 2\mu\,(f(\tilde x_k)-f^\ast)=2\mu E_k$, hence
\[
\condE{E_{k+1}} \le \bpar{1-\frac{\mu}{\rho L}}E_k.
\]
By induction and taking expectation,
\[
\E{E_k}\le \bpar{1-\frac{\mu}{\rho L}}^k E_0,
\]
which is the claim.

\medskip
\textbf{(ii) $f \in \pl \cap \ac \cap \qg$ \:}
Set $\beta := \sqrt{\mu_0 L_0}$ and define
\[
E_k =f(\tilde x_k)-f^\ast + \frac{\beta}{2}\norm{\tilde x_k-x^\ast}^2.
\]
Using the update
$\tilde x_{k+1}=\tilde x_k-\gamma\nabla f(\tilde x_k,\xi_k)$, we expand
\[
\norm{\tilde x_{k+1}-x^\ast}^2-\norm{\tilde x_k-x^\ast}^2
= -2\gamma \dotprod{\tilde x_k-x^\ast,\nabla f(\tilde x_k,\xi_k)}
+ \gamma^2\norm{\nabla f(\tilde x_k,\xi_k)}^2.
\]
Taking $\condE{\cdot}$, using Assumption~\ref{ass:sgc} (unbiasedness + SGC) and Lemma~\ref{lem:descent_sgc}, we obtain
\begin{align*}
\condE{E_{k+1}-E_k}
&\le -\frac{1}{2\rho L}\norm{\nabla f(\tilde x_k)}^2
+ \frac{\beta}{2}\Bigl(
-2\gamma \dotprod{\tilde x_k-x^\ast,\nabla f(\tilde x_k)}
+ \gamma^2\rho \norm{\nabla f(\tilde x_k)}^2
\Bigr)\\
&= \Bigl(-\frac{1}{2\rho L}+\frac{\beta}{2}\gamma^2\rho\Bigr)\norm{\nabla f(\tilde x_k)}^2
-\beta\gamma \dotprod{\tilde x_k-x^\ast,\nabla f(\tilde x_k)}.
\end{align*}
With $\gamma=\frac{1}{\rho L}$, we have $\gamma^2\rho=\frac{1}{\rho L^2}$, hence
\[
-\frac{1}{2\rho L}+\frac{\beta}{2}\gamma^2\rho
= -\frac{1}{2\rho L}+\frac{\beta}{2\rho L^2}
= -\frac{1}{2\rho L}\bpar{1-\frac{\beta}{L}} \le 0,
\]
because $\beta=\sqrt{\mu_0L_0}\le L_0\le L$. Therefore,
\[
\condE{E_{k+1}-E_k}
\le -\frac{\beta}{\rho L}\dotprod{\tilde x_k-x^\ast,\nabla f(\tilde x_k)}
= \frac{\beta}{\rho L}\dotprod{\nabla f(\tilde x_k),x^\ast-\tilde x_k}.
\]
By Lemma~\ref{lem:scalar_prod_bound},
\[
\dotprod{\nabla f(\tilde x_k),x^\ast-\tilde x_k}
\le -a\sqrt{\frac{\mu}{L_0}}\,(f(\tilde x_k)-f^\ast)
-\frac{a}{2}\sqrt{\mu\mu_0}\,\norm{\tilde x_k-x^\ast}^2.
\]
Multiplying by $\beta/(\rho L)$ and using $\beta\sqrt{\mu/L_0}=\sqrt{\mu\mu_0}$ (since $\beta=\sqrt{\mu_0L_0}$), we get
\[
\condE{E_{k+1}-E_k}
\le -\frac{a\sqrt{\mu\mu_0}}{\rho L}\,(f(\tilde x_k)-f^\ast)
-\frac{a\sqrt{\mu\mu_0}}{\rho L}\,\frac{\beta}{2}\norm{\tilde x_k-x^\ast}^2
= -\frac{a\sqrt{\mu\mu_0}}{\rho L}E_k.
\]
Hence
\[
\condE{E_{k+1}}\le \bpar{1-\frac{a\sqrt{\mu\mu_0}}{\rho L}}E_k,
\]
taking the expectation and by induction,
\[
\E{E_k}\le \bpar{1-\frac{a\sqrt{\mu\mu_0}}{\rho L}}^k E_0.
\]
Finally, since $f(\tilde x_k)-f^\ast\le E_k$, and using $f\in\qginff{\mu_0}$,
\[
\frac{\beta}{2}\norm{\tilde x_0-x^\ast}^2
\le \beta\,\frac{1}{\mu_0}(f(\tilde x_0)-f^\ast)
= \sqrt{\frac{L_0}{\mu_0}}(f(\tilde x_0)-f^\ast),
\]
so $E_0\le \bpar{1+\sqrt{\frac{L_0}{\mu_0}}}(f(\tilde x_0)-f^\ast)$. This yields the claim.
\end{proof}

\subsection{Proof of Theorem~\ref{thm:pl_accel}}\label{app:pl_accel}

\textbf{(i) \:}
Assume that $\eta_t\equiv\eta$, $\eta'_t\equiv\eta'$ are constants.
Let
\begin{equation}\label{eq:E_pl_accel}
    E_t = e^{\theta t}\Bigl(f(x_t)-f^\ast +\beta \frac{1}{2}\norm{z_t -  x^\ast}^2\Bigr).
\end{equation}
Then $E_t = E_0 + \int_0^t \dot{E}_s\,ds$, where
\begin{align}
    \dot{E}_t
    &= \theta  e^{\theta t}(f(x_t)-f^\ast)  +\theta \beta \frac{ e^{\theta t}}{2}\norm{z_t -  x^\ast}^2
    + e^{\theta t}\dotprod{\nabla f(x_t),\dot{x}_t} + \beta  e^{\theta t}\dotprod{z_t -  x^\ast,\dot{z}_t}
   \nonumber \\
    &= \theta  e^{\theta t}(f(x_t)-f^\ast)  +\theta \beta \frac{ e^{\theta t}}{2}\norm{z_t -  x^\ast}^2
    + e^{\theta t}\dotprod{\nabla f(x_t), \eta(z_t-x_t) - \gamma \nabla f(x_t)}\nonumber\\
    &\hspace{2cm}+ \beta   e^{\theta t}\dotprod{z_t - x^\ast,\eta'(x_t-z_t) - \gamma' \nabla f(x_t) } \nonumber\\
    &= \theta  e^{\theta t}(f(x_t)-f^\ast)  +\theta \beta \frac{ e^{\theta t}}{2}\norm{z_t -  x^\ast}^2
    - \gamma e^{\theta t}\norm{\nabla f(x_t)}^2 \nonumber\\
    &\hspace{2cm}+ e^{\theta t}\eta \dotprod{\nabla f(x_t), z_t-x_t}
    + \beta\eta' e^{\theta t}\dotprod{z_t - x^\ast,x_t - z_t}
    - \beta\gamma' e^{\theta t}\dotprod{\nabla f(x_t),z_t-x^\ast}.\label{eq:nest_ode_0_rew}
\end{align}
Using the decomposition
\[
\dotprod{\nabla f(x_t), z_t-x_t}
=
\dotprod{\nabla f(x_t), x^\ast-x_t}+\dotprod{\nabla f(x_t), z_t-x^\ast},
\]
and the identity
\[
\dotprod{z_t-x^\ast,x_t - z_t}
= \frac{1}{2}\norm{x_t-x^\ast}^2 -\frac{1}{2}\norm{z_t-x^\ast}^2-\frac{1}{2}\norm{x_t-z_t}^2
\le \frac{1}{2}\norm{x_t-x^\ast}^2 -\frac{1}{2}\norm{z_t-x^\ast}^2,
\]
we obtain from \eqref{eq:nest_ode_0_rew}
\begin{align}
    e^{-\theta t}\dot{E}_t
    &\le \theta(f(x_t)-f^\ast)
    + \frac{\beta}{2}\bigl(\theta-\eta'\bigr)\norm{z_t-x^\ast}^2
    + \frac{\beta\eta'}{2}\norm{x_t-x^\ast}^2 \nonumber\\
    &\hspace{1.5cm}+ \eta \dotprod{\nabla f(x_t), x^\ast-x_t}
    + (\eta-\beta\gamma')\dotprod{\nabla f(x_t), z_t-x^\ast}
    - \gamma \norm{\nabla f(x_t)}^2.\label{eq:pl_corr_nest_1_alt}
\end{align}

We choose $\eta'=\theta$ to cancel the $\norm{z_t-x^\ast}^2$ term, and we choose $\eta=\beta\gamma'$ to cancel the
$\dotprod{\nabla f(x_t),z_t-x^\ast}$ term. Dropping the negative term $-\gamma\norm{\nabla f(x_t)}^2$, we get
\begin{equation}\label{eq:pl_corr_nest_1_alt_simpl}
    e^{-\theta t}\dot{E}_t
    \le \theta(f(x_t)-f^\ast)
    + \eta \dotprod{\nabla f(x_t), x^\ast-x_t}
    + \frac{\beta\theta}{2}\norm{x_t-x^\ast}^2.
\end{equation}
By Lemma~\ref{lem:scalar_prod_bound},
\[
\dotprod{\nabla f(x_t), x^\ast-x_t}
\le  -a \sqrt{\frac{\mu}{L_0}}(f(x_t)-f^\ast) - \frac{a}{2}\sqrt{\mu \mu_0}\norm{x_t-x^\ast}^2.
\]
Plugging into \eqref{eq:pl_corr_nest_1_alt_simpl} yields
\begin{align}
    e^{-\theta t}\dot{E}_t
    &\le \Bigl(\theta-\eta a\sqrt{\frac{\mu}{L_0}}\Bigr)(f(x_t)-f^\ast)
    + \frac12\Bigl(\beta\theta-\eta a\sqrt{\mu\mu_0}\Bigr)\norm{x_t-x^\ast}^2.\label{eq:pl_corr_nest_4_alt}
\end{align}
We impose the cancellation of both coefficients in \eqref{eq:pl_corr_nest_4_alt}:
\[
\theta=\eta a\sqrt{\frac{\mu}{L_0}}
\qquad\text{and}\qquad
\beta\theta=\eta a\sqrt{\mu\mu_0}.
\]
These two equalities are equivalent to $\beta=\sqrt{\mu_0L_0}$ and $\theta=\eta a\sqrt{\mu/L_0}$.
With the additional choice $\eta=\beta\gamma'$, we may take
\[
\beta=\sqrt{\mu_0L_0},\qquad
\gamma'=\beta^{-1/2}=(\mu_0L_0)^{-1/4},\qquad
\eta=\beta\gamma'=(\mu_0L_0)^{1/4},
\qquad
\eta'=\theta=a\Bigl(\frac{\mu_0}{L_0}\Bigr)^{1/4}\sqrt{\mu}.
\]
With these parameters, \eqref{eq:pl_corr_nest_4_alt} gives $e^{-\theta t}\dot E_t\le 0$, hence $E_t\le E_0$.
Therefore, using \eqref{eq:E_pl_accel},
\[
f(x_t)-f^\ast \le e^{-\theta t}E_t \le e^{-\theta t}E_0
= \exp\!\Bigl(-a\Bigl(\frac{\mu_0}{L_0}\Bigr)^{1/4}\sqrt{\mu}\,t\Bigr)\,E_0.
\]
Finally, using $x_0 = z_0$ and $f\in\qginff{\mu_0}$,
\[
\frac{\beta}{2}\norm{x_0-x^\ast}^2
\le \beta\,\frac{1}{\mu_0}(f(x_0)-f^\ast)
= \sqrt{\frac{L_0}{\mu_0}}(f( x_0)-f^\ast),
\]
so $E_0\le \bpar{1+\sqrt{\frac{L_0}{\mu_0}}}(f( x_0)-f^\ast)$, yields the claim.

\textbf{(ii) \:} The proof in the discrete case exploits the continuized framework. This implies that the following analysis differs from the previous ones. We refer to Appendix~\ref{app:continuized_intro} for a brief introduction to the tools needed when carrying such analysis, or \cite{hermant2025continuized} for a more detailed introduction. Also, we will prove the following more general result.
\begin{theorem}\label{thm:sto_pl_accel}
Let $f\in \pl \cap \ac \cap \ls$.  Denote $L_0 := \inf \{L'\le L:  f \in \qgf{L'} \}$, $\mu_0 := \sup\{\mu' \ge \mu : f \in \qginff{\mu'}\}$.
 Let $\xdisctilde$ verify \eqref{eq:nm} with the \hyperlink{cont_param}{continuized parameterization} and with $\nabla f(\tilde x_k)$ replaced by $\nabla f(\tilde x_k,\xi_k)$ for random variables $\xi_k$ independent of $\tilde x_k$ and under Assumption~\ref{ass:sgc}. Let $\eta' = \frac{a}{\rho}\bpar{\frac{\mu_0}{L_0}}^{1/4}\sqrt{\frac{\mu}{L}},\quad \gamma = \frac{1}{\rho L},\quad \gamma' = \frac{1}{\rho} \frac{1}{( \mu_0 L_0)^{1/4}\sqrt{ L}}, $ and $\eta = \frac{1}{\rho}\frac{(\mu_0 L_0)^{1/4}}{\sqrt{L}}$. Then,  
    with probability  $1-\frac{1}{c_0} - \exp(-(c_1-1-\log(c_1))k)$, for some $c_0 > 1$ and $c_1 \in (0,1)$ we have
$$ f(\tilde x_k)-f^\ast \le K_1e^{-\frac{a}{\rho}\bpar{\frac{\mu_0}{L_0}}^{1/4}\sqrt{\frac{\mu}{L}}(1-c_1)k}, $$
where $K_1 := c_0\bpar{1+\sqrt\frac{L_0}{\mu_0}}(f(x_0)-f^\ast)$.
\end{theorem}
The proof of Theorem~\ref{thm:pl_accel} (ii) is deduced in the case $\rho = 1$.

\begin{proof}[Proof of Theorem~\ref{thm:sto_pl_accel}]
Let $\overline{x}_t= (t,x_t,z_t)$. It satisfies
$d\overline{x}_t = b(\overline{x}_t)dt + G(\overline{x}_t,\xi)dN(t,\xi)$, where
\[
b(\overline{x}_t) = \begin{pmatrix}
1 \\ \eta(z_t-x_t) \\ \eta'(x_t-z_t)
\end{pmatrix},
\qquad
G(\overline{x}_t,\xi) = \begin{pmatrix}
0 \\ -\gamma \nabla f(x_t,\xi) \\ -\gamma' \nabla f(x_t,\xi)
\end{pmatrix}.
\]
We apply Proposition~\ref{prop:sto_calc} to $\phi(t)= \varphi(\overline{x}_t)$, where
\[
\varphi(t,x,z) =  A_t(f(x)-f^\ast) + \frac{B_t}{2}\norm{z - x^\ast}^2.
\]
Hence
\begin{equation}\label{eq:prop_calc_sto}
\varphi(\overline{x}_t) = \varphi(\overline{x}_0)
+ \int_{0}^t \dotprod{\nabla \varphi(\overline{x}_s),b(\overline{x}_s)}ds
+ \int_{0}^t \mathbb{E}_{\xi}\!\left[ \varphi(\overline{x}_s + G(\overline{x}_s,\xi)) - \varphi(\overline{x}_s)\right]ds
+ M_t,
\end{equation}
where $M_t$ is a martingale.

We have
\[
\frac{\partial \varphi}{\partial s}(\overline{x}_s)
= \frac{dA_s}{ds}\,(f(x_s)-f^\ast) + \frac{1}{2}\frac{dB_s}{ds}\norm{z_s - x^\ast}^2,
\qquad
\frac{\partial \varphi}{\partial x}(\overline{x}_s) = A_s\nabla f(x_s),
\qquad
\frac{\partial \varphi}{\partial z}(\overline{x}_s) = B_s(z_s - x^\ast).
\]
So,
\begin{align}
\dotprod{\nabla \varphi(\overline{x}_s),b(\overline{x}_s)}
&=\frac{dA_s}{ds}(f(x_s)-f^\ast) + \frac{1}{2}\frac{dB_s}{ds}\norm{z_s - x^\ast}^2 \nonumber\\
&\quad + A_s \eta \dotprod{ \nabla f(x_s) ,z_s-x_s}+ B_s \eta' \dotprod{z_s-x^\ast,x_s - z_s}.
\label{eq:pl_sqc_eq_12}
\end{align}
Also,
\begin{align}
\mathbb{E}_{\xi}\!\left[\varphi(\overline{x}_s+ G(\overline{x}_s,\xi)) - \varphi(\overline{x}_s)\right]
&= \mathbb{E}_{\xi}\!\left[A_s(f(x_s - \gamma \nabla f(x_s,\xi)) - f(x_s))\right]\nonumber\\
&\quad + \frac{B_s}{2}\mathbb{E}_{\xi}\!\left[\norm{ z_s - \gamma'\nabla f(x_s,\xi) - x^\ast}^2 - \norm{z_s -x^\ast}^2\right]\nonumber\\
&= A_s\mathbb{E}_{\xi}\!\left[f(x_s - \gamma \nabla f(x_s,\xi)) - f(x_s)\right]
+ \frac{B_s \gamma^{'2}}{2}\mathbb{E}_{\xi}\!\left[\norm{\nabla f(x_s,\xi)}^2\right]\label{eq:pl_sqc_eq_3}\\
&\quad + B_s\gamma' \dotprod{\nabla f(x_s),x^\ast - z_s}, \label{eq:pl_sqc_eq_4}
\end{align}
where we used $\mathbb{E}_{\xi}[\nabla f(x_s,\xi)] = \nabla f(x_s)$.

As $f \in \ls$ and Assumption~\ref{ass:sgc} holds, by Lemma~\ref{lem:descent_sgc} (with $\gamma=\frac{1}{\rho L}$),
\[
A_s\mathbb{E}_{\xi}\!\left[f(x_s - \gamma \nabla f(x_s,\xi)) - f(x_s)\right]
\le -\frac{A_s}{2\rho L}\norm{\nabla f(x_s)}^2.
\]
Moreover, by the strong growth condition,
$\mathbb{E}_{\xi}\!\left[\norm{\nabla f(x_s,\xi)}^2\right]\le \rho \norm{\nabla f(x_s)}^2$. Hence
\begin{equation}\label{eq:pl_avg_pre_1_corr}
\eqref{eq:pl_sqc_eq_3}
\le \frac{1}{2}\Bigl(\rho B_s \gamma^{'2} - \frac{A_s}{\rho L}\Bigr)\norm{\nabla f(x_s)}^2.
\end{equation}

We use the identity
\begin{equation}\label{eq:pl_avg_pre_2}
\dotprod{z_s-x^\ast,x_s - z_s}
= \frac{1}{2}\norm{x_s-x^\ast}^2 -\frac{1}{2}\norm{z_s-x^\ast}^2-\frac{1}{2}\norm{x_s-z_s}^2
\le\frac{1}{2}\norm{x_s-x^\ast}^2 -\frac{1}{2}\norm{z_s-x^\ast}^2.
\end{equation}
Also,
\begin{align}\label{eq:pl_avg_pre_3_corr}
A_s \eta \dotprod{ \nabla f(x_s) ,z_s-x_s} +  B_s\gamma' \dotprod{\nabla f(x_s),x^\ast - z_s}
&= A_s \eta \dotprod{ \nabla f(x_s) ,x^\ast-x_s}
+ \bpar{A_s\eta - B_s \gamma'} \dotprod{ \nabla f(x_s) ,z_s-x^\ast}.
\end{align}

By Lemma~\ref{lem:scalar_prod_bound},
\[
\dotprod{ \nabla f(x_s) ,x^\ast-x_s}
\le  -a \sqrt{\frac{\mu}{L_0}}\bigl(f(x_s)-f^\ast\bigr)
      - \frac{a}{2}\sqrt{\mu\mu_0}\norm{x_s-x^\ast}^2.
\]

Gathering \eqref{eq:pl_sqc_eq_12}, \eqref{eq:pl_sqc_eq_4}, \eqref{eq:pl_avg_pre_1_corr}, \eqref{eq:pl_avg_pre_2},
and \eqref{eq:pl_avg_pre_3_corr}, we obtain
\begin{align*}
&\dotprod{\nabla \varphi(\overline{x}_s),b(\overline{x}_s)}
+ \mathbb{E}_{\xi}\!\left[\varphi(\overline{x}_s+ G(\overline{x}_s,\xi)) - \varphi(\overline{x}_s)\right] \\
&\le \frac{1}{2}\bpar{ \frac{dB_s}{ds}-B_s \eta'}\norm{z_s-x^\ast}^2
+  \frac{1}{2}\bpar{  B_s \eta' - A_s \eta\,a \sqrt{\mu\mu_0}}\norm{x_s-x^\ast}^2
+ \bpar{\frac{dA_s}{ds}- A_s \eta\,a \sqrt{\frac{\mu}{L_0}}}\bigl(f(x_s)-f^\ast\bigr)\\
&\quad + \bpar{A_s\eta - B_s \gamma'} \dotprod{ \nabla f(x_s) ,z_s-x^\ast}
+ \frac{1}{2}\bpar{\rho B_s \gamma^{'2} - \frac{A_s}{\rho L}}\norm{\nabla f(x_s)}^2.
\end{align*}

\noindent
\textbf{Parameter tuning \:}
Set $A_s = e^{\theta s}$ and $B_s = \sqrt{\mu_0 L_0}\,e^{\theta s}$, with
\[
\theta=\eta'=\frac{a}{\rho}\Bigl(\frac{\mu_0}{L_0}\Bigr)^{1/4}\sqrt{\frac{\mu}{L}},
\quad
\gamma=\frac{1}{\rho L},
\quad
\gamma'=\frac{1}{\rho}\frac{1}{(\mu_0 L_0)^{1/4}\sqrt{L}},
\quad
\eta=\frac{1}{\rho}\frac{(\mu_0 L_0)^{1/4}}{\sqrt{L}}.
\]
These parameters imply the cancellations
\begin{align}
A_s\eta - B_s \gamma' &= 0,\label{eq:cancel_param:0}\\
\rho B_s \gamma^{'2} - \frac{A_s}{\rho L} &= 0,\label{eq:cancel_param:1}\\
\frac{dB_s}{ds}-B_s \eta' &= 0,\label{eq:cancel_param:2}\\
\frac{dA_s}{ds}- A_s \eta\,a \sqrt{\frac{\mu}{L_0}} &= 0,\label{eq:cancel_param:3}\\
B_s \eta' - A_s \eta\,a \sqrt{\mu\mu_0} &= 0.\label{eq:cancel_param:4}
\end{align}
Therefore,
\[
\dotprod{\nabla \varphi(\overline{x}_s),b(\overline{x}_s)}
+ \mathbb{E}_{\xi}\!\left[\varphi(\overline{x}_s+ G(\overline{x}_s,\xi)) - \varphi(\overline{x}_s)\right] \le 0,
\]
so from \eqref{eq:prop_calc_sto},
\[
\varphi(\overline{x}_t) \le \varphi(\overline{x}_0) + M_t.
\]

We evaluate at $t=T_k$, take expectation, and use Theorem~\ref{thm:martingal_stopping} and
Proposition~\ref{prop:discretization_constant_param} to get
\[
\E{e^{\theta T_k}\bigl(f(\tilde x_k)-f^\ast\bigr)} \le \varphi(\overline{x}_0).
\]
Finally, as $x_0 = z_0$ and $f\in \qginf$, we have
\[
\varphi(\overline{x}_0)
= f(\tilde x_0)-f^\ast + \frac{\sqrt{\mu_0 L_0}}{2}\norm{\tilde x_0-x^\ast}^2
\le \Bigl(1+\sqrt{\frac{L_0}{\mu_0}}\Bigr)\,(f(\tilde x_0)-f^\ast),
\]
where we used $\tilde x_0 = \tilde z_0$,
and we obtain \eqref{eq:a.s.:2}, namely
\begin{equation*}   \mathbb{E}\left[e^{\beta T_k}(f(\tilde{x}_k)-f^\ast)\right] \le K_0,
         \end{equation*}with $\beta=\theta$ and $K_0=\Bigl(1+\sqrt{\frac{L_0}{\mu_0}}\Bigr)(f(\tilde x_0)-f^\ast)$.
By Lemma~\ref{lem:conv_high_proba}, with probability $1-\frac{1}{c_0}-\exp(-(c_1-1-\log c_1)k)$,
\[
f(\tilde x_k)-f^\ast
\le c_0\Bigl(1+\sqrt{\frac{L_0}{\mu_0}}\Bigr)(f(\tilde x_0)-f^\ast)\,
\exp\!\Bigl(-\theta(1-c_1)k\Bigr),
\]
which is the claim.
\end{proof}

\section{Trajectory-adapted convergence}\label{app:adptative_condition}
In this section, we address a setting in which we assume the aiming condition is verified on average along the trajectory (Assumption~\ref{ass:avg_corelation}).
We start by proving Theorem~\ref{thm:gf_pl_&_avg_cor}, that addresses the dynamical system case. Then, we will state and prove similar results in the discrete case.
  \subsection{Proof of Theorem~\ref{thm:gf_pl_&_avg_cor}}\label{app:gf_pl_&_avg_cor}
\paragraph{Proof with Gradient Flow \eqref{eq:gf}}

Let
\[
E_t = e^{a_{avg}\sqrt{\mu \mu_0}\, t}\left[f(x_t)-f^\ast + \frac{\beta}{2}\norm{x_t-x^\ast}^2 \right].
\]
Then
\begin{align*}
\dot E_t
&= a_{avg}\sqrt{\mu\mu_0}\,e^{a_{avg}\sqrt{\mu\mu_0}t}
\left[f(x_t)-f^\ast + \frac{\beta}{2}\norm{x_t-x^\ast}^2 \right]\\
&\quad + e^{a_{avg}\sqrt{\mu\mu_0}t}\left[\dotprod{\nabla f(x_t),\dot x_t}
+ \beta \dotprod{x_t-x^\ast,\dot x_t}\right]\\
&\overset{(i)}{=} a_{avg}\sqrt{\mu\mu_0}\,e^{a_{avg}\sqrt{\mu\mu_0}t}
\left[f(x_t)-f^\ast + \frac{\beta}{2}\norm{x_t-x^\ast}^2 \right]\\
&\quad + e^{a_{avg}\sqrt{\mu\mu_0}t}\left[-\norm{\nabla f(x_t)}^2
- \beta \dotprod{x_t-x^\ast,\nabla f(x_t)}\right]\\
&\le a_{avg}\sqrt{\mu\mu_0}\,e^{a_{avg}\sqrt{\mu\mu_0}t}
\left[f(x_t)-f^\ast + \frac{\beta}{2}\norm{x_t-x^\ast}^2 \right]
- \beta e^{a_{avg}\sqrt{\mu\mu_0}t}\dotprod{x_t-x^\ast,\nabla f(x_t)},
\end{align*}
where (i) follows from $\dot x_t=-\nabla f(x_t)$.

Integrating from $0$ to $t$ yields
\begin{equation}\label{eq:gf_ac_average_1_corr}
E_t \le E_0
+ \int_0^t a_{avg}\sqrt{\mu\mu_0}\,e^{a_{avg}\sqrt{\mu\mu_0}s}
\left[f(x_s)-f^\ast + \frac{\beta}{2}\norm{x_s-x^\ast}^2 \right]ds
- \int_0^t \beta e^{a_{avg}\sqrt{\mu\mu_0}s}\dotprod{x_s-x^\ast,\nabla f(x_s)}\,ds.
\end{equation}

By Assumption~\ref{ass:avg_corelation} (with $\theta=a_{avg}\sqrt{\mu\mu_0}$), we have
\[
\int_0^t e^{\theta s}\dotprod{\nabla f(x_s),x_s-x^\ast}\,ds
\ge a_{avg}\int_0^t e^{\theta s}\norm{\nabla f(x_s)}\norm{x_s-x^\ast}\,ds,
\]
hence
\[
\int_0^t \beta e^{\theta s}\dotprod{x_s-x^\ast,\nabla f(x_s)}\,ds
\ge
\int_0^t \beta a_{avg} e^{\theta s}\norm{x_s-x^\ast}\norm{\nabla f(x_s)}\,ds.
\]
Plugging into \eqref{eq:gf_ac_average_1_corr} gives
\begin{equation}\label{eq:gf_ac_average_2_corr}
E_t \le E_0
+ \int_0^t a_{avg}\sqrt{\mu\mu_0}\,e^{\theta s}
\left[f(x_s)-f^\ast + \frac{\beta}{2}\norm{x_s-x^\ast}^2 \right]ds
- \int_0^t \beta a_{avg} e^{\theta s}\norm{x_s-x^\ast}\norm{\nabla f(x_s)}\,ds.
\end{equation}

Using Lemma~\ref{lem:prod}, for all $s\ge 0$,
\[
-\norm{x_s-x^\ast}\norm{\nabla f(x_s)}
\le -\sqrt{\frac{\mu}{L_0}}(f(x_s)-f^\ast) - \frac{\sqrt{\mu \mu_0}}{2}\norm{x_s-x^\ast}^2.
\]
Multiplying by $\beta a_{avg}e^{\theta s}$ and inserting into \eqref{eq:gf_ac_average_2_corr}, we obtain
\begin{align*}
E_t
&\le E_0 + \int_0^t e^{\theta s}\Bigl[
a_{avg}\sqrt{\mu\mu_0}\,(f(x_s)-f^\ast)
- \beta a_{avg}\sqrt{\frac{\mu}{L_0}}\,(f(x_s)-f^\ast)\\
&\hspace{2.2cm}
+ \frac{\beta}{2}a_{avg}\sqrt{\mu\mu_0}\,\norm{x_s-x^\ast}^2
- \frac{\beta}{2}a_{avg}\sqrt{\mu\mu_0}\,\norm{x_s-x^\ast}^2
\Bigr]ds\\
&= E_0 + a_{avg}\int_0^t e^{\theta s}
\Bigl(\sqrt{\mu\mu_0}-\beta\sqrt{\frac{\mu}{L_0}}\Bigr)\,(f(x_s)-f^\ast)\,ds.
\end{align*}
Choose $\beta=\sqrt{\mu_0L_0}$, so that
$\sqrt{\mu\mu_0}-\beta\sqrt{\mu/L_0}=0$. Then $E_t\le E_0$ for all $t\ge 0$.
Since $f(x_t)-f^\ast \le e^{-\theta t}E_t$, we get
\[
f(x_t)-f^\ast \le e^{-\theta t}E_0
= e^{-a_{avg}\sqrt{\mu\mu_0}\,t}\Bigl(f(x_0)-f^\ast + \frac{\beta}{2}\norm{x_0-x^\ast}^2\Bigr).
\]
Finally, using $f(x_0)-f^\ast \ge \frac{\mu_0}{2}\norm{x_0-x^\ast}^2$ and $\beta=\sqrt{\mu_0L_0}$,
\[
\frac{\beta}{2}\norm{x_0-x^\ast}^2
\le \sqrt{\frac{L_0}{\mu_0}}\,(f(x_0)-f^\ast),
\]
hence
\[
f(x_t)-f^\ast \le \Bigl(1+\sqrt{\frac{L_0}{\mu_0}}\Bigr)e^{-a_{avg}\sqrt{\mu\mu_0}\,t}\,(f(x_0)-f^\ast).
\]

\paragraph{Proof with Nesterov ODE momentum \eqref{eq:nmo}}

Let
\[
E_t = e^{\theta t}\Bigl(f(x_t)-f^\ast\Bigr)+\frac{\beta}{2}e^{\theta t}\norm{z_t-x^\ast}^2.
\]
As in the proof of Theorem~\ref{thm:pl_accel} (i), choose $\eta'=\theta$ and $\eta=\beta\gamma'$.
Then we get (see \eqref{eq:pl_corr_nest_1_alt_simpl})
\begin{equation}\label{eq:pl_corr_avg_nest_1_corr}
 \dot{E}_t \le e^{\theta t}\Bigl(
 \theta(f(x_t)-f^\ast)
 +\eta\dotprod{\nabla f(x_t), x^\ast - x_t}
 +\frac{\beta\theta}{2}\norm{x_t-x^\ast}^2
 \Bigr).
\end{equation}

Integrating,
\begin{align}
E_t
&\le E_0 + \int_0^t e^{\theta s}\Bigl(
 \theta(f(x_s)-f^\ast)
 +\frac{\beta\theta}{2}\norm{x_s-x^\ast}^2
 \Bigr)\,ds
 + \eta\int_0^t e^{\theta s}\dotprod{\nabla f(x_s),x^\ast-x_s}\,ds \nonumber\\
&= E_0 + \int_0^t e^{\theta s}\Bigl(
 \theta(f(x_s)-f^\ast)
 +\frac{\beta\theta}{2}\norm{x_s-x^\ast}^2
 \Bigr)\,ds
 - \eta\int_0^t e^{\theta s}\dotprod{\nabla f(x_s),x_s-x^\ast}\,ds.\label{eq:pl_corr_avg_nest_2_corr}
\end{align}

By Assumption~\ref{ass:avg_corelation},
\[
\int_0^t e^{\theta s}\dotprod{\nabla f(x_s),x_s-x^\ast}\,ds
\ge a_{avg}\int_0^t e^{\theta s}\norm{\nabla f(x_s)}\norm{x_s-x^\ast}\,ds,
\]
hence
\begin{equation}\label{eq:pl_corr_avg_nest_3_corr}
E_t
\le E_0 + \int_0^t e^{\theta s}\Bigl(
 \theta(f(x_s)-f^\ast)
 +\frac{\beta\theta}{2}\norm{x_s-x^\ast}^2
 \Bigr)\,ds
 - \eta a_{avg}\int_0^t e^{\theta s}\norm{\nabla f(x_s)}\norm{x_s-x^\ast}\,ds.
\end{equation}

Using Lemma~\ref{lem:prod},
\[
-\norm{x_s-x^\ast}\norm{\nabla f(x_s)}
\le -\sqrt{\frac{\mu}{L_0}}(f(x_s)-f^\ast)
 -\frac{\sqrt{\mu\mu_0}}{2}\norm{x_s-x^\ast}^2,
\]
and inserting into \eqref{eq:pl_corr_avg_nest_3_corr}, we obtain
\begin{equation}\label{eq:pl_corr_avg_nest_4_corr}
E_t
\le E_0 + \int_0^t e^{\theta s}\left[
\Bigl(\theta-\eta a_{avg}\sqrt{\frac{\mu}{L_0}}\Bigr)(f(x_s)-f^\ast)
+\frac{1}{2}\Bigl(\beta\theta-\eta a_{avg}\sqrt{\mu\mu_0}\Bigr)\norm{x_s-x^\ast}^2
\right]ds.
\end{equation}

Choose $\beta=\sqrt{\mu_0L_0}$ and impose the cancellations
\[
\theta=\eta a_{avg}\sqrt{\frac{\mu}{L_0}}
\qquad\text{and}\qquad
\beta\theta=\eta a_{avg}\sqrt{\mu\mu_0}.
\]
These two equalities are equivalent to $\beta=\sqrt{\mu_0L_0}$ and
\[
\theta=a_{avg}\Bigl(\frac{\mu_0}{L_0}\Bigr)^{1/4}\sqrt{\mu},
\qquad
\eta=(\mu_0L_0)^{1/4}.
\]
With $\eta'=\theta$ and $\eta=\beta\gamma'$, we can take
\[
\gamma'=\frac{\eta}{\beta}=(\mu_0L_0)^{-1/4},
\qquad
\eta'=\theta=a_{avg}\Bigl(\frac{\mu_0}{L_0}\Bigr)^{1/4}\sqrt{\mu},
\qquad
\gamma\ge 0.
\]
Then the integrand in \eqref{eq:pl_corr_avg_nest_4_corr} is identically zero and we get $E_t\le E_0$.
Since $f(x_t)-f^\ast \le e^{-\theta t}E_t$, we deduce
\[
f(x_t)-f^\ast \le e^{-\theta t}E_0
= e^{-\theta t}\Bigl(f(x_0)-f^\ast+\frac{\beta}{2}\norm{z_0-x^\ast}^2\Bigr).
\]
As $z_0=x_0$, then
\[
f(x_t)-f^\ast \le e^{-\theta t}\Bigl(f(x_0)-f^\ast+\frac{\sqrt{\mu_0L_0}}{2}\norm{x_0-x^\ast}^2\Bigr).
\]
Using $f(x_0)-f^\ast \ge \frac{\mu_0}{2}\norm{x_0-x^\ast}^2$, we obtain
\[
f(x_t)-f^\ast
\le \Bigl(1+\sqrt{\frac{L_0}{\mu_0}}\Bigr)e^{-\theta t}\,(f(x_0)-f^\ast),
\qquad
\theta=a_{avg}\Bigl(\frac{\mu_0}{L_0}\Bigr)^{1/4}\sqrt{\mu}.
\]

\subsection{Discrete case}\label{app:pl_accel_avg}
We address the discrete case. As in Appendix~\ref{app:cv_results}, we consider stochastic gradients under Assumption~\ref{ass:sgc}.  
\begin{assumption}\label{ass:avg}
  For some $\{ A_i \}_{i =1,\dots,k}$ and $a_{avg} > 0$, $\{\tilde y_k \}_{k\ge 0}$ is such that
    $$\sum_{i = 1}^k A_i  \left[\dotprod{\nabla
     f(\tilde y_i,\xi_i),\tilde y_i-x^\ast}-a_{avg} \norm{\nabla
     f(\tilde y_i,\xi_i)}\norm{x^\ast - \tilde y_i}\right] \ge 0.$$
\end{assumption}
We consider \eqref{eq:nm}, a similar and simpler analysis can be carried for \eqref{eq:gd}.

\begin{theorem}\label{thm:pl_accel_avg}
Let $f\in \pl \cap \ls$. Denote
$L_0 := \inf \{L'\le L:  f \in \qgf{L'} \}$,
$\mu_0 := \sup\{\mu' \ge \mu : f \in \qginff{\mu'}\}$.
Let $\xdisctilde$ verify \eqref{eq:nm} with the \hyperlink{cont_param}{continuized parameterization} and with
$\nabla f(\tilde x_k)$ replaced by $\nabla f(\tilde x_k,\xi_k)$ for random variables $\xi_k$ independent of $\tilde x_k$,
under Assumption~\ref{ass:sgc}, and assume that Assumption~\ref{ass:avg} holds with $A_i = e^{\theta T_i}$, where
\[
\theta = \eta' = \frac{a_{avg}}{\rho}\Bigl(\frac{\mu_0}{L_0}\Bigr)^{1/4}\sqrt{\frac{\mu}{L}}.
\]
Let
\[
\gamma = \frac{1}{\rho L},
\quad
\gamma' = \frac{1}{\rho(\mu_0 L_0)^{1/4}\sqrt{L}},
\quad
\eta = \frac{1}{\rho}\frac{(\mu_0 L_0)^{1/4}}{\sqrt{L}}.
\]
Then, with probability $1-\frac{1}{c_0} - e^{-(c_1-1-\log(c_1))k}$ (for some $c_0>1$, $c_1\in(0,1)$),
\[
f(\tilde x_k)-f^\ast \le c_0\Bigl(1+\sqrt{\frac{L_0}{\mu_0}}\Bigr)(f(\tilde x_0)-f^\ast)\,
\exp\!\Bigl(-\theta(1-c_1)k\Bigr).
\]
\end{theorem}

\begin{proof}
We use the same notations as in the proof of Theorem~\ref{thm:sto_pl_accel}.
Let $\overline{x}_t=(t,x_t,z_t)$ and define
\[
\varphi(t,x,z)
=
A_t(f(x)-f^\ast)+\frac{B_t}{2}\norm{z-x^\ast}^2,
\qquad
A_t=e^{\theta t},
\quad
B_t=\sqrt{\mu_0L_0}\,e^{\theta t}.
\]

Applying Proposition~\ref{prop:sto_calc}, we have
\begin{equation}\label{eq:avg_prop_calc}
\varphi(\overline{x}_t)
=
\varphi(\overline{x}_0)
+
\int_0^t
\Bigl[
\dotprod{\nabla \varphi(\overline{x}_s),b(\overline{x}_s)}
+
\mathbb{E}_{\xi}\!\bigl(
\varphi(\overline{x}_s+G(\overline{x}_s,\xi))-\varphi(\overline{x}_s)
\bigr)
\Bigr]ds
+M_t,
\end{equation}
where $M_t$ is a martingale with $\E{M_t}=0$.

From the same algebraic computations as in
Theorem~\ref{thm:pl_accel} (ii),
we obtain
\begin{align}\label{eq:avg_drift_1}
&\dotprod{\nabla \varphi(\overline{x}_s),b(\overline{x}_s)}
+
\mathbb{E}_{\xi}\!\bigl(
\varphi(\overline{x}_s+G(\overline{x}_s,\xi))-\varphi(\overline{x}_s)
\bigr)
\\
&\le
\frac{1}{2}\Bigl(\frac{dB_s}{ds}-B_s\eta'\Bigr)\norm{z_s-x^\ast}^2
+
\frac{B_s\eta'}{2}\norm{x_s-x^\ast}^2
+
\frac{dA_s}{ds}(f(x_s)-f^\ast)
\nonumber\\
&\quad
+
A_s\eta\,\dotprod{\nabla f(x_s),x^\ast-x_s}
+
(A_s\eta-B_s\gamma')\dotprod{\nabla f(x_s),z_s-x^\ast}
\nonumber\\
&\quad
+
\frac12\Bigl(\rho B_s\gamma'^2-\frac{A_s}{\rho L}\Bigr)\norm{\nabla f(x_s)}^2.
\nonumber
\end{align}

With the parameter choices of the theorem and $B_s=\sqrt{\mu_0L_0}A_s$, we have
\[
A_s\eta-B_s\gamma'=0,
\qquad
\rho B_s\gamma'^2-\frac{A_s}{\rho L}=0,
\qquad
\frac{dB_s}{ds}-B_s\eta'=0.
\]
Thus \eqref{eq:avg_drift_1} reduces to
\begin{equation}\label{eq:avg_drift_2}
\dotprod{\nabla \varphi(\overline{x}_s),b(\overline{x}_s)}
+
\mathbb{E}_{\xi}\!\bigl(
\varphi(\overline{x}_s+G(\overline{x}_s,\xi))-\varphi(\overline{x}_s)
\bigr)
\le
\frac{dA_s}{ds}(f(x_s)-f^\ast)
+
\frac{B_s\eta'}{2}\norm{x_s-x^\ast}^2
+
A_s\eta\,\dotprod{\nabla f(x_s),x^\ast-x_s}.
\end{equation}

We now control the last term using the jump structure.
For any predictable process $h_s$,
\[
\int_0^t h_s\,dN_s
=
\int_0^t h_s\,ds
+
\int_0^t h_s\,(dN_s-ds),
\]
where $\int_0^t h_s(dN_s-ds)$ is a martingale.

Applying this identity with
\[
h_s
=
A_s\eta\Bigl[
\dotprod{\nabla f(x_s,\xi_s),x_s-x^\ast}
-
a_{avg}\norm{\nabla f(x_s,\xi_s)}\norm{x_s-x^\ast}
\Bigr],
\]
and using Assumption~\ref{ass:avg} with $A_i=e^{\theta T_i}$, we obtain for $t=T_k$
\begin{equation}\label{eq:avg_jump_control}
\int_0^{T_k} A_s\eta\,\dotprod{\nabla f(x_s),x^\ast-x_s}\,dN_s
\le
-\eta a_{avg}\int_0^{T_k} A_s\norm{\nabla f(x_s)}\norm{x_s-x^\ast}\,dN_s.
\end{equation}
This holds because by definition of the Poisson integrals, we have
\[\int_0^{T_k} A_s\eta\,\dotprod{\nabla f(x_s),x^\ast-x_s}\,dN_s = \sum_{i=1}^k A_i\eta \dotprod{\nabla f(x_{T_i^-}),x^\ast-x_{T_i^-}} = \sum_{i=1}^k A_i\eta \dotprod{\nabla f(\tilde y_{i-1}),x^\ast-\tilde y_{i-1}},\]
where the last equality follows from Proposition~\ref{prop:discretization_constant_param}.

Using Lemma~\ref{lem:prod},
\[
-\norm{\nabla f(x_s)}\norm{x_s-x^\ast}
\le
-\sqrt{\frac{\mu}{L_0}}(f(x_s)-f^\ast)
-
\frac{\sqrt{\mu\mu_0}}{2}\norm{x_s-x^\ast}^2,
\]
and inserting this inequality into \eqref{eq:avg_jump_control},
then back into \eqref{eq:avg_drift_2},
we obtain
\[
\dotprod{\nabla \varphi(\overline{x}_s),b(\overline{x}_s)}
+
\mathbb{E}_{\xi}\!\bigl(
\varphi(\overline{x}_s+G(\overline{x}_s,\xi))-\varphi(\overline{x}_s)
\bigr)
\le 0.
\]

Plugging this bound into \eqref{eq:avg_prop_calc} yields
\[
\varphi(\overline{x}_t)\le \varphi(\overline{x}_0)+M_t,
\]
where $M_t$ is a martingale.
Evaluating at $t=T_k$, taking expectations, and using
Theorem~\ref{thm:martingal_stopping}, we get
\[
\E{e^{\theta T_k}(f(\tilde x_k)-f^\ast)}\le \varphi(\overline{x}_0).
\]

Finally,
\[
\varphi(\overline{x}_0)
=
f(\tilde x_0)-f^\ast+\frac{\sqrt{\mu_0L_0}}{2}\norm{\tilde z_0-x^\ast}^2
\le
\Bigl(1+\sqrt{\frac{L_0}{\mu_0}}\Bigr)(f(\tilde x_0)-f^\ast),
\]
and Lemma~\ref{lem:conv_high_proba} concludes the proof.
\end{proof}

\subsection{Possible Extensions: Adapt to other Conditions and Exact Rates}\label{app:exact}
We could consider other extensions of the proposed viewpoint. First, we could consider other averaged conditions, such as an \textit{averaged PL} condition:
\[ \int_0^t e^{\theta t}\norm{\nabla f(x_s)}^2ds \ge 2\mu_{avg}\int_0^t (f(x_s)-f^\ast )ds, \]
for some $\mu_{avg} > 0$.
This inequality is automatically verified if $f \in \pl$, with $\mu_{avg} \ge \mu$. Interestingly, $\xcont$ could cross a saddle point $\hat x$, such that $f \notin \pl$, while having that this inequality is verified. In particular, \eqref{eq:nmo} could cross and escape such saddle point, and then converge to a global minimum. This saddle point could be ``averaged out''.

Alternatively, we also could consider point-wise parameters, as illustrated next.
\begin{proposition}\label{prop:exact_gf}
    Define the point-wise Polyak-\L{}ojasiewicz constant $\mu(x) :=  \frac{\norm{\nabla f(x)}^2}{2(f(x)-f^\ast)}$. If $\xcont \sim \eqref{eq:gf}$, we have
    $$f(x_t)-f^\ast = e^{-2\int_{0}^t \mu(x_s)ds}(f(x_0)-f^\ast).$$
\end{proposition}
We obtain an exact rate for the decrease of $f(x_t)-f^\ast$. If $f\in \pl$, we have in particular $2\int_0^t \mu(x_s)ds \ge 2\mu t$. We can carry similar analysis with \eqref{eq:nmo}, under a point-wise aiming condition.

\begin{proposition}\label{prop:exact_nmo}
    Define the point-wise: Polyak-Lojasiewicz constant $\mu(x) :=  \frac{\norm{\nabla f(x)}^2}{2(f(x)-f^\ast)}$, aiming condition constant $a(x) := \frac{\dotprod{\nabla f(x),x-x^\ast}}{\norm{\nabla f(x)}\norm{x-x^\ast}}$ and quadratic growth constant $\mu_0(x) = 2\frac{f(x)-f^\ast}{\norm{x-x^\ast}^2}$.
   If $\xcont \sim \eqref{eq:nmo}$ with $ \eta'_t = a(x_t)\sqrt{\mu(x_t)},\quad \eta_t = \frac{L_0}{\sqrt{\mu_0(x_t)}}, \quad \gamma'_t = \frac{1}{\sqrt{\mu_0(x_t)}},\quad \gamma_t \ge 0$, we have
    $$f(x_t)-f^\ast \le e^{-\int_0^t a(x_s)\sqrt{\mu(x_s)}ds}L_0\norm{x_0-x^\ast}^2,$$
    with $L_0 = \sup_{x\in \R^d} \mu_0(x)$.
\end{proposition}
This rate is however purely conceptual, as it involves exact values of parameters we cannot compute. The averaged conditions such as Assumption~\ref{ass:avg_corelation} are more practical, as it is sufficient to use a small enough value of $a_{avg}$, and not an exact one.
\subsubsection{Proof of Proposition~\ref{prop:exact_gf}}
\textbf{(i) Case of gradient flow \eqref{eq:gf}. \:}
    Set
    $$E_t:= e^{2\int_{0}^t \mu(x_s)ds}(f(x_t)-f^\ast). $$
    We have
    \begin{equation}
        \begin{aligned}\label{eq:pl_avg:1}
                    \dot{E_t} &= 2\mu(x_t)e^{2\int_{0}^t \mu(x_s)ds}(f(x_t)-f^\ast) + e^{2\int_{0}^t \mu(x_s)ds}\dotprod{\nabla f(x_t),\dot x_t}\\
        &\overset{\eqref{eq:gf}}{=} 2\mu(x_t)e^{2\int_{0}^t \mu(x_s)ds}(f(x_t)-f^\ast) - e^{2\int_{0}^t \mu(x_s)ds}\norm{\nabla f(x_t)}^2
        \end{aligned}
    \end{equation}

    Recall $\mu(x) := \frac{\norm{\nabla f(x)}^2}{2(f(x)-f^\ast)}$. Plugging it in \eqref{eq:pl_avg:1}, we deduce
    \begin{align}
        \dot{E_t} = 2\mu(x_t)e^{2\int_{0}^t \mu(x_s)ds}(f(x_t)-f^\ast) - 2\mu(x_t)e^{2\int_{0}^t \mu(x_s)ds}(f(x_t)-f^\ast)=0.
    \end{align}
    By integrating we deduce 
    \begin{equation}
        f(x_t)-f^\ast = e^{-2\int_{0}^t \mu(x_s)ds}(f(x_0)-f^\ast).
    \end{equation}

\subsubsection{Proof of Proposition~\ref{prop:exact_nmo}}

Let
\begin{equation}
    E_t
    =
    e^{\int_0^t a(x_s)\sqrt{\mu(x_s)}ds}
    \left(f(x_t)-f^\ast + \frac{\beta}{2}\norm{z_t-x^\ast}^2\right).
\end{equation}
Denote
\[
\Delta_t := e^{\int_0^t a(x_s)\sqrt{\mu(x_s)}ds},
\qquad
u_t := a(x_t)\sqrt{\mu(x_t)}.
\]

We have
\begin{align}
    \Delta_t^{-1}\dot E_t
    &= u_t(f(x_t)-f^\ast)
    + \frac{u_t\beta}{2}\norm{z_t-x^\ast}^2
    + \dotprod{\nabla f(x_t),\dot x_t}
    + \beta\dotprod{z_t-x^\ast,\dot z_t} \nonumber\\
    &= u_t(f(x_t)-f^\ast)
    + \frac{u_t\beta}{2}\norm{z_t-x^\ast}^2 \nonumber\\
    &\quad
    + \dotprod{\nabla f(x_t),\eta_t(z_t-x_t)-\gamma_t\nabla f(x_t)} \nonumber\\
    &\quad
    + \beta\dotprod{z_t-x^\ast,\eta'_t(x_t-z_t)-\gamma'_t\nabla f(x_t)} \nonumber\\
    &= u_t(f(x_t)-f^\ast)
    + \frac{u_t\beta}{2}\norm{z_t-x^\ast}^2 \nonumber\\
    &\quad
    + \eta_t\dotprod{\nabla f(x_t),x^\ast-x_t}
    + (\eta_t-\beta\gamma'_t)\dotprod{\nabla f(x_t),z_t-x^\ast} \nonumber\\
    &\quad
    + \beta\eta'_t\dotprod{z_t-x^\ast,x_t-z_t}
    - \gamma_t\norm{\nabla f(x_t)}^2 .
\end{align}

We choose $\eta'_t = u_t$ to cancel the $\norm{z_t-x^\ast}^2$ term, and
$\eta_t = \beta\gamma'_t$ to cancel the mixed scalar product. Then
\begin{equation}\label{eq:exact:nest_ode_1_corr}
    \Delta_t^{-1}\dot E_t
    =
    u_t(f(x_t)-f^\ast)
    + \beta\gamma'_t\dotprod{\nabla f(x_t),x^\ast-x_t}
    + \frac{\beta u_t}{2}\norm{x_t-x^\ast}^2
    - \gamma_t\norm{\nabla f(x_t)}^2 .
\end{equation}

By Lemma~\ref{lem:tec:exact_scalar_prod_bound},
\[
\dotprod{\nabla f(x_t),x^\ast-x_t}
=
- a(x_t)\sqrt{\frac{\mu(x_t)}{\mu_0(x_t)}}(f(x_t)-f^\ast)
- a(x_t)\sqrt{\mu(x_t)\mu_0(x_t)}\frac{1}{2}\norm{x_t-x^\ast}^2 .
\]

Fixing $\gamma'_t = \frac{1}{\sqrt{\mu_0(x_t)}}$ and ignoring the nonpositive term
$-\gamma_t\norm{\nabla f(x_t)}^2$, we obtain
\begin{equation}
    \Delta_t^{-1}\dot E_t
    =
    u_t\Bigl(1-\frac{\beta}{\mu_0(x_t)}\Bigr)(f(x_t)-f^\ast).
\end{equation}

Choosing $\beta = L_0 = \sup_{x\in\R^d}\mu_0(x)$ ensures $\dot E_t \le 0$ for all $t\ge0$.
Therefore $E_t \le E_0$, and
\[
f(x_t)-f^\ast
\le
e^{-\int_0^t a(x_s)\sqrt{\mu(x_s)}ds}\,E_0.
\]

Finally, since $\mu_0(x_0)\le L_0$, we have
\[
f(x_0)-f^\ast = \frac{\mu_0(x_0)}{2}\norm{x_0-x^\ast}^2 \le \frac{L_0}{2}\norm{x_0-x^\ast}^2,
\]
which yields
\[
f(x_t)-f^\ast
\le
e^{-\int_0^t a(x_s)\sqrt{\mu(x_s)}ds}\,L_0\norm{x_0-x^\ast}^2.
\]

\section{Continuized Toolbox}\label{app:continuized_intro}
The continuized framework \cite{even2021continuized} offers an alternative momentum method compared to the more classic ones. In practice, it results in a stochastic parameterization of \eqref{eq:nm}. It can lead to interesting results in some non-convex settings, see \cite{wang2023continuized,hermant2025continuized}. We refer to \cite[Section 3]{hermant2025continuized} for a methodological introduction of this concept, on the class of functions $\pl \cap \ls$. In this section, we will simply state the fundamental objects and tools. The fundamental objects are the continuized Nesterov equations:
\begin{align}\label{eq:nest_continuized}\tag{CNE}\left\{
    \begin{array}{ll}
        dx_t &= \eta(z_t-x_t)dt - \gamma \nabla f(x_t) dN_t \\
        dz_t &= \eta'(x_t-z_t)dt - \gamma' \nabla f(x_t) dN_t
    \end{array}
\right.
\end{align} where $dN(t) = \sum_{k\ge 0} \delta_{T_k}(dt)$ is a Poisson point measure with intensity $dt$, with $T_1,T_2,\dots$ random times such that for all $i$, $T_{i+1}-T_i \overset{i.i.d}{\sim} \mathcal{E}(1)$. It mixes a continuous component, the $dt$ factor, with a discrete component that acts at random times, the $dN_t$ factor. $\xz$ cannot be differentiated, but we can still use Lyapunov-approaches through Itô calculus. 
\begin{proposition}[\cite{even2021continuized}, Proposition 2]\label{prop:sto_calc}
Let $x_t \in \R^d$ be a solution of 
\begin{equation*}
    dx_t = b(x_t)dt + \int_\Xi G(x_t,\xi)dN(t,\xi)
\end{equation*}
and $\varphi : \R^d \to \R$ be a smooth function. Then,
\begin{equation}
    \varphi(x_t) = \varphi(x_0) + \int_{0}^t \dotprod{\nabla \varphi(x_s),b(x_s)}ds + \int_{[0,t]}  \mathbb{E}_\xi{ \varphi(x_s + G(x_s,\xi)) - \varphi(x_s)}ds + M_t,
\end{equation}

where $M_t$ is a martingale such that $\E{M_t}=0$, $\forall t\ge 0$.

\end{proposition}
Importantly, the following evaluations $\tilde{y}_k := x_{T_{k+1}^-}$, $\tilde{x}_k := x_{T_{k}}$ and $\tilde{z}_k := z_{T_{k}}$ of \eqref{eq:nest_continuized} can be written in algorithmic form. 

\begin{proposition}[\cite{hermant2025continuized}, Proposition 5]\label{prop:discretization_constant_param}
Let $\xz$ follow \eqref{eq:nest_continuized}. Define $\tilde{y}_k := x_{T_{k+1}^-}$, $\tilde{x}_k := x_{T_{k}}$ and $\tilde{z}_k := z_{T_{k}}$ as evaluations of this process. Then, $(\tilde{y}_k,\tilde{x}_k, \tilde{z}_k)$ writes as a Nesterov algorithm in the form of \eqref{eq:nm} with stochastic parameters 
\begin{align}\label{alg:constant_param}\left\{
    \begin{array}{ll}
        \tilde{y}_k &= (\tilde{z}_{k} - \tilde{x}_k)\frac{\eta}{\eta+\eta'}\bpar{1 -e^{-(\eta + \eta')(T_{k+1}-T_k)} } + \tilde{x}_k \\
        \tilde{x}_{k+1} &=  \tilde{y}_k  - \gamma \nabla f(\tilde{y}_k,\xi_k)\\
        \tilde{z}_{k+1} &=   \tilde{z}_{k} + \eta' \frac{\bpar{1- e^{-(\eta+\eta')(T_{k+1}-T_k)}}}{\eta' + \eta e^{-(\eta+\eta')(T_{k+1}-T_k)}}(\tilde{y}_k - \tilde{z}_{k}) - \gamma' \nabla f(\tilde{y}_k,\xi_k)
    \end{array}
\right.
\end{align}
\end{proposition}
In order to transfer a result involving $\xz$ to $(\tilde x_k,\tilde z_k)_{k\in\N}$, we need a stopping theorem, such as the following.
\begin{theorem}[Stopping theorem, \cite{hermant2025continuized}, Theorem 6]\label{thm:martingal_stopping}
    Let $(\varphi_t)_{t\in \R_+}$ be a nonnegative process with cadlag trajectories, such that it verifies
    $$ \varphi_t \le K_0 + M_t, $$
    for some positive random variable $K_0$, some martingale $\mart$ with $M_0 = 0$. Then, for an almost surely finite stopping time $\tau$, one has 
$$ \E{\varphi_\tau} \le \E{K_0}.$$
\end{theorem}
Finally, because the continuized analysis leads to results in expectation, in the case where we obtain linear convergence rates, concentration inequalities such as Chernoff's inequality allows to deduce a result on the trajectory, with high probability.
\begin{lemma}[\cite{hermant2025continuized}, Proposition 14]\label{lem:conv_high_proba}
     Assume $T_1,\dots,T_k$ are random variables such that $T_{i+1} - T_i$ are  i.i.d  of law $\mathcal{E}(1)$, with convention $T_0 = 0$. If the iterations of (\ref{alg:constant_param}) verify
                     \begin{equation}\label{eq:a.s.:2}       \mathbb{E}\left[e^{\beta T_k}(f(\tilde{x}_k)-f^\ast)\right] \le K_0,
         \end{equation}
     then, with probability  $1-\frac{1}{c_0} - e^{-(c_1-1-\log(c_1))k}$, for some $c_0 > 1$ and $c_1 \in (0,1)$ we have
         \begin{equation*}
    f(\tilde{x}_k)-f^\ast \le c_0K_0e^{-\beta(1-c_1)k}.
    \end{equation*}
\end{lemma}

\section{Proof of Proposition~\ref{lem:toy_example_ac}}
\textbf{Computation of $\nabla h(x)$.} We have
\begin{equation}
    \begin{aligned}
        \nabla h(x) &= \frac{x}{\norm{x}}f'(\norm{x})g\left(\frac{x}{\lVert x \rVert} \right) + \bpar{\frac{1}{\norm{x}}I_d - \frac{x x^T}{\norm{x}^3}}f(\lVert x \rVert)\nabla g\left(\frac{x}{\lVert x \rVert} \right)\\
        &=\frac{x}{\norm{x}}f'(\norm{x})g\left(\frac{x}{\lVert x \rVert} \right) + \frac{1}{\norm{x}}\bpar{I_d -\frac{x x^T}{\norm{x}^2}}f(\lVert x \rVert)\nabla g\left(\frac{x}{\lVert x \rVert} \right)
    \end{aligned}
\end{equation}
Noting $r := \norm{x}$ and $u :=x/\norm{x}$, this can be rewritten
\begin{equation}
    \begin{aligned}
        \nabla h(x) = \nabla  h(ru) = f'(r)g(u)u + \frac{1}{r}f(r) \nabla_{\sphere} g(u),
    \end{aligned}
\end{equation}
where $\nabla_{\sphere} g(u) = (I_d-uu^T)\nabla g(u)$ is the Riemannian gradient of $g$ on the sphere. $\nabla_{\sphere} g(u)$ belongs to the tangent space at $u$, namely $\nabla_{\sphere} g(u) \in  \mathcal{T}_u = \{ v, \dotprod{u,v} = 0 \}$. In particular, it implies that $u$ and $\nabla_{\sphere} g(u)$ are orthogonal, so
\[\norm{\nabla h(x)}^2 =f'(r)^ 2g(u)^2 + \frac{1}{r^2}f(r)^2 \norm{\nabla_{\sphere} g(u)}^2, \]
and 
\[\dotprod{u,\nabla_{\sphere} g(u)} = 0.\]
So, we have
\[\dotprod{\nabla h(x),x-x^\ast} = \dotprod{f'(r)g(u)u + \frac{1}{r}f(r) \nabla_{\sphere} g(u),ru} = rf'(r)g(u),\]
 and 
 \[\norm{\nabla h(x)}\norm{x-x^\ast} = \sqrt{f'(r)^ 2g(u)^2 + \frac{1}{r^2}f(r)^2 \norm{\nabla_{\sphere} g(u)}^2}r, \]
 such that
 \[\frac{\dotprod{\nabla h(x),x-x^\ast} }{\norm{\nabla h(x)}\norm{x-x^\ast}} = \frac{f'(r)g(u)}{\sqrt{f'(r)^ 2g(u)^2 + \frac{1}{r^2}f(r)^2 \norm{\nabla_{\sphere} g(u)}^2}}.\]
 So, $h \in \ac$ for $a$ such that
 \[a := \inf_{r>0,u \in \mathbb{S}^{d-1}} \frac{f'(r)g(u)}{\sqrt{f'(r)^ 2g(u)^2 + \frac{1}{r^2}f(r)^2 \norm{\nabla_{\sphere} g(u)}^2}} = \inf_{r>0,u \in \mathbb{S}^{d-1}} \frac{1}{\sqrt{1+ \bpar{\frac{f(r)}{rf'(r)}}^2 \bpar{\frac{\norm{\nabla_{\sphere} g(u)}}{g(u)}}^2}}  \]
 \subsection{Proof of Proposition~\ref{prop:converse_link}}\label{app:prop:converse}
 For the converse relation, we define a weaker relaxation of strong convexity than $\sqc$, defined by the Restricted Secant Inequality (RSI) \cite{zhang2013gradient}.
\begin{definition}
   For $\nu>0$, we call $\rsi$ the set of functions $f$ that satisfy the Restricted Secant Inequality, namely $\forall x \in \R^d,~ \dotprod{\nabla f(x),x-x^\ast} \ge \nu \norm{x-x^\ast}^2.$
\end{definition}
A function in $\sqc$ is also in $\rsif{\frac{\tau\mu}{2-\tau}}$, such that the RSI condition is weaker \cite{hermant2024study}. 
\begin{proposition}
    Let $f\in\rsi$. If $f$ also satisfies
    \begin{equation}\label{eq:eb+}
         \forall x \in \R^d,~ \norm{\nabla f(x)}\le L\norm{x-x^\ast},
    \end{equation}

    then $f \in \acf{\frac{\nu}{L}}$.
\end{proposition}
\begin{proof}
    We have \[\dotprod{\nabla f(x),x-x^\ast} \ge \nu \norm{x-x^\ast}^2 \ge \frac{\nu}{L}\norm{x-x^\ast}\norm{\nabla f(x)}.\]
\end{proof}
The condition \eqref{eq:eb+} is sometimes named $\text{EB}^+$, as the upper version of the Error-Bound condition \cite{guille2021study}. It can also be interpreted as being Lipschitz with respect to the minimizers. The result follows from the previous proposition.

\end{document}